\documentclass[11pt]{article}

\usepackage{amsmath,amssymb,enumitem,amsthm,subfigure,stmaryrd}
\usepackage{graphics,pstricks,pst-node,color,xy}
\xyoption{all}

\numberwithin{equation}{section}
\newtheorem{thm}{Theorem}[section] 
\newtheorem{prp}[thm]{Proposition}
\newtheorem{lmm}[thm]{Lemma}   
\newtheorem{crl}[thm]{Corollary}

\newtheorem{dfn}[thm]{Definition}
\newtheorem{mythm}{Theorem}

\theoremstyle{definition}

\newtheorem{rmk}{Remark}

\def\eset{\emptyset} 

\def\BE#1{\begin{equation}\label{#1}}
\def\EE{\end{equation}}
\def\lan{\langle}
\def\ran{\rangle}
\def\lr#1{\lan#1\ran}
\def\blr#1{\big\lan#1\big\ran}
\def\bblr#1{\bigg\lan#1\bigg\ran}
\def\ov#1{\overline{#1}}
\def\ti#1{\tilde{#1}}
\def\wt#1{\widetilde{#1}}
\def\e_ref#1{(\ref{#1})}
\def\smsize#1{\begin{small}#1\end{small}}

\def\sf#1{\textsf{#1}}
\def\tn#1{\textnormal{#1}}
\def\Lau#1{\llceil{#1}\rrceil}
\def\coeff#1{\llbracket{#1}\rrbracket}
\def\bigcoeff#1{\big\llbracket{#1}\big\rrbracket}

\def\lra{\longrightarrow}

\def\al{\alpha}
\def\be{\beta}
\def\ga{\gamma}
\def\de{\delta}
\def\ep{\epsilon}
\def\io{\iota}

\def\la{\lambda}
\def\om{\omega}
\def\si{\sigma}
\def\th{\theta}

\def\ze{\zeta}

\def\Om{\Omega}

\def\De{\Delta}

\def\i{\infty}
\def\hb{\hbar}

\def\cA{\mathcal A}
\def\bD{\mathbf D}

\def\C{\mathbb C}
\def\cC{\mathcal C}
\def\ctC{\wt{\mathcal{C}}}
\def\nc{\mathrm{c}}
\def\ntc{\wt{\mathrm{c}}}
\def\fC{\mathfrak C}

\def\bfd{\mathbf d}
\def\D{\mathfrak D}
\def\cD{\mathcal D}
\def\E{\mathbf e}
\def\F{\mathcal F}

\def\nH{\textnormal{H}}
\def\I{\mathfrak i}

\def\cL{\mathcal L}
\def\cM{\mathcal M}
\def\M{\mathfrak M}

\def\O{\mathcal O}
\def\cO{\mathcal O}
\def\P{\mathbb P}
\def\cP{\mathcal P}

\def\Pn{\mathbb P^{n-1}}
\def\R{\mathbb R}
\def\Q{\mathbb Q}

\def\bfr{\mathbf r}

\def\cS{\mathcal S}
\def\T{\mathbb T}
\def\cT{\mathcal T}

\def\U{\mathfrak U}
\def\V{\mathcal V}
\def\cY{\mathcal Y}
\def\Z{\mathbb Z}
\def\cZ{\mathcal Z}
\def\a{\mathbf a}
\def\p{\mathbf p}
\def\cP{\mathcal{P}}
\def\bM{\mathbf{M}}
\def\nd{\textnormal{d}}
\def\nD{\textnormal{D}}
\def\ne{\textnormal{e}}

\def\ev{\textnormal{ev}}

\def\mod{\textnormal{mod~}}

\def\Rs#1{\underset{#1}{\mathfrak R}}

\begin{document}

\title{Mirror Symmetry for\\ Closed, Open, and Unoriented Gromov-Witten Invariants}
\author{Alexandra Popa
and Aleksey Zinger\thanks{Partially supported by DMS grant 0846978}}
\date{\today}
\maketitle

\begin{abstract}
In the first part of this paper, we obtain mirror formulas for twisted genus 0 
two-point Gromov-Witten (GW) invariants of projective spaces and for 
the genus~0 two-point GW-invariants of Fano and Calabi-Yau complete intersections.
This extends previous results for projective hypersurfaces, following
the same approach, but we also completely describe the structure coefficients
in both cases and obtain relations between these coefficients that are vital to
the applications to mirror symmetry in the rest of this paper.
In the second and third parts of this paper, we confirm Walcher's mirror symmetry conjectures
for the annulus and Klein bottle GW-invariants of Calabi-Yau complete intersection threefolds;
these applications are the main results of this paper.
In a separate paper, the genus~0 two-point formulas are used to obtain mirror formulas
for the genus~1 GW-invariants of all Calabi-Yau complete intersections.
\end{abstract}

\tableofcontents

\section{Introduction}
\label{intro_sec}

Gromov-Witten invariants of projective varieties are
counts of curves that are conjectured (and known in some cases)
to possess a rich structure.
In \cite{Gi96}, \cite{Gi99},  and~\cite{LLY},
the original mirror prediction of~\cite{CdGP} for the genus~0 GW-invariants of 
a quintic threefold is verified and shown to be a special case of 
mirror formulas satisfied by the genus~0 one-point GW-invariants 
of complete intersections.
In \cite{BeK}, \cite{bcov0}, \cite{C}, and~\cite{GhT},
these results are used to obtain mirror formulas for two-point genus~0 GW-invariants 
of projective hypersurfaces.
We begin this paper by extending the approach of~\cite{bcov0} to 
Fano and Calabi-Yau projective complete intersections, give a complete description 
of the (equivariant) structure coefficients in both cases, and obtain relations between them.
The mirror formulas and relations of Sections~\ref{equivmainthm_sec} and~\ref{pfs_sec} 
are used in Sections~\ref{annulus_sec} and~\ref{klein_sec} to confirm
the mirror symmetry predictions of Walcher~\cite{W1} concerning 
the annulus and Klein bottle invariants of Calabi-Yau complete intersection threefolds
in the presence of an anti-holomorphic involution; these are the main results of this paper.
Unlike proofs of mirror symmetry in other settings (such as 
in \cite{Po}, \cite{PSW},  and~\cite{bcov1}), our arguments do not rely
on a priori knowledge that the final answers are independent of the toric weights
(i.e.~are purely non-equivariant).

Throughout this paper, $\a=(a_1,a_2,\ldots,a_l)$ denotes a tuple of 
positive integers and $X_{\a}\subset\Pn$, 
a smooth complete intersection of multi-degree~$\a$; for example, $X_{\eset}=\Pn$.
Let
$$\lr{\a}\equiv\prod_{k=1}^la_k,\qquad \a!\equiv\prod_{k=1}^la_k!,
\qquad\a^{\a}\equiv\prod_{k=1}^la_k^{a_k}, \qquad  
|\a|\equiv\sum_{k=1}^l a_k, \qquad
\nu_{\a}\!\equiv n\!-\!|\a|.$$
We consider only the Fano cases, $\nu_{\a}\!>\!0$, and 
the Calabi-Yau cases, $\nu_{\a}\!=0$.
All cohomology groups in this paper are with rational coefficients unless specified otherwise.
Let $\nH\in\!H^2(\Pn)$ denote the hyperplane class. 

Mirror formulas relate GW-invariants of $X_{\a}$ to the hypergeometric series
\BE{tiFdfn_e}
F(w,q)\equiv\sum_{d=0}^{\i}
q^dw^{\nu_{\a}d}\frac{\prod\limits_{k=1}^l\prod\limits_{r=1}^{a_kd}(a_kw+r)}
{\prod\limits_{r=1}^d(w+r)^n}\,.\EE
This is a power series in $q$ with constant term 1 
whose coefficients are rational functions in $w$ which are regular at~$w=0$.
As in \cite{ZaZ}, we denote the subgroup of all such power series by~$\cP$
and define
\BE{DMDfn_e}
\begin{aligned}
&\bD\!:\Q(w)\big[\big[q\big]\big]\lra \Q(w)\big[\big[q\big]\big], 
&\quad& \bM:\cP\lra\cP  \qquad\hbox{by}\\
&\bD H(w,q)\equiv \left\{1+\frac{q}{w}\frac{\nd}{\nd q}\right\}H(w,q),
&\quad&
\bM H(w,q)\equiv\bD\bigg(\frac{H(w,q)}{H(0,q)}\bigg)\,.
\end{aligned}\EE
If $\nu_{\a}\!=0$ and $p\in\Z^{\ge0}$, let
\BE{Ipdfn_e} I_p(q)\equiv \bM^pF(0,q).\EE
For example,
$$I_0(q)=\sum_{d=0}^{\i}q^d\frac{(a_1d)!(a_2d)!\ldots(a_ld)!}{(d!)^n}
\quad\text{if}\quad\nu_{\a}\!=0.$$
If $\nu_{\a}=0$, let 
\BE{mirmap_e}
J(q)\equiv\frac{1}{I_0(q)}\left\{
\sum_{d=1}^{\i}q^d
\frac{\prod\limits_{k=1}^l(a_kd)!}{(d!)^n}
\left(\sum_{k=1}^{l}\sum_{r=d+1}^{a_kd}\frac{a_k}{r}\right)\right\}\quad
\hbox{and}\quad
Q\equiv q\,\ne^{J(q)}.\EE
Thus, the map $q\!\lra\!Q$ is a change of variables;
it will be called the \sf{mirror map}.

\subsection{Mirror formulas for closed GW-invariants}

In light of previous work on genus~0 two-point GW-invariants of projective complete intersections,
the precise statements of Theorems~\ref{nonequiv_thm} and~\ref{main_thm}
concerning these invariants are primarily stepping stones to the results on
open and unoriented GW-invariants of Theorems~\ref{annulus_thm} and~\ref{klein_thm}.
Nevertheless, in this section, we illustrate Theorems~\ref{nonequiv_thm} and~\ref{main_thm}
with explicit examples and describe relations with other work on formulas for 
genus~0 two-point GW-invariants.

Given a smooth subvariety $X\subset\Pn$,
we will write  $\ov\M_{0,m}(X,d)$ for 
the moduli space of stable degree~$d$ maps into $X$
from genus $0$ curves with $m$ marked points and
$$\ev_i:\ov\M_{0,m}(X,d)\lra X$$
for the evaluation map at the $i$-th marked point; see~\cite[Chapter~24]{MirSym}.
For each $i\!=\!1,2,\ldots,m$, let $\psi_i\in H^2(\ov\M_{0,m}(X,d))$
be the first Chern class of the universal cotangent line bundle for
the $i$-th marked point.
Gromov-Witten invariants are obtained by integration of classes 
against the virtual fundamental class of $\ov\M_{0,m}(X,d)$:
\BE{GWdfn_e}\begin{split}
\blr{\tau_{p_1}(\nH^{b_1}),\ldots,\tau_{p_m}(\nH^{b_m})}^X_d
&\equiv \blr{\psi^{p_1}\nH^{b_1},\ldots,\psi^{p_m}\nH^{b_m}}^X_d\\
&\equiv\int_{[\ov\M_{0,m}(X,d)]^{vir}}
\big(\psi_1^{p_1}\ev_1^*\nH^{b_1}\big)\ldots\big(\psi_m^{p_m}\ev_m^*\nH^{b_m}\big),
\end{split}\EE
where $H\!\in\!H^2(\Pn)$ is the hyperplane class.

The two-point mirror formulas take the simplest shape in the two extremal cases:
$\nu_{\a}\!=\!0,n$.

\begin{mythm}\label{proj_thm}
The degree $d\!\geq\!1$ genus~0 two-point descendant 
invariants of $\Pn$ with $n\!\ge\!2$ are given by the following identity in
$\left(\Q[\nH_1,\nH_2]/\{\nH_1^n,\nH_2^n\}\right)\big[\big[\hb_1^{-1},\hb_2^{-1}\big]\big]$:
\begin{equation*}\begin{split}
&\sum_{p_1,p_2\ge0}\! \bblr{\frac{\nH^{n-1-p_1}}{\hb_1\!-\!\psi},
\frac{\nH^{n-1-p_2}}{\hb_2\!-\!\psi}}^{\Pn}_d\!\!\nH_1^{p_1}\nH_2^{p_2}\\
&\hspace{1.5in}
=\frac{1}{\hb_1\!+\!\hb_2}
\sum_{\begin{subarray}{c}p_1+p_2=n-1\\ p_1,p_2\ge0\end{subarray}}
\sum_{\begin{subarray}{c}d_1+d_2=d\\ d_1,d_2\ge0\end{subarray}}
\frac{(\nH_1\!+\!d_1\hb_1)^{p_1}\,(\nH_2\!+\!d_2\hb_2)^{p_2}}
{\prod\limits_{r=1}^{d_1}(\nH_1\!+\!r\hb_1)^n\,
\prod\limits_{r=1}^{d_2}(\nH_2\!+\!r\hb_2)^n}\,.
\end{split}\end{equation*}
\end{mythm}

\begin{mythm}\label{descend_thm}
The genus~0 two-point descendant invariants of 
a Calabi-Yau complete intersection $X_{\a}$ in $\Pn$ are given by
the following identity in $\left(\Q[\nH_1,\nH_2]/\{\nH_1^{n-l},\nH_2^{n-l}\}\right)\big[\big[\hb_1^{-1},\hb_2^{-1},q\big]\big]$:
\begin{equation*}\begin{split}
&\sum_{p_1,p_2\ge0}\sum_{d=1}^{\i}Q^d
\bblr{\frac{\nH^{n-1-l-p_1}}{\hb_1-\psi},\frac{\nH^{n-1-l-p_2}}{\hb_2-\psi}}^{X_{\a}}_d
\nH_1^{p_1}\nH_2^{p_2}\\
&\qquad=\frac{\lr{\a}}{\hb_1+\hb_2}
\sum_{\begin{subarray}{c}p_1+p_2=n-1-l\\ p_1,p_2\ge0\end{subarray}}
\left(-1+\ne^{-J(q)\left(\frac{\nH_1}{\hb_1}+\frac{\nH_2}{\hb_2}\right)}
\frac{\bM^{p_1}F\left(\frac{\nH_1}{\hb_1},q\right)}{I_{p_1}(q)}
\frac{\bM^{p_2}F\left(\frac{\nH_2}{\hb_2},q\right)}{I_{p_2}(q)}\right)
\nH_1^{p_1}\nH_2^{p_2}\,,
\end{split}\end{equation*}
with $Q$ and $q$ related by the mirror map~\e_ref{mirmap_e}.
\end{mythm}

\begin{rmk}
In both theorems, the sums on the right-hand side of the identities
are power series in $\hb_1^{-1}$ and~$\hb_2^{-1}$
(to see this in Theorem~\ref{proj_thm}, divide both the numerator 
and denominator
of each $(p_1,p_2,d_1,d_2)$-summand by $\hb_1^{nd_1}\hb_2^{nd_2}$).
The two theorems state in particular that these sums are divisible by $\hb_1\!+\hb_2$.
This can be seen directly in the case of Theorem~\ref{proj_thm} as follows.
Divisibility by $\hb_1\!+\!\hb_2$ of a series in $\hb_1^{-1}$ and~$\hb_2^{-1}$
is equivalent to the vanishing of the series evaluated at $(\hb_1,\hb_2)=(\hb,-\hb)$.
The sum on the right-hand side of Theorem~\ref{proj_thm} evaluated 
at $(\hb_1,\hb_2)=(\hb,-\hb)$ and multiplied by $H_1\!-\!H_2\!+\!d\hb$~is
\begin{equation*}\begin{split}
&\sum_{\begin{subarray}{c}d_1+d_2=d\\ d_1,d_2\ge0\end{subarray}}
\frac{(\nH_1\!+\!d_1\hb)^n-(\nH_2\!-\!d_2\hb)^n}{\prod\limits_{r=1}^{d_1}(\nH_1\!+\!r\hb)^n
\prod\limits_{r=1}^{d_2}(\nH_2\!-\!r\hb)^n}
= \frac{1}{\prod\limits_{r=1}^{d-1}(\nH_1\!+\!r\hb)^n}
-\frac{1}{\prod\limits_{r=1}^{d-1}(\nH_2\!-\!r\hb)^n}\\
&\hspace{1in}
+\sum_{\begin{subarray}{c}d_1+d_2=d\\ d_1,d_2\ge1\end{subarray}}\left(\frac{1}{\prod\limits_{r=1}^{d_1-1}(\nH_1\!+\!r\hb)^n
\prod\limits_{r=1}^{d_2}(\nH_2\!-\!r\hb)^n}
-\frac{1}{\prod\limits_{r=1}^{d_1}(\nH_1\!+\!r\hb)^n
\prod\limits_{r=1}^{d_2-1}(\nH_2\!-\!r\hb)^n}
\right)=0;
\end{split}\end{equation*}
the first equality above uses $\nH_1^n,\nH_2^n=0$.
\end{rmk}

Theorem~\ref{descend_thm} leads to a simple identity for 
the primary GW-invariants of a Calabi-Yau complete intersection~$X_{\a}$.
Applying $Q\frac{\nd}{\nd Q}=\frac{q}{I_1(q)}\frac{\nd}{\nd q}$
to both sides of the identity in Theorem~\ref{descend_thm} and considering
the coefficient of 
$\nH_1^{n-1-l-b}\nH_2^{b+1}$, we~obtain
$$\frac{1}{\hb_1\hb_2}\sum_{d=1}^{\i}d\blr{\nH^{b},\nH^{n-2-l-b}}_d^{X_{\a}}\,Q^d
=\frac{\lr{\a}}{\hb_1+\hb_2}
\left(\frac{I_{n-1-l-b}(q)}{\hb_1\,I_1(q)}+
\frac{I_{b+1}(q)}{\hb_2\,I_1(q)}-\frac{1}{\hb_1}-\frac{1}{\hb_2}\right),$$
since  $1\!+\!q\frac{\nd J(q)}{\nd q}\!=\!I_1(q)$.
Multiplying both sides 
by $\hb_1\!+\!\hb_2$ 
and setting $\hb_2\!=-\hb_1$, we~obtain 
\begin{equation}\label{Irefl_e0}
I_{b+1}(q)=I_{n-1-l-b}(q)\quad\forall~
b=0,1,\ldots,n\!-\!l\!-\!2.
\end{equation}
The last two equations give
\begin{equation}\label{CY2GW_e}
\lr{\a}+\sum_{d=1}^{\i}d\blr{\nH^{b_1},\nH^{b_2}}^{X_{\a}}_d\,Q^d
=\lr{\a}\frac{I_{b_1+1}(q)}{I_1(q)}
\qquad\textnormal{if}~~b_1\!+b_2=n\!-\!l\!-\!2.
\end{equation}
Taking $(b_1,b_2)=(1,n\!-\!l\!-\!3)$ in \e_ref{CY2GW_e} and 
applying the divisor equation \cite[Section~26.3]{MirSym}, 
we obtain the following identity for one-point primary GW-invariants
of a Calabi-Yau complete intersection $X_{\a}$:
\begin{equation}\label{CY1GW_e}
\lr{\a}+\sum_{d=1}^{\i}d^2\blr{\nH^{n-l-3}}^{X_{\a}}_d\,Q^d
=\lr{\a}\frac{I_2(q)}{I_1(q)}.
\end{equation}

Tables~\ref{BPS0_table}-\ref{BPS4_table}  show low-degree 
two-point genus~$0$ BPS numbers, 
defined from the GW-invariants by equation~(2) in~\cite{KP}, 
for all multi-degree~$\a$ complete intersections $X_{\a}$
in~$\Pn$ with $n\!\le\!10$.
As predicted by Conjecture~0 in~\cite{KP}, these numbers are integers;
using a computer program,
we have confirmed this conjecture for all degree $d\!\le\!100$ two-point 
BPS counts in all Calabi-Yau complete intersections $X_{\a}$ 
in~$\Pn$ with $n\!\le\!10$. 
The degree~1 and~2 BPS numbers in these cases match 
the usual Schubert calculus computations on $G(2,n)$ and $G(3,n)$, respectively;
see~\cite{Ka}.
The degree~3 numbers for the hypersurfaces agree with~\cite{ES};
it should be possible to verify our degree~3 numbers 
for the other complete intersections in the tables 
by the approach of~\cite{ES} as well.

\begin{table}\centering{\renewcommand{\arraystretch}{1.15}
\scalebox{0.85}{
\begin{tabular}{c||c|c|c|c}
\hline
d&1&2&3&4\\
\hline
$X_7$&1707797&510787745643&222548537108926490&113635631482486991647224\\
$X_{26}$&616896&41762262528&4088395365564096&468639130901813987328\\
$X_{35}$&344925&10528769475&465037227025650&24049433312314947000\\
$X_{44}$&284672&6749724672&231518782306304&9297639201854554112\\
$X_{225}$&257600&4672315200&121622886740800&3703337959222528000\\
$X_{234}$&169344&1695326976&24368988329856&409711274829020160\\
$X_{333}$&134865&959370561&9805843550034&117225412143917130\\
$X_{2224}$&126976&755572736&6403783700480&63420292743217152\\
$X_{2233}$&101088&427633344&2578114145376&18160214808655872\\
\hline
\end{tabular}}}
\vspace{0.5mm}
\caption{Low-degree genus~$0$ BPS numbers $(\nH^2,\nH^2)$ for some Calabi-Yau~$5$-folds}
\label{BPS0_table}
\vspace{5.2mm}

\centering{\renewcommand{\arraystretch}{1.15}
\scalebox{0.85}{
\begin{tabular}{c||c|c|c|c}
\hline
d&1&2&3&4\\
\hline
$X_8$&37502976&224340704157696&2000750410187341381632&21122119007324663457380794368\\
$X_{27}$&12302724&14461287750168&25229820971457458076&52062878981745707203195872\\
$X_{36}$&5983632&2687545163520&1790676521197504848&1410987322122907728701952\\
$X_{45}$&4207200&1199825510400&507532701727557600&253883290498940295168000\\
$X_{226}$&4568832&1218545282304&480017733854171904&223463727594724776026112\\
$X_{235}$&2556900&308135971800&54819457086152700&11523817961861217228000\\
$X_{244}$&2113536&197815492608&27330245107728384&4461495054506601185280\\
$X_{334}$&1682208&112043367936&11011993317434016&1278661763157122064384\\
\hline
\end{tabular}}}
\vspace{0.5mm}
\caption{ Low-degree genus~$0$ BPS numbers $(\nH^2,\nH^3)$ for some Calabi-Yau~$6$-folds}
\label{BPS1_table}
\vspace{5.2mm}

\centering{\renewcommand{\arraystretch}{1.15}
\scalebox{0.85}{
\begin{tabular}{c||c|c|c}
\hline
d&2&3&4\\
\hline
$X_9$&93777295128674544&17873898563070361396216980&4116769336772585598746250465113376\\
$X_{28}$&4927955151077376&162926148665902467481600&6500105641339003383917401800704\\
$X_{37}$&705385191838824&7728929806910065428150&102149074253694894133257041184\\
$X_{46}$&232110378925056&1366213248304683678720&9698512727764286393809084416\\
$X_{55}$&161520243390000&777366857564506697500&4511987527454184551984500000\\
\hline
\end{tabular}}}
\vspace{0.5mm}
\caption{Low-degree genus~$0$ BPS numbers $(\nH^2,\nH^4)$ for some Calabi-Yau $7$-folds}
\label{BPS2_table}
\vspace{5.2mm}

\centering{\renewcommand{\arraystretch}{1.15}
\scalebox{0.85}{
\begin{tabular}{c||c|c|c}
\hline
d&2&3&4\\
\hline
$X_9$&156037426159482684&33815935806268253433549768&8638744084627099110538662706812804\\
$X_{28}$&7991674345455616&299081290134892802629632&13191988997947686388859151876096\\
$X_{37}$&1140060797165178&14119492055187150903348&206104052757048604579337400666\\
$X_{46}$&374346228782592&2489348580867704950272&19510528916120073780261924864\\
$X_{55}$&260419900772500&1415758838048143140000&9071479905327228206518687500\\
\hline
\end{tabular}}}
\vspace{0.5mm}
\caption{Low-degree genus~$0$ BPS numbers $(\nH^3,\nH^3)$ for some Calabi-Yau $7$-folds}
\label{BPS3_table}
\vspace{5.2mm}

\centering{\renewcommand{\arraystretch}{1.15}
\scalebox{0.82}{\setlength{\tabcolsep}{2pt}
\begin{tabular}{c||c|c|c}
\hline
d&2&3&4\\
\hline
$(\nH^2,\nH^5)$&40342298386119224000&174824389112955477418055016000&942582519217090098297647146585590400000\\
$(\nH^3,\nH^4)$&100290980400305376000&546627811934015785499223984000&3538531932815556807325167617597092800000\\
\hline
\end{tabular}}}
\vspace{0.5mm}
\caption{Low-degree genus~$0$ BPS numbers for $X_{10}\subset\P^9$}
\label{BPS4_table}
\end{table}

The genus~0 GW-invariants of the form~\e_ref{GWdfn_e} with $m\!=\!1$
are often assembled into a generating function, known 
as \textsf{the small $J$-function} or \textsf{Givental's $J$-function}.
Explicit closed formulas for the small $J$-function are obtained in
\cite{Gi96}, \cite{Gi99},  and~\cite{LLY} and used in many computations throughout GW-theory.
In light of \cite[Theorem~1]{LP}, the small $J$-function determines  all genus~0 
GW-invariants of projective complete intersections of the form~\e_ref{GWdfn_e}.
However, \cite[Theorem~1]{LP} is yet to be directly used to express a generating function
for the GW-invariants of the form~\e_ref{GWdfn_e}, even with $m\!=\!2$, in terms of 
the small $J$-function.
The relation \cite[Formula~1.1]{BeK} determines a generating function for the GW-invariants 
of the form~\e_ref{GWdfn_e} with $m\!=\!2$ in terms of the small $J$-function and
indirectly encodes the consequences of \cite[Theorem~1]{LP} relevant to the $m\!=\!2$ case.
Unfortunately, \cite[Formula~1.1]{BeK} is not a completely explicit relationship.
A different approach, more in the spirit of  \cite{Gi96} and~\cite{Gi99}, 
is used in~\cite{bcov0} to express  a generating function for the GW-invariants 
of the form~\e_ref{GWdfn_e} with $m\!=\!2$ as a linear combination of derivatives
in terms of the small $J$-function, with details sufficient for the computation 
of genus~1 invariants in~\cite{bcov1}.
The same approach is used in Section~\ref{pfs_sec} of this paper to obtain 
a more precise description of the structure coefficients in 
Theorems~\ref{nonequiv_thm} and~\ref{main_thm}, which is vital to the computations
of open and oriented invariants in Sections~\ref{annulus_sec} and~\ref{klein_sec}.
In~\cite{C}, the approach of~\cite{bcov0} is incorporated into
the Mirror Principle of~\cite{LLY}.

In the theory of Frobenius structures, all genus~0 GW-invariants
with descendants at only one marked point,
i.e.~as in~\e_ref{GWdfn_e} with arbitrary~$m$, but with 
pull-backs of arbitrary elements of~$H^*(X)$, not just powers of the hyperplane class,
and with  $p_i\!=\!0$ for all $i\!<\!m$,
are assembled into a generating function, called \textsf{the big $J$-function};
see \cite[Section~1]{Gi98}. 
According to Dubrovin's Reconstruction Formula,
the big $J$-function determines a generating function for all genus~0 GW-invariants
of a symplectic manifold~$X$, i.e.~as in~\e_ref{GWdfn_e}, but 
with pull-backs of arbitrary elements of~$H^*(X)$;
see \cite[(6.46),(6.48)]{Du}, \cite[Section~1]{Gi98}, \cite[Theorem~1]{Gi04}.
However, there is no simple closed formula for the big $J$-function,
even when restricted to the pull-back of $H^*(\Pn)$ in~$H^*(X)$.
Unfortunately, the big $J$-function is sometimes called the one-point $J$-function;
this has lead to some confusion in GW-theory as to whether 
\cite[Formula~1.1]{BeK} and Theorem~\ref{main_thm} in this paper are somehow 
contained in Dubrovin's Reconstruction Formula.
In light of the divisor and string relations \cite[Section~26.3]{MirSym},
the small $J$-function is essentially the restriction of 
the big $J$-function to $H^0(X)\!\oplus\!H^2(\Pn)|_X$;
the analogous statement holds for the small and big generating functions
with $\psi$-classes at two marked points.
While~\e_ref{main_e2} can be viewed as the restriction to $H^0(X)\!\oplus\!H^2(\Pn)|_X$
of a similar relationship for big generating functions,
this is not the case with~\e_ref{main_e1} as Dubrovin's Reconstruction Formula mixes
invariants with different numbers of marked points.
In fact, \cite{GhT} is essentially dedicated to showing that
Dubrovin's Reconstruction Formula restricts to the collection of GW-invariants of 
the form~\e_ref{GWdfn_e}, with $m\!=\!2$ and with pull-backs of powers of the hyperplane class,
and leads to some version of~\e_ref{main_e1};
the argument in~\cite{GhT} uses a version of the differential operators that 
appear in our Theorem~\ref{main_thm}.
In summary, \cite[Formula~1.1]{BeK} and Theorem~\ref{main_thm} in this paper are 
{\it not} contained in Dubrovin's Reconstruction Formula.

\subsection{Mirror formulas for open and unoriented GW-invariants}

Given a symplectic manifold $(X,\om)$ endowed with an anti-symplectic involution $\Om\!:X\!\lra\!X$,
it is natural to fix an $\om$-compatible almost complex structure~$J$ anti-commuting with~$\nd\Om$
and consider $J$-holomorphic maps $\ti{f}\!:\ti{C}\!\lra\!X$
from (possibly nodal) Riemann surfaces~$\ti{C}$ endowed with anti-holomorphic involutions~$\tau$ 
so that 
$$\ti{f}\circ\tau=\Om\circ\ti{f}\!: \ti{C}\lra X.$$
Such a triple $(\ti{C},\tau,\ti{f})$ will be called an \sf{$\Om$-invariant map to~$X$}.
The notion of isomorphic stable maps in GW-theory naturally extends to that of isomorphic 
stable $\Om$-invariant maps, by requiring compatibility with the anti-holomorphic involutions.
Such triples have three discrete parameters: the degree~$\be$ of~$\ti{f}$, 
the genus~$g$ of~$\ti{C}$, and the number~$h$ of fixed components of~$\tau$,
which roughly correspond to the boundary components of~$\ti{C}/\tau$.  
Each moduli space $\ov\M_{g,h,0}(X,\Om,\be)$ of stable $\Om$-invariant maps with fixed distinct
parameters is expected to carry a virtual fundamental class, giving rise to open
and unoriented
GW-invariants.\footnote{The last subscript in $\ov\M_{g,h,0}(X,\Om,\be)$ indicates no marked points.}

If $\ti{C}$ is a smooth Riemann surface of  genus~0 (i.e.~$\ti{C}\!=\!\P^1$), 
the only possible quotients $\ti{C}/\tau$
are $\R P^2$ ($h\!=\!0$) and the disk ($h\!=\!1$).
The former invariants should vanish according to \cite[Section~3.3]{W1}.
Disk invariants of symplectic 4- and 6-folds endowed with anti-symplectic involutions
are defined in \cite{So}.
Such manifolds include smooth Calabi-Yau complete intersection threefolds $X_{\a}\!\subset\!\Pn$
preserved by the involution
\BE{Omdfn_e}
\Om:\Pn\!\lra\!\Pn\,, \qquad
\Om([z_1,z_2,\ldots,z_n])\equiv
\begin{cases}
[\bar{z}_2,\bar{z}_1,\ldots,\bar{z}_{n},\bar{z}_{n-1}],&\hbox{if}~2|n;\\
[\bar{z}_2,\bar{z}_1,\ldots,\bar{z}_{n-1},\bar{z}_{n-2},\bar{z}_n],
&\hbox{if}~2\!\not|n.
\end{cases}\EE
The disk invariants are computed in~\cite{PSW} in the case $\a\!=\!(5)$
and for other CY CI threefolds $X_{\a}$ in~\cite{Sh}.
In these cases, the disk invariants are related to Euler classes of 
certain vector bundles over moduli spaces of stable $\Om$-invariant maps
to~$\Pn$.
The Localization Theorem~\cite{ABo} then reduces these invariants 
to sums over graphs.

Positive-genus analogues of the disk invariants of~\cite{So} have not yet been defined 
mathematically.
However, physical considerations of~\cite{W1} lead to explicit localization data
for such invariants of~$(X_{\a},\Om)$, whenever $X_{\a}$ is a Calabi-Yau 
complete intersection threefold.
As in the disk case, the localization data describes the invariants as sums of 
rational functions in several variables over graphs;
the resulting sums are in particular predicted to be weight-independent,
i.e.~not dependent on the variables involved.
If $\ti{C}$ is a smooth Riemann surface of  genus~1 (i.e.~$\ti{C}$ is a two-torus), 
the only possible quotients $\ti{C}/\tau$ are
the Klein bottle ($h\!=\!0$), the Mobius band ($h\!=\!1$), and the annulus ($h\!=\!2$).
According to \cite[Section~3.3]{W1}, the Mobius band invariants vanish.
In the other two genus~1 subcases, the invariants are predicted to be described
by explicit mirror formulas. 

After reviewing the equivariant setting of \cite{W1} in Section~\ref{OmEquiv_subs},
we recall the graph-sum description of the annulus invariants in Section~\ref{setup_subs}
and of the Klein bottle invariants in Section~\ref{kleindfn_subs}.
We then confirm the mirror symmetry predictions of~\cite{W1} for these invariants, 
{\it without} assuming weight independence;
see Theorems~\ref{annulus_thm} and~\ref{klein_thm} below.
This confirmation implies that the annulus and Klein bottle invariants of
Calabi-Yau complete intersections threefolds~$X_{\a}$ 
are well-defined (independent of the torus weights).
Since the power series $I_0(q)$, $I_1(q)$, $I_2(q)$, and $J(q)$ defined by \e_ref{Ipdfn_e}
and~\e_ref{mirmap_e} are the same for the tuples $(a_1,\ldots,a_l)$ and $(a_1,\ldots,a_l,1)$, 
these invariants for  $X_{(a_1,\ldots,a_l)}$ and~$X_{(a_1,\ldots,a_l,1)}$ are the same,
as expected.
Both types of invariants vanish for odd-degree maps; thus, both theorems concern
only even-degree invariants.\footnote{On the other hand, the disk invariants
in even degrees are expected to vanish according to \cite{W1};
see also \cite[Section~1.5]{PSW}.}

\begin{mythm}\label{annulus_thm} 
The degree $2d$ annulus invariants $A_{2d}$ as described in Definition~\ref{A2d_dfn}
are weight-independent and satisfy
\begin{equation}\label{ann_e}
Q\frac{\nd}{\nd Q}\Big[\sum_{d=1}^{\i}Q^dA_{2d}\Big]
=-\frac{1}{2\lr{\a}}\frac{I_1(q)}{I_2(q)}
\left[\Big\{Q\frac{\nd}{\nd Q}\Big\}^2Z_{disk}(Q)\right]^2\,,
\end{equation}
where $q$ and $Q$ are related by the mirror map \e_ref{mirmap_e}
and $Z_{disk}(Q)$ is the disk potential given by~\e_ref{Zdiskdfn_e}.
\end{mythm} 

This confirms the prediction of \cite[(5.21),(5.22)]{W1} with $f_t^{(0,2)}\!=\!0$.
In~\cite{W1}, $\ti{n}_d^{(0,2)}\!=\!A_d$,
the variables $(z,q\!=\!\ne^t)$ are our variables $(q,Q)$,
\BE{Zdiskdfn_e}
-\I\De(q)=Z_{disk}(Q)\equiv \sum_{d\in\Z^+\,\tn{odd}}\!\!\!Q^{\frac{d}{2}}N_d^{disk}\EE
is the disk potential (describing the disk invariants $N_d^{disk}$), and 
the power series~$C$ is the genus~0 generating function
whose third derivative with respect to~$t$ is given by~\e_ref{CY1GW_e}.
It is shown in \cite{Sh}, as well as in \cite{PSW} in the $\a\!=\!(5)$ case, that 
\BE{Fdisk_e}
 Z_{disk}(Q)=  \frac{2}{I_0(q)}
\sum\limits_{d\in\Z^+\,\tn{odd}}
q^{\frac{d}{2}}\frac{\prod\limits_{r=1}^l(a_rd)!!}{(d!!)^n}\EE
whenever all components of $\a$ are odd; otherwise, $Z_{disk}(Q)=0$.
This formula is obtained from equation~\e_ref{diskdfn_e}, which is
the disk analogue of the graph-sum definition~\e_ref{A2ddfn_e} of the annulus invariants.
The argument in~\cite{PSW} deducing \e_ref{Fdisk_e} from~\e_ref{diskdfn_e}
is rather delicate, limited to the case $\a\!=\!(5)$, and 
relies on the weight independence of the right-hand side
in~\e_ref{diskprp_e}, which is an a priori fact in the disk case.
The $(b,p)\!=\!(0,0)$ case of Lemma~\ref{res_lmm} (which is used in 
the proof of Theorem~\ref{annulus_thm}) gives a simple direct argument for this step
in~\cite{PSW}; this approach works for all tuples~$\a$ and does not 
presume weight independence.
A similar argument is used in~\cite{Sh}, based
on an independently discovered variation of  Lemma~\ref{res_lmm}
which is applicable in an overlapping set of cases.

\begin{mythm}\label{klein_thm} 
The degree $2d$ one-point Klein bottle invariants $\ti{K}_{2d}$ as described 
in Definition~\ref{K2d_dfn} are weight-independent and satisfy
\BE{klein_e}
\sum_{d=1}^{\i}Q^d\ti{K}_{2d}
=-Q\frac{\nd}{\nd Q}\ln\big((1\!-\!\a^{\a}q)^{1/4}I_1(q)\big)\,,\EE
where $q$ and $Q$ are related by the mirror map \e_ref{mirmap_e}.
\end{mythm}

By the divisor relation, this corresponds to the prediction of \cite[(5.26),(5.27)]{W1}.
In~\cite{W1}, the variables $(z,q)$ are our variables $(q,Q)$,
$$\ti{n}_d^{(1,0)_k}\!=\!\ti{K}_d/d, \qquad\hbox{and}\quad diss=1-\a^{\a}q.$$

\subsection{Outline of the paper}

Theorems~\ref{proj_thm} and \ref{descend_thm} are immediate consequences 
of Theorem~\ref{nonequiv_thm} in Section~\ref{mainthm_sec}.
Theorem~\ref{nonequiv_thm} follows immediately from Theorem~\ref{main_thm};
the latter is an equivariant version of the former and extends
\cite[Theorem~1.1]{bcov0} from line bundles to split bundles of arbitrary rank.
Theorem~\ref{nonequiv_thm} is
preceded by an explicit recursive formula
which facilitates the computation of genus~$0$ descendant invariants of the Fano
projective complete intersections as well. 
An example of such a computation is given
for the primary GW-invariants of $X_3\!\subset\!\P^4$.
The equivariant setting is introduced in Section~\ref{equivmainthm_sec},
where we state Theorem~\ref{main_thm}.

It is well-known that equivariant one-point GW-invariants of~$X_{\a}$
are expressed in terms of an equivariant version of 
the hypergeometric series~$F$; see~\e_ref{Z1ms_e}.
It is shown in~\cite{bcov0} that closed formulas for two-point genus~0 
GW-invariants of hypersurfaces are explicit transforms of
the one-point formulas; by Theorem~\ref{main_thm} in this paper, 
this is the case for all projective complete intersections.
The first part of  Section~\ref{pfs_sec} extends the proof of the analogous result
in \cite{bcov0} to the proof of Theorem~\ref{main_thm}.
It consists of showing that both sides satisfy certain good properties which guarantee uniqueness.
The proof that the two-point GW-invariants satisfy these properties 
is nearly identical to the analogous statement of \cite{bcov0}
which uses the Atiyah-Bott Localization Theorem~\cite{ABo}; 
details are given in Section~\ref{Z_subs}.
The main difference with~\cite{bcov0} occurs in constructing
the equivariant hypergeometric series with required properties;
see \e_ref{Ydfn_e} and~\e_ref{Ymldfn_e}.
Unlike \cite{bcov0}, we pay close attention to the Fano case as well,
giving explicit recursions for the structure coefficients of Theorem~\ref{main_thm}
and thus for the non-equivariant structure coefficients of 
Theorem~\ref{nonequiv_thm}.
In Section~\ref{mainthmapp_sec}, we obtain relations 
between the equivariant coefficients appearing in Theorem~\ref{main_thm};
these relations are used in the proofs of Theorems~\ref{annulus_thm} and~\ref{klein_thm}
in the rest of the paper.

In Section~\ref{annulus_sec}, we use the explicit equivariant recursions of
Theorem~\ref{main_thm} in the case $X_{\a}$ is a Calabi-Yau complete 
intersection threefold to study a recent prediction of Walcher \cite{W1}
concerning annulus GW-invariants. 
In particular, we use the Residue Theorem on~$S^2$ to show that the localization formulas
given in \cite{W1} are indeed weight-independent and 
sum up to the simple expression obtained in \cite{W1} based on physical considerations;
see Theorem~\ref{annulus_thm}. 
Along the way, we streamline one of the steps used in~\cite{PSW} 
to obtain a mirror formula for the disk invariants of the quintic threefold~$X_{(5)}$.

In Section~\ref{klein_sec}, we show that the natural one-point analogues of 
the localization formulas in \cite{W1} for Klein bottle invariants are 
indeed weight-independent and yield the closed formula predicted in~\cite{W1};
see Theorem~\ref{klein_thm}.
Unlike the annulus case, this case has a truly genus~one flavor;
the proof of Theorem~\ref{klein_thm} thus has little (if any) similarity to
the proof of Theorem~\ref{annulus_thm}.
We prove Theorem~\ref{klein_thm} by breaking the graphs at the special vertex of each loop
and using a special property of the generating functions for one-point genus~0 GW-invariants,
as in~\cite{bcov1} and~\cite{Po}. 
In contrast to~\cite{bcov1} and~\cite{Po}, we do not presuppose that the sums are weight-independent
and thus carry out the final step in the computation in a completely different way,
without using the Residue Theorem on~$S^2$ (it is still used in earlier steps).

Sections~\ref{annulus_sec} and~\ref{klein_sec} can thus be seen as 
the analogues for the annulus and Klein bottle invariants of
the localization computations confirming mirror symmetry predictions
for the closed genus~0 invariants (\cite{Gi96}, \cite{Gi99},  \cite{LLY}), 
the closed genus~1 invariants (\cite{bcov1}, \cite{Po}), 
and the disk invariants (\cite{PSW}).
However, the localization setup which serves as the starting point for these 
computations still requires a full mathematical justification for 
the annulus and Klein bottle invariants and cannot be presumed to be weight-independent
(the last property is used in \cite{bcov1}, \cite{Po}, and~\cite{PSW}).

We would like to thank V.~Shende and J.~Walcher for pointing out mistakes
in the description of the disk invariants in the original version of this paper
and Y.-P.~Lee for explaining results on GW-invariants
obtained in the literature on Frobenius structures.

\section{Main Theorem for Closed GW-Invariants}
\label{mainthm_sec}

Gromov-Witten invariants of a complete intersection $X_{\a}\subset\Pn$ 
can be computed from twisted GW-invariants of~$\Pn$. 
Let $\pi\!:\U\!\lra\!\ov\M_{0,m}(\Pn,d)$
be the universal curve and $\ev:\U\!\lra\!\Pn$
the natural evaluation map; see \cite[Section~24.3]{MirSym}.
Denote by 
$$\V_{\a}\lra\ov\M_{0,m}(\Pn,d)$$
the vector bundle corresponding to the locally free sheaf
$$\bigoplus_{k=1}^l\pi_*\ev^*\O_{\Pn}(a_k)\lra \ov\M_{0,m}(\Pn,d).$$
The Euler class $e(\V_{\a})$ relates genus~0 GW-invariants of 
$X_{\a}$ to genus~0 GW-invariants of~$\Pn$ by
\BE{hyp_e}
\int_{[\ov\M_{0,m}(X_{\a},d)]^{vir}}\eta
=\int_{[\ov\M_{0,m}(\Pn,d)]}\eta\, e(\V_{\a})
\qquad\forall\,\eta\in H^*\big(\ov\M_{0,m}(\Pn,d)\big);\EE
see \cite[Section 2.1.2]{BDPP}.

For each $i\!=\!1,\ldots,m$, there is a well-defined bundle map
$$\wt\ev_i\!:\V_{\a}\lra\ev_i^*\bigoplus\limits_{k=1}^{l}\O_{\Pn}(a_k), \qquad
\wt\ev_i\big([\cC,f;\xi]\big)=\big[\xi(x_i(\cC))\big],$$
where $x_i(\cC)$ is the $i$-th marked point of $\cC$.
Since it is surjective, its kernel is again a vector bundle.
Let
$$\V_{\a}'\equiv\ker\wt\ev_1\lra \ov\M_{0,m}(\Pn,d)
\qquad\hbox{and}\qquad
\V_{\a}''\equiv\ker\wt\ev_2\lra \ov\M_{0,m}(\Pn,d),$$
whenever $m\!\ge\!1$ and $m\!\ge\!2$, 
respectively.\footnote{In \cite[Part 4]{BDPP}, 
$\V_{\a}=W_{m,d}$, $\varepsilon_{m,d}=e(W_{m,d})$, 
and $\varepsilon'_{2,d}=e(W'_{2,d})$.
In \cite[Section 9]{Gi96}, $e(\V_{\a})=E_d$ and $e(\V_{\a}')=E_d'$.}
With $\ev_1,\ev_2:\ov\M_{0,2}(\Pn,d)\lra\Pn$ denoting
the evaluation maps at the two marked points, define
\begin{equation}\label{nonequivZ_e}
Z_p(\hb,Q)\equiv
\nH^{l+p}+\sum_{d=1}^{\i}\!Q^d\ev_{1*}\!
\left[\frac{e(\V_{\a}'')\ev_2^*\nH^{l+p}}{\hb\!-\!\psi_1}\right]
\in\big(H^*(\Pn)\big)[\hb^{-1}]\big[\big[Q\big]\big]
\end{equation}
for $p\in\Z$ with $p\!\ge\!-l$, and set
\begin{equation}\label{nonequivZ_e2}
Z^*(\hb_1,\hb_2,Q)\equiv
\sum_{d=1}^{\i}\!Q^d\left(\ev_1\!\times\!\ev_2\right)_*\!\!
\left[\frac{e(\V_{\a})}{(\hb_1\!-\!\psi_1)(\hb_2\!-\!\psi_2)}\right]\in
\big(H^*(\Pn\!\times\!\Pn)\big)\![\hb_1^{-1},\hb_2^{-1}]\big[\big[Q\big]\big].
\end{equation}
These power series determine all numbers~\e_ref{hyp_e} with 
$\eta=(\psi_1^{p_1}\ev_1^*\nH^{b_1})(\psi_2^{p_2}\ev_2^*\nH^{b_2})$.
The motivation behind the choice of indexing in~\e_ref{nonequivZ_e}
is analogous to the $l=1$ case in \cite[Section~1.1]{bcov0}. 

Define
\begin{alignat}{1}\label{Fdfn}
F_{-l}(w,q)&\equiv\sum_{d=0}^{\i}
q^dw^{d\nu_{\a}}\frac{\prod\limits_{k=1}^l\prod\limits_{r=0}^{a_kd-1}(a_kw+r)}
{\prod\limits_{r=1}^d(w+r)^n}\in\cP;\\
F_{-l+p}&\equiv\bD^pF_{-l}=\bM^pF_{-l}
\qquad\forall~p=1,2,\ldots,l.\notag
\end{alignat}
In particular, $F_0=F$.
For $\nu_{\a}\!>\!0$, we also define $\nc^{(d)}_{p,s},\ntc^{(d)}_{p,s}\in\!\Q$
with $p,d,s\ge0$ by 
\begin{alignat}{1}
\label{littlec_e}
\sum_{d=0}^{\i}\sum_{s=0}^{\i}\nc^{(d)}_{p,s}w^sq^d&=
\sum_{d=0}^{\i}q^d\frac{(w\!+\!d)^p\!\!\prod\limits_{k=1}^l\prod\limits_{r=1}^{a_kd}(a_kw+r)}
{\prod\limits_{r=1}^d(w+r)^n}
=w^p\bD^pF\big(w,q/w^{\nu_{\a}}\big),\\
\label{littletic_e}
\sum_{\begin{subarray}{c}d_1+d_2=d\\ d_1,d_2\ge0\end{subarray}}
\sum_{r=0}^{p-\nu_{\a}d_1}\ntc^{(d_1)}_{p,r}\nc^{(d_2)}_{r,s}&=\de_{d,0}\de_{p,s}
\qquad \forall\,d,s\!\in\!\Z^{\ge0},\, s\!\le\!p\!-\!\nu_{\a}d.
\end{alignat}
Since $\nc^{(0)}_{p,s}=\de_{p,s}$, \e_ref{littletic_e} expresses $\ntc^{(d)}_{p,s}$
in terms of the numbers $\ntc^{(d_1)}_{p,r}$
with $d_1\!<\!d$; the numbers $\ntc^{(d)}_{p,s}$ with $s\!>\!p\!-\!\nu_{\a}d$
will not be needed.
For example, 
\begin{equation}\label{cs1_e}
\ntc^{(0)}_{p,s}=\de_{p,s}\,, \qquad
\sum_{s=0}^{p-\nu_{\a}}\ntc^{(1)}_{p,s}w^s+
\lr\a\,\frac{\prod\limits_{k=1}^l\prod\limits_{r=1}^{a_k-1}(a_kw+r)}{(w+1)^{n-l-p}}
\in w^{p-\nu_{\a}+1}\Q[[w]].
\end{equation}
For $p\!\ge\!1$, set 
$$F_p(w,q)\equiv\begin{cases}
\bM^pF(w,q),&\hbox{if}~\nu_{\a}\!=\!0;\\
\sum\limits_{d=0}^{\i}
\sum\limits_{s=0}^{p-\nu_{\a}d}
\frac{\ntc^{(d)}_{p,s}\,q^d}{w^{p-\nu_{\a}d-s}}\bD^sF(w,q),&
\hbox{if}~\nu_{\a}\!>\!0.
\end{cases}$$
Thus, $F_p=\bD^pF$ unless $p\!\ge\!\nu_{\a}$.

\begin{mythm}\label{nonequiv_thm}
For every $l$-tuple of positive integers $\a=(a_1,a_2,\ldots,a_l)$,
$$Z^*(\hb_1,\hb_2,Q)= \frac{\lr{\a}}{\hb_1+\hb_2}
\sum_{\begin{subarray}{c} p_1+p_2=n-1-l\\ p_1,p_2\ge0 \end{subarray}}
\!\!\!\!\!\!\big(-\pi_1^*\nH^{l+p_1}\pi_2^*\nH^{l+p_2}+
\pi_1^*Z_{p_1}(\hb_1,Q)\pi_2^*Z_{p_2}(\hb_2,Q)\big),$$
where $\pi_1,\pi_2\!:\Pn\!\times\!\Pn\!\lra\Pn$ are the two projection maps.
For every $p\!\ge\!-l$,
\begin{equation*}\begin{split}
Z_p(\hb,Q)&=\nH^{l+p} \begin{cases}
\ne^{-J(q)\frac{\nH}{\hb}}\frac{F_p\left(\frac{\nH}{\hb},q\right)}{I_p(q)},
&\tn{if}~\nu_{\a}\!=\!0,\\
\ne^{-\a!q\frac{\nH}{\hb}}F_p\left(\frac{\nH}{\hb},q\right),
&\tn{if}~\nu_{\a}\!=1,\\
F_p\left(\frac{\nH}{\hb},q\right),
&\tn{if}~\nu_{\a}\!\ge\!2,\end{cases}
\qquad\tn{with}\qquad
Q=\begin{cases}
q\,\ne^{J(q)},& \tn{if}~\nu_{\a}\!=\!0,\\
\nH^{\nu_{\a}}q,& \tn{if}~\nu_{\a}\!\ge\!1,
\end{cases}\end{split}\end{equation*}
where $J$ is as in~\e_ref{mirmap_e}.
\end{mythm}

Since dropping a component of $\a$ equal to 1 has no effect on 
the power series $F$ in \e_ref{tiFdfn_e},
this also has no effect on the right-hand sides of 
the two formulas in Theorem~\ref{nonequiv_thm}, as expected.
If $\a\!=\!\eset$ and $n\!\ge\!2$,
$$\bD^pF(w,q)=1+\sum_{d=1}^{\i}
q^d\frac{(w\!+\!d)^p w^{nd-p}}{\prod\limits_{r=1}^d(w+r)^n}, \quad
Z_p(\hb,Q)=\nH^p\,\bD^pF\left(\frac{\nH}{\hb},\frac{Q}{\nH^n}\right) \qquad\forall\,p\le n\!-\!1,$$
giving Theorem~\ref{proj_thm}.
Theorem~\ref{descend_thm} is just the $\nu_{\a}\!=\!0$ case of 
Theorem~\ref{nonequiv_thm}.

\begin{rmk}
If $n\!-\!1\!-\!l\!<\!2\nu_{\a}$, only the coefficients $\ntc^{(1)}_{p,s}$
matter for the purposes of Theorem~\ref{nonequiv_thm};
these are given by~\e_ref{cs1_e}. 
For example, if $n=5$ and $\a=(3)$, then
$\ntc^{(1)}_{2,0}=\ntc^{(1)}_{3,0}=-6$ and $\ntc^{(1)}_{3,1}=-21$.
Thus, 
\begin{equation*}\begin{split}
F_p(w,q)&=\sum_{d=0}^{\i}q^dw^{2d-p}(w\!+\!d)^p
\frac{\prod\limits_{r=1}^{3d}(3w+r)}
{\prod\limits_{r=1}^d(w\!+\!r)^5} \qquad \hbox{if~}~p=0,1;\\
F_2(w,q)&=1+3q\sum_{d=0}^{\i}q^d
w^{2d}(w\!+\!d)
\frac{(7w\!+\!7d\!+\!5)\prod\limits_{r=1}^{3d}(3w+r)}
{(w\!+\!d\!+\!1)^2\prod\limits_{r=1}^d(w\!+\!r)^5} \,;\\
F_3(w,q)&=1+6q\sum_{d=0}^{\i}q^d
w^{2d-1}(w\!+\!d)^2
\frac{\prod\limits_{r=1}^{3d}(3w+r)}
{(w\!+\!d\!+\!1)\prod\limits_{r=1}^d(w\!+\!r)^5} \,.
\end{split}\end{equation*}
Thus, modulo $(\hb^{-1})^2$,
\begin{alignat*}{2}
F_0\left(\frac{\nH}{\hb},\frac{Q}{\nH^2}\right) &\cong 1, &\qquad
F_1\left(\frac{\nH}{\hb},\frac{Q}{\nH^2}\right) 
&\cong 1+6\frac{Q}{\nH}\hb^{-1},\\
F_2\left(\frac{\nH}{\hb},\frac{Q}{\nH^2}\right) 
&\cong 1+15\frac{Q}{\nH}\hb^{-1},&\qquad
F_3\left(\frac{\nH}{\hb},\frac{Q}{\nH^2}\right) 
&\cong 1+6\frac{Q}{\nH}\hb^{-1}+18\frac{Q^2}{\nH^3}\hb^{-1}\,.
\end{alignat*}
The two identities in Theorem~\ref{nonequiv_thm} give
\begin{equation}\label{X3p2_e}\begin{split}
&\sum_{d=1}^{\i}Q^d\sum_{p_1,p_2\ge0}\! \bblr{\frac{\nH^{3-p_1}}{\hb_1\!-\!\psi},
\frac{\nH^{3-p_2}}{\hb_2\!-\!\psi}}^{X_3}_d\!\!\nH_1^{p_1}\nH_2^{p_2}\\
&\hspace{1in}
=\frac{3}{\hb_1\!+\!\hb_2}
\sum_{\begin{subarray}{c}p_1+p_2=3\\ p_1,p_2\ge0\end{subarray}}
\left[-1+F_{p_1}\left(\frac{\nH_1}{\hb_1},\frac{Q}{\nH_1^2}\right)
F_{p_2}\left(\frac{\nH_2}{\hb_2},\frac{Q}{\nH_2^2}\right)\right]
\nH_1^{p_1}\nH_2^{p_2}\,,
\end{split}\end{equation}
modulo $\nH_1^4,\nH_2^4$.
Thus, considering the coefficient of $\hb_1^{-1}\hb_2^{-1}$ in \e_ref{X3p2_e}, 
we find that 
$$\blr{\nH^3}^{X_3}_1=\blr{\nH,\nH^3}^{X_3}_1=18,
\qquad \blr{\nH^2,\nH^2}^{X_3}_1=45,
\qquad \blr{\nH^3,\nH^3}^{X_3}_2=54.$$
This agrees with the usual Schubert calculus computations on $G(2,5)$ and $G(3,5)$.
\end{rmk}

In general, if $d\in\Z^{\ge0}$ and  $\nu_{\a} d\le p\!\le\!n\!-\!1\!-\!l$, then
\begin{equation}\label{ntcsym_e}
\sum_{\begin{subarray}{c}d_1+d_2=d\\ d_1,d_2\ge0\end{subarray}}
\ntc^{(d_1)}_{p-\nu_{\a}d_2,p-\nu_{\a}d}\ntc^{(d_2)}_{n-1-l-p+\nu_{\a}d_2,n-1-l-p}
=\begin{cases}
1,&\hbox{if}~d=0;\\
-\a^{\a},&\hbox{if}~d=1;\\
0,&\hbox{if}~d\ge 2.
\end{cases}\end{equation}
This identity is used to simplify four-pointed formulas in \cite{g0ci}; 
we verify it in the appendix.

\section{Equivariant Setting}
\label{equivmainthm_sec}

In Section~\ref{ClosedEquiv_subs} below, we review the equivariant setup used 
in \cite{bcov0}, closely following \cite[Section~3.1]{bcov0}.
After defining equivariant versions of the generating functions~$Z_p$ and~$Z^*$ 
and of the hypergeometric series~$F$ and~$F_{-l}$,
we state an equivariant version of Theorem~\ref{nonequiv_thm}; 
see Theorem~\ref{main_thm} below. 
We conclude Section~\ref{ClosedEquiv_subs} with explicit recursive formulas
for the equivariant structure coefficients that appear in Theorem~\ref{main_thm}.
In Section~\ref{OmEquiv_subs}, we describe the restricted equivariant setup used in~\cite{W1}.

The quotient of the classifying space for the $n$-torus $\T$ is 
$B\T\equiv(\P^{\i})^n$.
Thus, the group cohomology of~$\T$ is
$$H_{\T}^*\equiv H^*(B\T)=\Q[\al_1,\ldots,\al_n],$$
where $\al_i\!\equiv\!\pi_i^*c_1(\ga^*)$,
$\ga\!\lra\!\P^{\i}$ is the tautological line bundle,
and $\pi_i\!: (\P^{\i})^n\!\lra\!\P^{\i}$ is
the projection to the $i$-th component.
The field of fractions of $H^*_{\T}$ will be denoted by
$$\Q_{\al}\equiv \Q(\al_1,\ldots,\al_n).$$
We denote the equivariant $\Q$-cohomology of a topological space $M$
with a $\T$-action by $H_{\T}^*(M)$.
If the $\T$-action on $M$ lifts to an action on a complex vector bundle $V\!\lra\!M$,
let $\E(V)\in H_{\T}^*(M)$ denote the \sf{equivariant Euler class of} $V$.
A continuous $\T$-equivariant map $f\!:M\!\lra\!M'$ between two compact oriented 
manifolds induces a pushforward homomorphism
$$f_*\!: H_{\T}^*(M) \lra H_{\T}^*(M').$$

\subsection{Setting for closed GW-invariants}
\label{ClosedEquiv_subs}

The standard action of $\T$ on $\C^{n-1}$,
$$\big(\ne^{\I\th_1},\ldots,\ne^{\I\th_n}\big)\cdot (z_1,\ldots,z_n) 
\equiv\big(\ne^{\I\th_1}z_1,\ldots,\ne^{\I\th_n}z_n\big),$$
descends to a $\T$-action on $\Pn$.
The latter has $n$ fixed points,
\begin{equation}\label{fixedpt_e} 
P_1=[1,0,\ldots,0], \qquad P_2=[0,1,0,\ldots,0], 
\quad\ldots,\quad P_n=[0,\ldots,0,1].
\end{equation}
The curves preserved by this action are the lines through the fixed points,
$$\ell_{ij}\equiv
\big\{[z_1,z_2,\ldots,z_n]\!\in\!\Pn\!:~
z_k\!=\!0~~\forall\,k\!\notin\!\{i,j\}\big\}.$$
This standard $\T$-action on~$\Pn$ lifts to a natural 
$\T$-action on the tautological line bundle $\ga\!\lra\!\Pn$,
since $\ga\!\subset\!\Pn\!\times\!\C^n$ is preserved by the diagonal $\T$-action.
With 
$$x\equiv\E(\ga^*)\in H_{\T}^*(\Pn)$$
denoting the equivariant hyperplane class,
the equivariant cohomology of $\Pn$ is given~by
\begin{equation}\label{pncoh_e}
H_{\T}^*(\Pn)= \Q[x,\al_1,\ldots,\al_n]\big/(x\!-\!\al_1)\ldots(x\!-\!\al_n).
\end{equation}
Let $x_1,x_2\in H_{\T}^*(\Pn\!\times\!\Pn)$ be the pull-backs
of $x$ by the two projection maps.

The standard action of $\T$ on $\Pn$ induces $\T$-actions on 
$\ov\M_{0,m}(\Pn,d)$, $\U$, $\V_{\a}$, $\V_{\a}'$, and~$\V_{\a}''$;
see Sections~\ref{intro_sec} and \ref{mainthm_sec} for the notation.
Thus, $\V_{\a}$, $\V_{\a}'$, and~$\V_{\a}''$ have well-defined 
equivariant Euler classes
$$\E(\V_{\a}),\E(\V_{\a}'),\E(\V_{\a}'')\in H_{\T}^*\big(\ov\M_{0,m}(\Pn,d)\big).$$
The universal cotangent line bundle for the $i$-th marked point
also has a well-defined equivariant Euler class, which will still be
denoted by~$\psi_i$.
In analogy with \e_ref{nonequivZ_e} and~\e_ref{nonequivZ_e2}, we define 
\begin{alignat}{1}\label{equivZ_e}
\cZ_p(x,\hb,Q)&\equiv
x^{l+p}\!+\!\sum_{d=1}^{\i}\!Q^d\ev_{1*}\!
\left[\frac{\E(\V_{\a}'')\ev_2^*x^{l+p}}{\hb\!-\!\psi_1}\right]
\in\big(H_{\T}^*(\Pn)\big)\big[\big[\hb^{-1},Q\big]\big],\\
\label{equivZ_e2}
\cZ^*(x_1,x_2,\hb_1,\hb_2,Q)&\equiv
\sum_{d=1}^{\i}\!Q^d\left(\ev_1\!\times\!\ev_2\right)_*\!\!
\left[\frac{\E(\V_{\a})}{(\hb_1\!-\!\psi_1)(\hb_2\!-\!\psi_2)}\right]\\
&\hspace{2in}\notag
\in \big(H_{\T}^*(\Pn\!\times\!\Pn)\big)\!\big[\big[\hb_1^{-1},\hb_2^{-1},Q\big]\big],
\end{alignat}
where $p\in\Z$ with $p\!\ge\!-l$ and $\ev_1,\ev_2:\ov\M_{0,2}(\Pn,d)\lra\Pn$,
as before.
By~\e_ref{equivZ_e} and~\e_ref{pncoh_e},
\begin{equation}\label{symmrel_e}\begin{split}
&\sum_{r=0}^n(-1)^r\si_r\cZ_{p-r}(x,\hb,Q)\\
&\quad=x^{l+p-n}\sum_{r=0}^n(-1)^r\si_rx^{n-r}+
\sum_{d=1}^{\i}\!Q^d\ev_{1*}\!
\left[\frac{\E(\V_{\a}'')\ev_2^*\left(x^{l+p-n}\sum\limits_{r=0}^n(-1)^r\si_rx^{n-r} \right)}{\hb\!-\!\psi_1}\right]=0
\end{split}\end{equation}
if $n\!-\!l\!\le\!p\le\!n\!-\!1$,
where  $\si_r\!\in\!\Q_{\al}$ is the $r$-th elementary symmetric polynomial
in $\al_1,\ldots,\al_n$.

The equivariant versions of the hypergeometric series $F$ and $F_{-l}$ 
that we need are
\begin{alignat}{1}\label{Ydfn_e}
\cY(x,\hb,q)&\equiv\sum_{d=0}^{\i}q^d
\frac{\prod\limits_{k=1}^l\prod\limits_{r=1}^{a_kd}(a_kx+r\hb)}
{\prod\limits_{r=1}^{d}\prod\limits_{k=1}^{n}(x\!-\!\al_k\!+\!r\hb)}\,,\\
\label{Ymldfn_e}
\cY_{-l}(x,\hb,q)&\equiv \sum_{d=0}^{\i}q^d
\frac{\prod\limits_{k=1}^l\prod\limits_{r=0}^{a_kd-1}(a_kx+r\hb)}
{\prod\limits_{r=1}^{d}\prod\limits_{k=1}^{n}(x\!-\!\al_k\!+\!r\hb)}\,.
\end{alignat}
With $I_p(q)\!\equiv\!1$ if $\nu_{\a}\!\neq\!0$
(and given by \e_ref{Ipdfn_e} if $\nu_{\a}\!=\!0$), 
for $p\!\ge\!0$ define $\D^p\cY_0(x,\hb,q)$ inductively~by
\BE{Yp_dfn_e2}\begin{split}
\D^0\cY_0(x,\hb,q)&\equiv \cY_0(x,\hb,q)\equiv
\frac{x^l}{I_0(q)}\cY(x,\hb,q),\\
\D^p\cY_0(x,\hb,q)&\equiv
\frac{1}{I_p(q)}
\left\{x+\hb\, q\frac{\nd}{\nd q}\right\}\D^{p-1}\cY_0(x,\hb,q)\,\qquad\forall~ p\ge1.
\end{split}\EE
For $p=0,1,\ldots,l\!-\!1$, let
\begin{equation}\label{Yp_dfn_e}\begin{split}
\cY_{-l+p}(x,\hb,q)&\equiv \D^{-l+p}\cY_0(x,\hb,q)
\equiv \left\{x+\hb\, q\frac{\nd}{\nd q}\right\}^p\cY_{-l}(x,\hb,q)\\
&=\sum_{d=0}^{\i}q^d(x\!+\!d\hb)^p
\frac{\prod\limits_{k=1}^l\prod\limits_{r=0}^{a_kd-1}(a_kx+r\hb)}
{\prod\limits_{r=1}^{d}\prod\limits_{k=1}^{n}(x\!-\!\al_k\!+\!r\hb)}\,.
\end{split}\end{equation}
Thus, $\cY_0(x,\hb,q)=
\frac{1}{I_0(q)}\left\{x+\hb\, q\frac{\nd}{\nd q}\right\}\D^{-1}\cY_0(x,\hb,q)$.

\begin{mythm}\label{main_thm}
For every $l$-tuple of positive integers $\a=(a_1,a_2,\ldots,a_l)$,
\begin{gather}
\label{main_e2}\begin{split}
\cZ^*(x_1,x_2,\hb_1,\hb_2,Q)&=\\
 \frac{\lr{\a}}{\hb_1+\hb_2}&
\sum_{\begin{subarray}{c}p_1+p_2+r=n-1\\ p_1,p_2\ge0 \end{subarray}}
\!\!\!\!\!\!\!\!\!(-1)^r\si_r\big(-x_1^{l+p_1}x_2^{p_2}+
\cZ_{p_1}(x_1,\hb_1,Q)\cZ_{p_2-l}(x_2,\hb_2,Q)\big),
\end{split}\end{gather}
$\si_r\!\in\!\Q_{\al}$ is the $r$-th elementary symmetric polynomial
in $\al_1,\ldots,\al_n$.
Furthermore, there exist $\ctC_{p,s}^{(r)}(q)\in q\cdot\Q[\al_1,\ldots,\al_n][[q]]$ 
such that  
the coefficient of $q^d$ in $\ctC_{p,s}^{(r)}(q)$ is a degree~$r\!-\!\nu_{\a}d$ 
homogeneous symmetric polynomial in $\al_1,\al_2,\ldots,\al_n$ and
\BE{main_e1}
\cZ_p(x,\hb,Q)= \begin{cases}
\cY_p(x,\hb,Q),&\tn{if}~\nu_{\a}\!\geq2,\\
\ne^{-\frac{Q}{\hb}\a!}\cY_p(x,\hb,Q), &\tn{if}~\nu_{\a}\!=1,\\
\ne^{-J(q)\frac{x}{\hb}}\ne^{-C_1(q)\frac{\si_1}{\hb}}\cY_p(x,\hb,q),
 &\tn{if}~\nu_{\a}\!=0,
\end{cases}\EE
where
\BE{main_e0}
\cY_p(x,\hb,q)\equiv\D^p\cY_0(x,\hb,q)
+\sum_{r=1}^{p}\sum_{s=0}^{p-r}\ctC_{p,s}^{(r)}(q)\hb^{p-r-s}\D^s\cY_0(x,\hb,q),\EE
with $Q$ and $q$ related by the mirror map~\e_ref{mirmap_e}
and $C_1(q)\in q\cdot\Q[[q]]$ given by~\e_ref{C1_e}.
\end{mythm} 

Theorem~\ref{main_thm} implies Theorem~\ref{nonequiv_thm} as follows. 
Setting $\al_1,\ldots,\al_n\!=\!0$ in \e_ref{main_e2} and using
that $Z_p(\hb,Q)=0$ if $p\!\ge\!n\!-\!l$ by \e_ref{nonequivZ_e}, we obtain the first identity in Theorem~\ref{nonequiv_thm}.
By \e_ref{tiFdfn_e}, \e_ref{DMDfn_e}, \e_ref{Ydfn_e}, and \e_ref{Yp_dfn_e2}, 
\begin{equation}\label{YvsF_e}
\D^p\cY_0(x,\hb,q)\big|_{\al=0}=x^{l+p}\cdot
\begin{cases}
\frac{1}{I_p(q)}\bM^pF(x/\hb,q),&\hbox{if}~\nu_{\a}\!=\!0;\\
\bD^pF(x/\hb,q/x^{\nu_a}),&\hbox{if}~\nu_{\a}\!>\!0.
\end{cases}\end{equation}
Thus, if $\nu_{\a}\!=\!0$, setting $\al_1,\ldots,\al_n\!=\!0$ 
in Theorem~\ref{main_thm} gives the corresponding case of the second identity
in Theorem~\ref{nonequiv_thm}.

We now completely describe the power series $\ctC_{p,s}^{(r)}$ of Theorem~\ref{main_thm};
it will be shown in Section~\ref{pfs_sec} below that they indeed satisfy~\e_ref{main_e1}
and~\e_ref{main_e0}.
For $p,r,s\!\ge\!0$, define $\cC_{p,s}^{(r)},\ctC_{p,s}^{(r)}\in\Q[\al_1,\ldots,\al_n][[q]]$
by
\begin{alignat}{1}\label{Crec_e}
x^l\hb^p\sum_{s=0}^{\i}\sum_{r=0}^s\cC_{p,s}^{(r)}(q)x^{s-r}\hb^{-s}
&=\D^p\cY_0(x,\hb,q) \,,\\
\label{tiCrec_e}
\sum_{\begin{subarray}{c}r_1+r_2=r\\ r_1,r_2\ge0\end{subarray}}\sum_{t=0}^{p-r_1}
\ctC_{p,t}^{(r_1)}(q)
\cC^{(r_2)}_{t,s-r_1}(q)&=\de_{r,0}\de_{p,s}
\quad\forall~r,s\!\in\!\Z^{\ge0},\,r\!\le\!s\!\le\!p.
\end{alignat}
By \e_ref{Ydfn_e}, \e_ref{Yp_dfn_e2}, and \e_ref{Crec_e}, 
the coefficient of $q^d$ in $\cC^{(r)}_{p,s}$ is a degree~$r\!-\!\nu_{\a}d$ 
homogeneous symmetric polynomial in~$\al$.
By~\e_ref{YvsF_e} and  $\D^p\cY_0\in x^{l+p}+q\cdot\Q_{\al}(x,\hb)[[q]]$, 
$$\cC_{p,p}^{(0)}(q)=1, \qquad \cC_{p,s}^{(0)}(q)=0~~\forall\,p>s, \qquad 
\cC_{p,s}^{(r)}(q)\in\de_{r,0}\de_{p,s}+q\cdot\Q[\al_1,\ldots,\al_n]\big[\big[q\big]\big].$$
Thus, \e_ref{tiCrec_e} expresses $\ctC^{(r)}_{p,s-r}$
with $r\!\le\!s\!\le\!p$ in terms of the series $\ctC^{(r_1)}_{p,t}$
with $r_1\!<\!r$ or $r_1\!=\!r$ and $t\!<\!s\!-\!r$; 
the series $\ctC^{(r)}_{p,s}$ with $r\!+\!s\!>\!p$ are not needed.
In particular, $\ctC^{(0)}_{p,s}=\de_{p,s}$
and the coefficient of $q^d$ in $\ctC^{(r)}_{p,s}$ is a degree~$r\!-\!\nu_{\a}d$ 
homogeneous symmetric polynomial in~$\al$.
If $\nu_{\a}\!>\!0$,
$$\cC^{(\nu_{\a}d)}_{p,s}\big|_{\al=0}=\nc^{(d)}_{p,s-\nu_{\a}d}q^d
\qquad\forall\,s\!\ge\!\nu_{\a}d$$
by \e_ref{Crec_e}, \e_ref{YvsF_e},  and \e_ref{littlec_e}. 
Thus, setting $\al\!=\!0$ in~\e_ref{tiCrec_e} and comparing with~\e_ref{littletic_e}
with~$s$ replaced by $s\!-\!\nu_{\a}d$, we conclude that 
\begin{equation}\label{cCred_e}
\ctC_{p,s}^{(d\nu_{\a})}(q)|_{\al=0}=\ntc_{p,s}^{(d)}q^d
\qquad\hbox{if}\quad \nu_{\a}\neq0.
\end{equation}
Setting $\al\!=\!0$ in \e_ref{main_e1} and~\e_ref{main_e0}
and using \e_ref{YvsF_e} and~\e_ref{cCred_e}, 
we obtain the $\nu_{\a}\!\neq\!0$ case of the second identity
in Theorem~\ref{nonequiv_thm}.

In the Calabi-Yau case, $|\a|\!=\!n$, 
it is a consequence of \e_ref{symmrel_e} and \e_ref{main_e1} that the structure coefficients 
in~\e_ref{main_e0} satisfy
\BE{ctCperrel_e}
\sum_{r=0}^p(-1)^r\si_r\ctC_{n-l-r,n-l-p}^{(p-r)}(q)=(-1)^p\si_p\prod_{s=0}^{n-l-p}\!\!\!I_s(q)
\qquad\forall\,p\!=\!1,2,\ldots,n\!-\!l\,.
\EE
This relation, a special case of  which is used in the Klein bottle invariants computation 
in Section~\ref{kleinpf_subs}, is proved in Section~\ref{ctC_subs}.
On the other hand, the only property of the structure coefficients $\ctC_{p,s}^{(r)}$
needed to compute the closed genus~1 GW-invariants of
Calabi-Yau hypersurfaces in~\cite{bcov1}
is that they lie in the ideal generated 
by $\si_1,\ldots,\si_{n-1}$.
It is shown in \cite{Po} that the same is 
the case for all Calabi-Yau complete intersections.

\subsection{Setting for open and unoriented GW-invariants}
\label{OmEquiv_subs}

We now recall the equivariant setting used in the graph-sum definition of open
and unoriented invariants in~\cite{W1}.
Throughout this section, $\Om$ denotes the anti-holomorphic involution
of~$\Pn$ described by~\e_ref{Omdfn_e} and its natural lifts to 
the tautological line bundle $\cO_{\Pn}(-1)$ and to the vector bundle
\begin{equation}\label{cLdfn_e}
\cL\equiv\bigoplus_{r=1}^l\cO_{\Pn}(a_k)\lra\Pn\,.
\end{equation}
Define
$$\bar{\cdot}:\{1,\ldots,n\}\lra\{1,\ldots,n\}
\qquad\hbox{by}\qquad \Om(P_i)=P_{\bar{i}}\,,$$
where $P_1,\ldots,P_n\!\in\!\Pn$ are the $\T^n$-fixed points; see~\e_ref{fixedpt_e}.
Denote by $m$ the integer part of $n/2$ and
by $\la_1,\ldots,\la_m$ the weights of the standard representation
of~$\T^m$ on~$\C^m$.
The embedding 
$$\io\!:\T^m\lra\T^n, \qquad 
(u_1,u_2,\ldots,u_m)\lra 
\begin{cases} (u_1,u_1^{-1},\ldots,u_m,u_m^{-1}),&\hbox{if}~n\!=\!2m;\\
(u_1,u_1^{-1},\ldots,u_m,u_m^{-1},1),&\hbox{if}~n\!=\!2m\!+\!1;
\end{cases}$$
induces $\T^m$-actions on $\Pn$ and moduli spaces of stable $\Om$-invariant maps to~$\Pn$.
Note that 
\begin{equation}\label{sp_weights_e}
(\al_1,\ldots,\al_n)\big|_{\T^m}
=\begin{cases}
(\la_1,-\la_1,\ldots,\la_m,-\la_m),&\hbox{if}~n\!=\!2m;\\
(\la_1,-\la_1,\ldots,\la_m,-\la_m,0),&\hbox{if}~n\!=\!2m\!+\!1.
\end{cases}\end{equation}

The following lemma describes the $\T^m$-fixed and $\Om$-fixed
zero- and one-dimensional subspaces of~$\Pn$.
Recall that a subspace $Y\!\subset\!\Pn$ is called $\Om$-fixed
($\T^m$-fixed) if $\Om(Y)\!=\!Y$ ($u(Y)\!=\!Y$ for all $u\!\in\!\T^m$).
If $n\!=\!2m\!+\!1$, $1\!\le\!i\!\le\!m$, and $a,b\!\in\!\C^*$, then
$$\cT_i(a,b)\equiv\big\{[z_1,z_2,\ldots,z_n]\in\Pn\!:~
z_k\!=\!0~~\forall\,k\!\notin\!\{2i\!-\!1,2i,n\},~
az_{2i-1}z_{2i}\!+\!bz_n^2\!=\!0\big\}$$
is a smooth conic contained in the plane spanned by $P_{2i-1},P_{2i},P_n$
and passing through~$P_{2i-1}$ and~$P_{2i}$.

\begin{lmm}\label{fixed_curves}
\begin{enumerate}[label=(\arabic{*}),ref=\arabic{*}]
\item The $\T^m$-fixed points in $\Pn$ are $P_1,P_2,\ldots,P_n$.
\item\label{fixed_curves_2} If $n\!=\!2m$, the $\T^m$-fixed irreducible curves in $\Pn$ 
are the lines $\ell_{ij}$ with $1\!\le\!i\!\neq\!j\!\le\!n$.
If $n\!=\!2m\!+\!1$, the $\T^m$-fixed irreducible curves in $\Pn$ 
are the lines $\ell_{ij}$ with $1\!\le\!i\!\neq\!j\!\le\!n$
and the conics $\cT_i(a,b)$ with $1\!\le\!i\!\le\!m$ and $a,b\!\in\!\C^*$.
\item\label{Om-invar}
If $n\!=\!2m$, the $\Om$-fixed $\T^m$-fixed irreducible curves in $\Pn$
are the lines $\ell_{2i-1,2i}$ with $1\!\le\!i\!\le\!m$.
If $n\!=\!2m\!+\!1$, the $\Om$-fixed $\T^m$-fixed irreducible curves 
in $\Pn$ are the lines $\ell_{2i-1,2i}$ with $1\!\le\!i\!\le\!m$
and the conics $\cT_i(a,b)$ with $1\!\le\!i\!\le\!m$ and 
$a\bar{b}=\bar{a}b$.
\end{enumerate}
\end{lmm}

\begin{proof} If a point $p\!\in\!\Pn$ has two distinct nonzero coordinates,
$z_i$ and $z_j$, then the function $z_i/z_j$ is not constant
on the orbit of~$p$; this implies the first claim.
If a point $p\!\in\!\Pn$ has three distinct nonzero coordinates,
$z_{2i-\ep_i}$, $z_{2j-\ep_j}$, and $z_k$ with $i\!\neq\!j\!\le\!m$ 
and $\ep_i,\ep_j\!=\!0,1$, then the image of the orbit of~$p$ under 
the function $(z_k/z_{2i-\ep_i},z_k/z_{2j-\ep_j})$
is the complement of the coordinate axes in~$\C^2$.
Thus, if $C$ is a $\T^m$-fixed irreducible curve which is not one of the lines
$\ell_{ij}$, $C$  lies in the plane spanned by $P_{2i-1}$, $P_{2i}$,
and $P_n$ for some $i\!\le\!m$ and $n\!=\!2m\!+\!1$.
If $p\!\in\!C$ has three nonzero coordinates,
its orbit is then one-dimensional and is contained in some conic $\cT_i(a,b)$;
since $C$ and $\cT_i(a,b)$ are irreducible, it follows that $C\!=\!\cT_i(a,b)$,
confirming the second claim.
The third claim follows immediately from the second, since 
$\Om$ sends the coordinate functions $z_{2i-1}$ and~$z_{2i}$
to~$\bar{z}_{2i}$ and~$\bar{z}_{2i-1}$, respectively, and 
$\cT_i(a,b)\!=\!\cT_i(\bar{a},\bar{b})$ if and only if $a\bar{b}\!=\!\bar{a}b$.
\end{proof}

If $\ti{C}$ is a genus~1 Riemann surface, let $\ti{C}_0\!\subset\!\ti{C}$
be the \sf{principal}, genus-carrying, component or union of components;
this is either a smooth torus or a circle of spheres.
Any anti-holomorphic involution $\tau\!:\ti{C}\!\lra\!\ti{C}$ restricts
to an anti-holomorphic involution on~$\ti{C}_0$;
if in addition, $\ti{C}_0$ is a smooth, then 
$\ti{C}_0/\tau$ is topologically either a Klein bottle, a Mobius band, or an annulus.

\begin{crl}\label{fixedmap_crl}
Let $\ti{C}$ be a genus~1 Riemann surface
and $\ti{f}\!:(\ti{C},\tau)\!\lra\!(\Pn,\Om)$ an $\Om$-invariant map
representing a $\T^m$-fixed equivalence class.
If the principal component~$\ti{C}_0$ of~$\ti{C}$ is a smooth torus, then $\ti{f}|_{\ti{C}_0}$
is constant.
If $\ti{f}|_{\ti{C}_0}$ is constant or the fixed locus of~$\tau$ is not smooth,
then $n\!=\!2m\!+\!1$ and $\ti{f}(p)\!=\!P_n$ for some node $p\!\in\!\ti{C}_0$.  
\end{crl}

\begin{proof} If $\ti{C}_0$ is a smooth torus and $\ti{f}|_{\ti{C}_0}$ is 
not constant, $\ti{f}(\ti{C}_0)$ is a 
$\T^m$-fixed irreducible curve in $\Pn$ and
the branch locus $B\!\subset\!\ti{f}(\ti{C}_0)$ of $\ti{f}|_{\ti{C}_0}$
is $\T^m$-fixed as well.
However, this is impossible, since $B$ contains at least four points
(by~\e_ref{fixed_curves_2} in Lemma~\ref{fixed_curves} and Riemann-Hurwitz),
while a zero-dimensional $\T^m$-fixed subset of 
$\ti{f}(\ti{C}_0)$ contains at most two points (also by~\e_ref{fixed_curves_2} in Lemma~\ref{fixed_curves}).
This proves the first claim.
The second claim is clear, since $\ti{f}(\ti{C}_0)$ and 
$\ti{f}(p)$ for every $p\!\in\!\ti{C}$ fixed by~$\tau$ is 
an $\Om$-fixed $\T^m$-fixed subset of~$\Pn$.
\end{proof}


\section{Proof of Theorem~\ref{main_thm}}
\label{pfs_sec}

After minor algebraic modifications,
the proofs of \e_ref{symmrel_e} and \e_ref{main_e2} are nearly identical 
to the proofs of \cite[(1.6)]{bcov0} and \cite[(1.17)]{bcov0}, 
which are the $l\!=\!1$ cases of~\e_ref{symmrel_e}  and~\e_ref{main_e2},
respectively.
The verification of~\e_ref{main_e1} is similar to the verification
of \cite[(1.16)]{bcov0}, which is the $\a\!=\!(n)$ case of~\e_ref{main_e1},
once suitable functions $\cY_p$ are constructed.

For any ring $R$, let
$$R\Lau{\hb}\equiv R[[\hb^{-1}]]+R[\hb]$$
denote the $R$-algebra of Laurent series in $\hb^{-1}$ (with finite principal part).
If
$$Z(\hb,Q)=\sum_{d=0}^{\i}
\bigg(\sum_{r=-N_d}^{\i}\!\!\!Z_d^{(r)}\hb^{-r}\bigg)Q^d
\in R\Lau{\hb}\big[\big[Q\big]\big]$$
for some $Z_d^{(r)}\!\in\!R$, we define
$$Z(\hb,Q) \cong \sum_{d=0}^{\i}
\bigg(\sum_{r=-N_d}^{p-1}\!\!\!Z_d^{(r)}\hb^{-r}\bigg)Q^d\quad(\mod \hb^{-p}),$$ 
i.e.~we drop $\hb^{-p}$ and higher powers of $\hb^{-1}$, 
instead of higher powers of~$\hb$.
If in addition $R\!\supset\!\Q$, let
\begin{equation*}\begin{split}
R_{\al}&\equiv \Q_{\al}\otimes_{\Q}\!R
=R[\al_1,\ldots,\al_n]_{\lr{u|u\in\Q[\al_1,\ldots,\al_n]-0}}\,, \\
H_{\T}^*(\Pn;R)&\equiv H_{\T}^*(\Pn)\otimes_{\Q}\!R
=R[\al_1,\ldots,\al_n,x]\big/\prod_{k=1}^n(x\!-\!\al_k);
\end{split}\end{equation*}
the latter ring is the $\T$-equivariant cohomology of $\Pn$ with coefficients in~$R$.
If $R$ is a field, let $R^*\!=\!R\!-\!0$ be the set of invertible elements and 
$$R(\hb) \lhook\joinrel\lra R\Lau{\hb}$$
be the embedding given by taking the Laurent series of rational functions at $\hb^{-1}\!=\!0$.

\subsection{Summary}
\label{pfsumm_subs}

With $\ev_1:\ov\M_{0,2}(\Pn,d)\lra\Pn$ as before, let 
\begin{alignat}{1} \label{Z_e}
\cZ(x,\hb,Q)& \equiv
1+\sum_{d=1}^{\i}\!Q^d\ev_{1*}\left[\frac{\E(\V_{\a}')}{\hb\!-\!\psi_1}\right]
\in \left(H^*_{\T}(\Pn)\right)\big[\big[\hb^{-1},Q\big]\big],\\
\label{C1_e}
C_1(q)&\equiv\frac{1}{I_0(q)}\sum_{d=1}^{\i}q^d
\frac{\prod\limits_{k=1}^l(a_kd)!}{(d!)^n}\left(\sum_{r=1}^d\frac{1}{r}\right).
\end{alignat}
By \cite[Theorems 9.5, 10.7, 11.8]{Gi96} and \e_ref{Ydfn_e}, 
\BE{Z1ms_e}
\cZ(x,\hb,Q)= \begin{cases}
\cY(x,\hb,Q),&\textnormal{if}~\nu_{\a}\!\ge2,\\
\ne^{-\frac{Q}{\hb}\a!}\cY(x,\hb,Q),&\textnormal{if}~\nu_{\a}\!=1,\\
\ne^{-J(q)\frac{x}{\hb}}\ne^{-C_1(q)\frac{\si_1}{\hb}}\cY(x,\hb,q)/I_0(q),
&\textnormal{if}~\nu_{\a}\!=0,
\end{cases}\EE
with $Q$ and $q$ related by the mirror map~\e_ref{mirmap_e}.

We will follow the five steps given in \cite{bcov0} to verify 
\e_ref{main_e2} and \e_ref{main_e1}:
\begin{enumerate}[label=(M\arabic*),leftmargin=*]

\item\label{transM_item} if $R\!\supset\!\Q$ is a field,
 $Y,Z\in H_{\T}^*(\Pn;R)\Lau{\hb}\big[\big[Q\big]\big]$, 
$$Y(x\!=\!\al_i,\hb,Q)\in R_{\al}(\hb)\big[\big[Q\big]\big]\subset 
R_{\al}\Lau{\hb}\big[\big[Q\big]\big] \qquad\forall\,i\!=\!1,2,\ldots,n,$$
$Z$ is recursive in the sense of Definition~\ref{recur_dfn}, and
$Y$ and $Z$ satisfy the mutual polynomiality condition (MPC) of Definition~\ref{MPC_dfn}, 
then the transforms of $Z$ of Lemma~\ref{Phistr_lmm4} are also recursive
and satisfy the MPC with respect to the corresponding transforms of~$Y$;


\item\label{uniqM_item} 
if  $R\!\supset\!\Q$ is a field,
$Y,Z\in H_{\T}^*(\Pn;R)\Lau{\hb}\big[\big[Q\big]\big]$,
$$Y(x\!=\!\al_i,\hb,Q)\in R_{\al}^*+Q\cdot R_{\al}(\hb)\big[\big[Q\big]\big]
\subset R_{\al}\Lau{\hb}\big[\big[Q\big]\big]
 \qquad\forall\,i\!=\!1,2,\ldots,n,$$
$Z$ is recursive in the sense of Definition~\ref{recur_dfn}, and
$Y$ and $Z$ satisfy the mutual polynomiality condition (MPC) of Definition~\ref{MPC_dfn}, 
then $Z$ is determined by its  ``mod $\hb^{-1}$ part'';


\item\label{recM_item} $\cY_p(x,\hb,Q)$ and $\cZ_p(x,\hb,Q)$ 
are $\fC$-recursive in the sense of Definition~\ref{recur_dfn} with $\fC$ as in \e_ref{Cdfn};

\item\label{polM_item}
$\big(\cY(x,\hb,Q),\cY_p(x,\hb,Q)\big)$ and $\big(\cZ(x,\hb,Q),\cZ_p(x,\hb,Q)\big)$ 
satisfy the MPC;

\item\label{equalmodM_item} the two sides of~\e_ref{symmrel_e} and~\e_ref{main_e1}, 
viewed as elements of $H_{\T}^*(\Pn)\Lau{\hb}\big[\big[Q\big]\big]$,  agree mod~$\hb^{-1}$. 

\end{enumerate}

The first two claims above, \ref{transM_item} and \ref{uniqM_item}, sum up
Lemma~\ref{Phistr_lmm4} and Proposition~\ref{uniqueness_prp}, respectively.
Section~\ref{Z_subs} extends the arguments in~\cite{bcov0} to show that 
the GW-generating function~$\cZ_p$  is $\fC$-recursive and satisfies
the MPC with respect to~$\cZ$; this confirms the claims of~\ref{recM_item} 
and~\ref{polM_item} concerning~$\cZ_p$.
It is immediate from~\e_ref{equivZ_e} that 
\begin{equation}\label{Z2wmod_e}
\cZ_p(x,\hb,Q) \cong x^{l+p} \quad(\mod\hb^{-1}\big)
\qquad\forall~p=-l,-l+1,\ldots
\end{equation}
In Section~\ref{Y_subs}, we show that $\cY_{-l}$ is $\fC$-recursive
and satisfies the MPC with respect to~$\cY$.
The admissibility of transforms~\ref{deriv_ch} and~\ref{mult_ch} 
of Lemma~\ref{Phistr_lmm4} implies that the power series $\cY_p$ 
defined by~\e_ref{main_e0} is also $\fC$-recursive 
and satisfies the MPC with respect to~$\cY$
for all~$p$, no matter what the coefficients $\ctC_{p,s}^{(r)}$ are;
this confirms the claims of~\ref{recM_item}  and~\ref{polM_item} concerning~$\cY_p$.
By~\e_ref{Yp_dfn_e2} and~\e_ref{Yp_dfn_e},
\BE{cYcong_e1} \cY_p(x,\hb,Q)\cong x^{l+p} \quad(\mod\hb^{-1}\big)
\qquad\forall~p=-l,-l+1,\ldots,0.\EE
By \e_ref{Crec_e}, 
\BE{cYcong_e2}\sum\limits_{r=0}^{p}
\sum\limits_{s=0}^{p-r}\ctC_{p,s}^{(r)}(q)\hb^{p-r-s}\D^s\cY_0(x,\hb,q)
\cong x^{l+p} \quad(\mod\hb^{-1}\big)
\qquad\forall~p\in\Z^{\ge0}\,\EE
if and only if the coefficients $\ctC_{p,s}^{(r)}(q)$ are 
given by~\e_ref{tiCrec_e}.\footnote{LHS of \e_ref{tiCrec_e} is the coefficient of 
$\hb^px^{l-r}(x/\hb)^s$ in the double sum viewed as an element of 
$\Q_{\al}[[q]][x]\Lau{\hb}$.}
Since $\ctC^{(0)}_{p,s}\!=\!\de_{p,s}$, 
\e_ref{main_e1} follows from \ref{transM_item}, \ref{uniqM_item}, \e_ref{Z1ms_e},
\e_ref{Z2wmod_e}, and~\e_ref{cYcong_e2}.


The proof of \e_ref{main_e2} follows the same principle, which we apply to
a multiple~of
\BE{Z_dfn}
\cZ(x_1,x_2,\hb_1,\hb_2,Q)\equiv
\frac{\lr{\a}x_1^l}{\hb_1+\hb_2}
\sum_{\begin{subarray}{c}p_1+p_2+r=n-1\\ p_1,p_2\ge0 \end{subarray}}
\!\!\!\!\!\!\!\!\!(-1)^r\si_rx_1^{p_1}x_2^{p_2}
+\cZ^*(x_1,x_2,\hb_1,\hb_2,Q).\EE
For each $i\!=\!1,2,\ldots,n$, let
\begin{equation}\label{phidfn_e}
\phi_i\equiv\prod_{k\neq i}(x\!-\!\al_k) \in H_{\T}^*(\Pn).
\end{equation}
By the Localization Theorem~\cite{ABo},
$\phi_i$ is the equivariant Poincar\'{e} dual of the fixed point $P_i\!\in\!\Pn$;
see \cite[Section~3.1]{bcov0}.
Since
\begin{equation}\label{eulerclass_e}
x|_{P_i}=\al_i\,, \qquad \E(\V_{\a})=\lr{\a}\,\ev_2^*x^l\,\E(\V_{\a}''),
\end{equation}
by the defining property of the cohomology push-forward \cite[(3.11)]{bcov0}
and the string relation \cite[Section~26.3]{MirSym}
\begin{equation}\label{pushch_e}\begin{split}
\cZ(\al_i,\al_j,\hb_1,\hb_2,Q)
&=\int_{P_i\times P_j}\cZ(x_1,x_2,\hb_1,\hb_2,Q)
=\int_{\Pn\times\Pn}\cZ(x_1,x_2,\hb_1,\hb_2,Q)\phi_i\!\times\!\phi_j\\
&=\frac{\lr{\a}\al_i^l}{\hb_1\!+\!\hb_2}
\prod_{k\neq i}(\al_j-\al_k)
+\sum_{d=1}^{\i}Q^d\!\!
\int_{\ov\M_{0,2}(\Pn,d)}\!
\frac{\E(\V_{\a})\,\ev_1^*\phi_i\,\ev_2^*\phi_j}{(\hb_1\!-\!\psi_1)(\hb_2\!-\!\psi_2)}\\
&=\lr{\a}\frac{\hb_1\hb_2}{\hb_1\!+\!\hb_2}\sum_{d=0}^{\i}Q^d
\int_{\ov\M_{0,3}(\Pn,d)}\!\!\!
\frac{\E(\V_{\a}'')\,\ev_1^*\phi_i\,\ev_2^*(x^l\,\phi_j)}{(\hb_1\!-\!\psi_1)(\hb_2\!-\!\psi_2)}.
\end{split}\end{equation}
Thus, by Lemmas~\ref{recgen_lmm} and~\ref{polgen_lmm},
$(\hb_1\!+\!\hb_2)\cZ(x_1,x_2,\hb_1,\hb_2,Q)$ is $\fC$-recursive and satisfies the MPC 
with respect to $\cZ(x,\hb,Q)$ for $(x,\hb)\!=\!(x_1,\hb_1)$ and $x_2\!=\!\al_j$ fixed.
By symmetry, it is also $\fC$-recursive and satisfies the MPC 
with respect to $\cZ(x,\hb,Q)$ for $(x,\hb)\!=\!(x_2,\hb_2)$ and 
$x_1\!=\!\al_i$ fixed.
By~\ref{uniqM_item},  it is thus sufficient to compare 
\begin{equation}\label{Zcomp_e} 
(\hb_1\!+\!\hb_2)\cZ(x_1,x_2,\hb_1,\hb_2,Q)
\quad\hbox{and}\quad
\lr{\a}\!\!\!\!\!\!
\sum_{\begin{subarray}{c}p_1+p_2+r=n-1\\ p_1,p_2,r\ge0 \end{subarray}}
\!\!\!\!\!\!\!\!\!(-1)^r\si_r\cZ_{p_1}(x_1,\hb_1,Q)\cZ_{p_2-l}(x_2,\hb_2,Q)
\end{equation}
for all $x_1\!=\!\al_i$ and $x_2\!=\!\al_j$ with $i,j\!=\!1,2,\ldots,n$
modulo~$\hb_1^{-1}$:
\begin{alignat*}{1}
&\big(\hb_1\!+\!\hb_2\big)\cZ(\al_i,\al_j,\hb_1,\hb_2,Q) \cong 
\lr{\a}\al_i^l\!\!\!\!\!\!\!  
\sum_{\begin{subarray}{c}p_1+p_2+r=n-1\\ p_1,p_2,r\ge0\end{subarray}}
\!\!\!\!\!\!\!\!(-1)^r\si_r\al_i^{p_1}\al_j^{p_2}
+\sum_{d=1}^{\i}Q^d\int_{\ov\M_{0,2}(\Pn,d)}\!\!\!
\frac{\E(\V_{\a})\ev_1^*\phi_i\ev_2^*\phi_j}{\hb_2\!-\!\psi_2};\\
&\lr{\a}\!\!\!\!\!\!
\sum_{\begin{subarray}{c}p_1+p_2+r=n-1\\ p_1,p_2,r\ge0\end{subarray}}
\!\!\!\!\!\!\!\!
(-1)^r\si_r \cZ_{p_1}(\al_i,\hb_1,Q)\cZ_{p_2-l}(\al_j,\hb_2,Q)\cong
\lr{\a}\!\!\!\!\!\! 
\sum_{\begin{subarray}{c}p_1+p_2+r=n-1\\ p_1,p_2,r\ge0\end{subarray}}\!\!\!\!\!\!\!
(-1)^r\si_r\al_i^{l+p_1}\cZ_{p_2-l}(\al_j,\hb_2,Q).
\end{alignat*}
In order to see that the two right-hand side power series are the same,
it is sufficient to compare them modulo~$\hb_2^{-1}$:
\begin{alignat*}{2}
&\lr{\a}\al_i^l\!\!\!\!\!\!\!\!\!\!
\sum_{\begin{subarray}{c}p_1+p_2+r=n-1\\ p_1,p_2,r\ge0\end{subarray}}
\!\!\!\!\!\!\!\!\! (-1)^r\si_r \al_i^{p_1}\al_j^{p_2}
+\sum_{d=1}^{\i}Q^d\int_{\ov\M_{0,2}(\Pn,d)}\!\!\!
\frac{\E(\V_{\a})\ev_1^*\phi_i\ev_2^*\phi_j}{\hb_2\!-\!\psi_2}\cong
\lr{\a}\al_i^l\!\!\!\!\!\!
\sum_{\begin{subarray}{c}p_1+p_2+r=n-1\\ p_1,p_2,r\ge0\end{subarray}}
   \!\!\!\!\!\!\!\!(-1)^r\si_r \al_i^{p_1}\al_j^{p_2};\\
&\lr{\a}\!\!\!\!\!
\sum_{\begin{subarray}{c}p_1+p_2+r=n-1\\ p_1,p_2,r\ge0\end{subarray}}
\!\!\!\!\!\!
(-1)^r\si_r\al_i^{l+p_1}\cZ_{p_2-l}(\al_j,\hb_2,Q)\cong
\lr{\a}\!\!\!\!\!\!
\sum_{\begin{subarray}{c}p_1+p_2+r=n-1\\ p_1,p_2,r\ge0\end{subarray}}
\!\!\!\!\!\!(-1)^r\si_r\al_i^{l+p_1}\al_j^{p_2}.
\end{alignat*}
From this we conclude that the two expressions in~\e_ref{Zcomp_e} are the same;
this proves~\e_ref{main_e2}.

\subsection{Recursivity, polynomiality,  and admissible transforms} 
\label{poliC_subs}

We now describe properties of power series, such as $\cY_p$ and $\cZ_p$,
that impose severe restrictions on them; see Proposition~\ref{uniqueness_prp}.
This section sums up the results in \cite[Sections 2.1, 2.2]{bcov0},
extending them slightly.
Let
$$[n]=\{1,2,\ldots,n\}.$$
If $R$ is a ring, $f\!\in\!R[[Q]]$, and $d\!\in\!\Z^{\ge0}$,
let $\coeff{f}_{Q;d}\!\in\!R$ denote the coefficient of~$Q^d$ in~$f$.

\begin{dfn}\label{recur_dfn}
Let $R\!\supset\!\Q$ be a field and $C\equiv(C_i^j(d))_{d,i,j\in\Z^{+}}$
any collection of elements of~$R_{\al}$.
A~power series $Z\!\in\!H_{\T}^*(\Pn;R)\Lau{\hb}\big[\big[Q\big]\big]$ is \sf{$C$-recursive} if
the following holds:
if $d^*\!\in\!\Z^{\ge 0}$ is such~that 
$$\bigcoeff{Z(x\!=\!\al_i,\hb,Q)}_{Q;d^*-d}\in R_{\al}(\hb)
\qquad\forall\,d\!\in\![d^*],\,i\!\in\![n],$$
and $\bigcoeff{Z(\al_i,\hb,Q)}_{Q;d}$ is regular at 
$\hb\!=\!(\al_i\!-\!\al_j)/d$ for all $d\!<\!d^*$ and $i\!\neq\!j$, then
\BE{recur_dfn_e}
\bigcoeff{Z(\al_i,\hb,Q)}_{Q;d^*}-
\sum_{d=1}^{d^*}\sum_{j\neq i}\frac{C_i^j(d)}{\hb-\frac{\al_j-\al_i}{d}}
\bigcoeff{Z(\al_j,z,Q)}_{Q;d^*-d}\big|_{z=\frac{\al_j-\al_i}{d}}
\in R_{\al}\big[\hb,\hb^{-1}\big].\EE
A power  series $Z\!\in\!H_{\T}^*(\Pn;R)\Lau{\hb}\big[\big[Q\big]\big]$ is called \sf{recursive} if
it is $C$-recursive for some collection  
$C\equiv(C_i^j(d))_{d,i,j\in\Z^{+}}$ of elements of~$R_{\al}$.
\end{dfn}

Thus, if $Z\!\in\!H_{\T}^*(\Pn;R)\Lau{\hb}\big[\big[Q\big]\big]$ is recursive, 
then
$$Z(x\!=\!\al_i,\hb,Q)\in R_{\al}(\hb)\big[\big[Q\big]\big]
\qquad\forall\,i\!\in\![n].$$

\begin{dfn}\label{MPC_dfn}
Let $R\!\supset\!\Q$ be a field. For any
$$Y\!\equiv\!Y(x,\hb,Q),Z\!\equiv\!Z(x,\hb,Q)\in H_{\T}^*(\Pn;R)\Lau{\hb}\big[\big[Q\big]\big],$$
define $\Phi_{Y,Z}\in R_{\al}\Lau{\hb}\big[\big[z,Q\big]\big]$ by
$$\Phi_{Y,Z}(\hb,z,Q)=
\sum_{i=1}^n\frac{\ne^{\al_iz}}{\prod\limits_{k\neq i}(\al_i\!-\!\al_k)}
Y\big(\al_i,\hb,Q\ne^{\hb z}\big)Z(\al_i,-\hb,Q).$$
If $Y,Z\!\in\!H_{\T}^*(\Pn;R)\Lau{\hb}\big[\big[Q\big]\big]$,
the pair $(Y,Z)$ \sf{satisfies the mutual polynomiality condition} (\sf{MPC})~if  
$$\Phi_{Y,Z}\in R_{\al}[\hb]\big[\big[z,Q\big]\big].$$
\end{dfn}

If $Y,Z\!\in\!H_{\T}^*(\Pn;R)\Lau{\hb}\big[\big[Q\big]\big]$ and
\BE{YZrat_e}Y(x\!=\!\al_i,\hb,Q),
Z(x\!=\!\al_i,\hb,Q)\in R_{\al}(\hb)\big[\big[Q\big]\big] \qquad\forall\,i\!\in\![n],\EE
then the pair $(Y,Z)$ satisfies the MPC if and only if 
the pair $(Z,Y)$ does; see \cite[Lemma~2.2]{bcov0}.
Thus, if \e_ref{YZrat_e} holds, the statement that $Y$ and $Z$ satisfy the MPC
is unambiguous.

\begin{prp}\label{uniqueness_prp}
Let $R\!\supset\!\Q$ be a field.
If $Y,Z\in H_{\T}^*(\Pn;R)\Lau{\hb}\big[\big[Q\big]\big]$,
$$Y(x\!=\!\al_i,\hb,Q)\in R_{\al}^*+Q\cdot R_{\al}(\hb)\big[\big[Q\big]\big]
\subset R_{\al}\Lau{\hb}\big[\big[Q\big]\big]
 \qquad\forall\,i\!\in\![n],$$
$Z$ is recursive, and $Y$ and $Z$ satisfy the MPC, 
then  $Z\cong 0~(\mod \hb^{-1})$ if and only if $Z=0$.
\end{prp}

This is essentially \cite[Proposition 2.1]{bcov0}, 
where the recursivity for~$Z$ is with respect to 
a specific collection~$C$.\footnote{The 
assumptions on the second line in \cite[Proposition 2.1]{bcov0}
should have been that 
$$Y,Z\in H_{\T}^*(\Pn)\Lau{\hb}\big[\big[Q\big]\big]
\qquad\hbox{and}\qquad
Y(x\!=\!\al_i,\hb,Q),Z(x\!=\!\al_i,\hb,Q)\!\in\!\Q_{\al}(\hb)\big[\big[Q\big]\big]
\quad\forall\,i\!\in\![n].$$
Since $Z$ is recursive in our case,  
$Z(\al_i,\hb,Q)\!\in\!R_{\al}(\hb)[[Q]]$.}
However, the proof of \cite[Proposition 2.1]{bcov0} does not use any properties
of the specific collection needed for the purposes of~\cite{bcov0}
and thus applies to any collection of elements of~$R_{\al}$.

\begin{lmm}\label{Phistr_lmm4}
Let $R\!\supset\!\Q$ be a field.
If  $Y,Z\!\in\!H_{\T}^*(\Pn;R)\Lau{\hb}\big[\big[Q\big]\big]$,
$$Y(x\!=\!\al_i,\hb,Q)\in R_{\al}(\hb)\big[\big[Q\big]\big]
\subset R_{\al}\Lau{\hb}\big[\big[Q\big]\big]
 \qquad\forall\,i\!\in\![n],$$
$Z$ is recursive, and
$Y$ and $Z$ satisfy the MPC, then
\begin{enumerate}[label=(\roman*)]

\item\label{deriv_ch}  
$\bar{Z}\!\equiv\!\left\{x\!+\!\hb\, Q\frac{\nd}{\nd Q}\right\}Z$
is recursive and satisfies the MPC with respect to~$Y$;


\item\label{mult_ch} if $f\!\in\!R_{\al}[\hb]\big[\big[Q\big]\big]$, then 
$fZ$ is recursive and satisfies the MPC with respect to~$Y$;


\item\label{multoverh_ch} if $f\!\in\!Q\cdot R_{\al}[[Q]]$, then
$\ne^{f/\hb}Z$ is recursive and  satisfies the MPC with respect to~$\ne^{f/\hb}Y$;


\item\label{transf_ch} if $g\!\in\!Q\cdot R_{\al}[[Q]]$,  
$$\bar{Y}(x,\hb,Q)\!\equiv\!\ne^{x g(Q)/\hb}Y\big(x,\hb,Q \ne^{g(Q)}\big) \quad\hbox{and}\quad
\bar{Z}(x,\hb,Q)\!\equiv\!\ne^{x g(Q)/\hb}Z\big(x,\hb,Q \ne^{g(Q)}\big),$$
then $\bar{Z}$ is recursive and  satisfies the MPC with respect to~$\bar{Y}$.
\end{enumerate}
\end{lmm}

Parts~\ref{deriv_ch}, \ref{multoverh_ch}, and \ref{transf_ch}
of this lemma are (i), (iv)\footnote{There is a typo in the statement of \cite[Lemma 2.3,(iv)]{bcov0}:
$\ti\Q_{\al}[u]$ should be $\ti\Q_{\al}[[u]]$.}, and (v), respectively, of \cite[Lemma 2.3]{bcov0} 
with $Q\!=\!u\!=\!\ne^t$ and $\Q$ and $\ti\Q_{\al}$ replaced by~$R_{\al}$.
The conclusions in these parts continue to hold because their proofs given in \cite{bcov0}
do not use anything about the collection $C=(C_i^j(d))_{d,i,j}$ of~\cite{bcov0}.
Moreover, replacing $\Q$ and $\ti\Q_{\al}$ by~$R_{\al}$ does not affect the proofs either.
Part~\ref{mult_ch} of Lemma~\ref{Phistr_lmm4} generalizes parts~(ii) and~(iii)
of \cite[Lemma~2.3]{bcov0}
as suggested in the remark following \cite[Lemma~2.3]{bcov0};
it continues to hold for the same reasons as 
the other parts do.\footnote{There are two typos in the proof 
of \cite[Lemma~2.3]{bcov0}:
the last term in the last equality of (2.19) should be
$$C_i^j(d)u^d\left\{d+\frac{\nd}{\nd t}\right\}Z((\al_j-\al_i)/d,\al_j,u)$$
and the third argument of $Z$ on both sides of the first equation 
in the proof of (v) should be
$u\ne^{g(u)}.$}

\subsection{Recursivity and MPC for GW power series}
\label{Z_subs}

This section establishes that the power series $\cZ_p(x,\hb,Q)$ 
defined in~\e_ref{equivZ_e} is $\fC$-recursive in the sense of Definition~\ref{recur_dfn}, 
where
\BE{Cdfn}
\fC_i^j(d)\equiv  \frac{\prod\limits_{k=1}^l    
          \prod\limits_{r=0}^{a_kd-1}\left(a_k\al_i+r\,\frac{\al_j-\al_i}{d}\right)}
{d\underset{(r,k)\neq(d,j)}{\prod\limits_{r=1}^d\prod\limits_{k=1}^n}
            \left(\al_i-\al_k+r\,\frac{\al_j-\al_i}{d}\right)}\in\Q_{\al} \,.\EE
We also show that $\cZ_p(x,\hb,Q)$ satisfies the MPC of Definition~\ref{MPC_dfn}
with respect to the power series $\cZ(x,\hb,Q)$ defined in~\e_ref{Z_e}.
These statements are 
special cases of Lemmas~\ref{recgen_lmm} and~\ref{polgen_lmm} below, since
\begin{equation*}\label{strrel_e}\begin{split}
\int_{\ov\M_{0,2}(\Pn,d)}\!\!\!
\frac{\E(\V_{\a}'')\ev_2^*\eta}{\hb\!-\!\psi_1}\ev_1^*\phi_i
=\hb\!\!\int_{\ov\M_{0,3}(\Pn,d)}\!\!\!
\frac{\E(\V_{\a}'')\ev_2^*\eta}{\hb\!-\!\psi_1}\ev_1^*\phi_i
\quad\forall\,\eta\!\in\!H_{\T}^*(\Pn),\,d\!\in\!\Z^{+},
\end{split}\end{equation*}
by the string relation \cite[Section~26.3]{MirSym}.

\begin{lmm}\label{recgen_lmm} 
If $m\!\ge\!3$, $\ev_j\!:\ov\M_{0,m}(\Pn,d)\!\lra\!\Pn$ is 
the evaluation at the $j$-th marked point, and
$\eta_j\!\in\!H_{\T}^*(\Pn)$ 
and $\be_j\!\in\!\Z^{\geq 0}$ for $j\!=\!1,\ldots,m$,
then the power series
\begin{equation*}\label{recgen_e}
\cZ_{\eta,\be}(x,\hb,Q)\equiv\sum_{d=0}^{\i}Q^d
\ev_{1*}\!\left[
\frac{\E(\V_{\a}'')}{\hb\!-\!\psi_1}
\prod\limits_{j=2}^{m} \big(\psi_j^{\be_j}\ev_j^*\eta_j\big)\right]
\in H_{\T}^*(\Pn)\Lau{\hb}\big[\big[Q\big]\big]
\end{equation*}
is $\fC$-recursive with $\fC$ given by \e_ref{Cdfn}.
\end{lmm}

\begin{lmm}\label{polgen_lmm} 
For all $m\!\ge\!3$, $\eta_j\!\in\!H_{\T}^*(\Pn)$,
and $\be_j\!\in\!\Z^{\geq 0}$, the power series $\hb^{m-2}\cZ_{\eta,\be}(x,\hb,Q)$,
with $\cZ_{\eta,\be}(x,\hb,Q)$ as in Lemma~\ref{recgen_lmm},
satisfies the MPC with respect to $\cZ(x,\hb,Q)$.
\end{lmm}

Similarly to \e_ref{pushch_e},
$$\cZ_{\eta,\be}(\al_i,\hb,Q)=
\sum_{d=0}^{\i}Q^d\!\int_{\ov\M_{0,m}(\Pn,d)}\bigg(
\frac{\E(\V_{\a}'')\ev_1^*\phi_i}{\hb\!-\!\psi_1}
\prod\limits_{j=2}^{m} \big(\psi_j^{\be_j}\ev_j^*\eta_j\big)\bigg)$$
for all $i$.
Thus, the $l\!=\!1$ cases of  Lemmas~\ref{recgen_lmm} and \ref{polgen_lmm} 
are  Lemmas~1.1 and~1.2 in \cite[Section~1.3]{bcov0}.\footnote{There is a typo in 
\cite[(1.19)]{bcov0}: $x$ should be $\al_i$.}
The proof of  \cite[Lemma~1.1]{bcov0} consists of applying the Localization 
Theorem~\cite{ABo} to $\cZ_{\eta,\be}(\al_i,\hb,Q)$ with respect to 
the $\T$-action on $\ov\M_{0,m}(\Pn,d)$ induced by the standard $\T$-action 
on $\Pn$.\footnote{There are two typos in the proof of this lemma 
in \cite{bcov0}: in the second factor of the second equation in
\cite[(3.20)]{bcov0}, $\Gamma_0$ should be $\Gamma_c$ and on the right-hand side 
of \cite[(3.23)]{bcov0}, $\cZ_{\Gamma}$ should be $\cZ_{\Gamma_c}$.}
The proof of \cite[Lemma~1.2]{bcov0} consists of applying the Localization Theorem
on a certain subspace of $\ov\M_{0,m}(\P^1\!\times\!\Pn,(1,d))$ with respect to
the action of the $(n\!+\!1)$-torus $\T\times\C^*$ induced by a certain action 
of this torus on $\P^1\times\Pn$.\footnote{The portion of \cite[Section~3.3]{bcov0} following
the statement of Lemma~3.1 is unnecessary: in light of \cite[(3.27)]{bcov0},
\cite[(3.29)]{bcov0} immediately implies the statement of \cite[Lemma~1.2]{bcov0}.
There are also five typos in the proof 
of this lemma in \cite{bcov0}:
in the first equation of \cite[(3.32)]{bcov0} and on the left-hand side 
of \cite[(3.33)]{bcov0}, $\gamma^*$ should be $\gamma_1^*$;
in the second equation of \cite[(3.32)]{bcov0}, $\Gamma_0$ should be $\Gamma_1$ 
when it is pulled-back by $\pi_1^*$ and $\Gamma_2$ when it is pulled-back by $\pi_2^*$;
in the second bracket on the right-hand side
of \cite[(3.33)]{bcov0}, the numerator of the integrand
should be $\E(\V_{0}'')\eta^{\be}\ev_1^*\phi_i$.}
The proofs of  Lemmas~\ref{recgen_lmm} and~\ref{polgen_lmm} are nearly identical:
the base spaces and the tori actions remain the same, while the Euler classes of the vector 
bundles in our case are products of the Euler classes in~\cite{bcov0}.
So only the following modifications need to be made to the proofs in~\cite{bcov0}: 
\begin{enumerate}[label=(\arabic*)]
\item $\V_0''$ should be replaced by $\V_{\a}''$ everywhere in 
\cite[Sections 3.2-3.4]{bcov0};
\item the first equation in \cite[(3.21)]{bcov0} should be replaced by
$$\E(\V_{\a}'')=
\prod_{k=1}^l\prod_{r=0}^{a_kd_0-1}\left(a_k\al_i+r\frac{\al_j-\al_i}{d_0}\right);$$
\item on the right hand-sides of \cite[(3.23), (3.24)]{bcov0} and 
in the last sentence of \cite[Section 3.2]{bcov0},
$C_i^j(d_0)$ should be replaced by $\fC_i^j(d_0).$
\end{enumerate}

\subsection{Properties of the hypergeometric series $\cY_p$}
\label{Y_subs}

Below we verify the two claims 
concerning the power series $\cY_p(x,\hb,q)$
that remain to be proved:
\begin{enumerate}[label=(\alph*)]
\item\label{Ymprec_item} $\cY_{-l}(x,\hb,q)$ is $\fC$-recursive; 
\item\label{YmpPhi_item} $\Phi_{\cY,\cY_{-l}}\!\in\!\Q_{\al}[\hb]\big[\big[z,q\big]\big]$.
\end{enumerate}
The proofs are similar to \cite[Section 2.3]{bcov0},
which treats the $l\!=\!1$ Calabi-Yau case.
The proof of~\ref{Ymprec_item} in~\cite{bcov0}
is applicable in the Fano case as well and so requires 
little modification;
the consideration  of the Fano case in~\ref{YmpPhi_item}
requires only slightly more care.

If $f\!=\!f(z)$ is a rational function in $\hb$ and possibly some other
variables, for any $z_0\!\in\!\P^1\!\supset\!\C$ let
\BE{Rsdfn_e}\Rs{z=z_0}f(z) \equiv \frac{1}{2\pi\I}\oint f(z)\nd z,\EE
where the integral is taken over a positively oriented loop around $z\!=\!z_0$
with no other singular points of $f\nd z$,
denote the residue of the 1-form~$f\nd z$.
If $z_1,\ldots,z_k\!\in\!\P^1$ is any collection of points, let
\BE{Rssumdfn_e}\Rs{z=z_1,\ldots,z_k}f(z)
\equiv\sum_{i=1}^{k}\Rs{z=z_i}f(z).\EE
If $f$ is regular at $z\!=\!0$, let $\left\llbracket f\right\rrbracket_{z;p}$ denote 
the coefficient of $z^p$ in the power series expansion of $f$ around $z\!=\!0$.

\begin{proof}[Proof of~\ref{Ymprec_item}]
In this argument, we view $\cY_{-l}$ as an element of $\Q_{\al}(x,\hb)[[q]]$.
By \e_ref{Ymldfn_e} and \e_ref{Cdfn},
$$\frac{\fC_i^j(d)q^d}{\hb-\frac{\al_j-\al_i}{d}} 
\, \cY_{-l}\big(\al_j,(\al_j\!-\!\al_i)/d,q\big)
=\Rs{z=\frac{\al_j-\al_i}{d}}\bigg\{\frac{1}{\hb\!-\!z}\cY_{-l}(\al_i,z,q)\bigg\}.$$
Thus, by the Residue Theorem on~$S^2$,
\begin{equation}\label{recurpf_e1}\begin{split}
\sum_{d=1}^{\i}\sum_{j\neq i}\frac{\fC_i^j(d)q^d}{\hb-\frac{\al_j-\al_i}{d}}
&\cY_{-l}\big(\al_j,(\al_j\!-\!\al_i)/d,q\big)
=-\Rs{z=\hb,0,\i}\bigg\{\frac{1}{\hb\!-\!z}\cY_{-l}(\al_i,z,q)\bigg\}\\
&\qquad\qquad\qquad\qquad
=\cY_{-l}(\al_i,\hb,q)-\Rs{z=0,\i}\bigg\{\frac{1}{\hb\!-\!z}\cY_{-l}(\al_i,z,q)\bigg\}.
\end{split}\end{equation}
On the other hand, 
\begin{alignat*}{1}
\Rs{z=\i}\bigg\{\frac{1}{\hb\!-\!z}\cY_{-l}(\al_i,z,q)\bigg\}&=1,\\
\Rs{z=0}\bigg\{\frac{1}{\hb\!-\!z}\coeff{\cY_{-l}(\al_i,z,q)}_{q;d}\bigg\}
&=\left\llbracket\frac{1}{\hb\!-\!z}
\frac{\prod\limits_{k=1}^l\prod\limits_{r=0}^{a_kd-1}(a_k\al_i\!+\!rz)}
{d!\prod\limits_{r=1}^{d}\prod\limits_{k\neq i}(\al_i\!-\!\al_k\!+\!rz)}
\right\rrbracket_{z;d-1}\in\Q_{\al}[\hb^{-1}]\quad\forall\,d\!\in\!\Z^{\ge0}.
\end{alignat*}
Thus, \e_ref{recurpf_e1} implies that $\cY_{-l}$ is $\fC$-recursive.
\end{proof}

\begin{proof}[Proof of~\ref{YmpPhi_item}]
In this argument, we view $\cY$ and $\cY_{-l}$ as elements of $\Q_{\al}[x][[\hb^{-1},q]]$;
in particular,
$$\frac{\ne^{xz}}{\prod\limits_{k=1}^n(x\!-\!\al_k)}
\cY\big(x,\hb,q\ne^{\hb z}\big)\cY_{-l}(-x,\hb,q)
\in \Q_{\al}(x)[[\hb^{-1},z,q]]$$
viewed as a function of $x$ has residues only at $x\!=\!\al_i$ with $i\!\in\![n]$
and $x\!=\!\i$.
By \e_ref{Ydfn_e} and~\e_ref{Ymldfn_e},
$$\frac{\ne^{\al_iz}}{\prod\limits_{k\neq i}(\al_i\!-\!\al_k)}
\cY\big(\al_i,\hb,q\ne^{\hb z}\big)\cY_{-l}(\al_i,-\hb,q)
=\Rs{x=\al_i}\left\{\frac{\ne^{xz}}{\prod\limits_{k=1}^n(x\!-\!\al_k)}
\cY\big(x,\hb,q\ne^{\hb z}\big)\cY_{-l}(x,-\hb,q)\right\}.$$
Thus, by the Residue Theorem on~$S^2$,
\begin{equation*}\begin{split}
\Phi_{\cY,\cY_{-l}}(\hb,z,q) &=-
\Rs{x=\i}\left\{\frac{\ne^{xz}}{\prod\limits_{k=1}^n(x\!-\!\al_k)}
\cY\big(x,\hb,q\ne^{\hb z}\big)\cY_{-l}(x,-\hb,q)\right\}\\
&=\sum_{d_1,d_2=0}^{\i}\sum_{p=0}^{\i}
\frac{z^{n-1+p+\nu_{\a}(d_1+d_2)}}
{(n-1+p+\nu_{\a}(d_1\!+\!d_2))!}
q^{d_1+d_2}\ne^{\hb d_1z}\\
&\qquad\times\left\llbracket\frac{1}{\prod\limits_{k=1}^n(1\!-\!\al_kw)}
\frac{\prod\limits_{k=1}^l\prod\limits_{r=1}^{a_kd_1}(a_k\!+\!r\hb w)}
{\prod\limits_{r=1}^{d_1}\prod\limits_{k=1}^{n}(1\!-\!(\al_k\!-\!r\hb)w)}
\cdot\frac{\prod\limits_{k=1}^l\prod\limits_{r=0}^{a_kd_2-1}(a_k\!-\!r\hb w)}
{\prod\limits_{r=1}^{d_2}\prod\limits_{k=1}^{n}(1\!-\!(\al_k\!+\!r\hb)w)}\right\rrbracket_{w;p}.
\end{split}\end{equation*}
The $(d_1,d_2,p)$-summand above is $q^{d_1\!+\!d_2}$\!
times an element of $\Q_{\al}[\hb]\big[\big[z\big]\big]$. 
\end{proof}

\section{Some Applications of Theorem~\ref{main_thm}}
\label{mainthmapp_sec}

Sections~\ref{fewobs_sec} and~\ref{ctC_subs} below contain some corollaries of 
Theorem~\ref{main_thm}.
While the identities we obtain in these sections appear purely technical
in nature, they enter in vital ways in the proofs of Theorem~\ref{annulus_thm}
and Theorem~\ref{klein_thm} in the rest of the paper.

\subsection{Differentiating \e_ref{main_e2}}
\label{fewobs_sec}

In this section, we obtain a description for a derivative of \e_ref{main_e2}.
Let $\nD= q\frac{\nd}{\nd q}$.

By \e_ref{main_e0}, \e_ref{Crec_e}, and \e_ref{tiCrec_e},  
\BE{Ypexp_e}
\cY_p(x,\hb,q)=x^{l+p}+x^l\sum_{s=1}^{\i}\sum_{r=0}^{p+s}B_{p,s}^{(r)}(q)x^{p+s-r}\hb^{-s}
\EE
for some $B_{p,s}^{(r)}(q)\in q\cdot\Q_{\al}[[q]]$ such that the coefficient of $q^d$
in $B_{p,s}^{(r)}(q)$ is a homogeneous polynomial in $\al_1,\ldots,\al_n$ of degree $r\!-\!\nu_{\a}d$.
If $\nu_{\a}\!=\!0$,
\BE{BvsI_e} I_{p+1}(q)=1+DB_{p,1}^{(0)}(q)\EE
by~\e_ref{YvsF_e}, \e_ref{Ipdfn_e}, and
$$x^lF(x/{\hb},q)=\{x\!+\!\hb\nD\}\D^{-1}\cY_0(x,\hb,q)\big|_{\al=0};$$
if $\nu_{\a}\!\neq\!0$, 
\e_ref{BvsI_e} is immediate from $B_{p,1}^{(0)}$ being a constant in $q$.
For any $i,j\!=\!1,2,\ldots,n$, let
\begin{equation}\label{Siisums_e}\begin{split}
S_{ij}^{(2)}(\hb_1,\hb_2,q)
=\sum_{\begin{subarray}{c}p_1+p_2+r=n-l\\ p_1,p_2\ge0\end{subarray}}
\!\!\!\!\!\!\!\!\!\!(-1)^r\si_r\!
\Bigg[\sum_{r'=0}^{p_1}\!DB_{p_1-1,1}^{(r')}(q)
\cY_{p_1-r'}(\al_i,\hb_1,q)\Bigg]\!\cY_{p_2}(\al_j,\hb_2,q).
\end{split}\end{equation}

\begin{lmm}\label{B1_lmm}
For all $i=1,2,\ldots,n$ and $p\!\ge\!-l$, 
$$\big\{\al_i+\hb\nD\big\}\cY_p(\al_i,\hb,q)
=\cY_{p+1}(\al_i,\hb,q)+
\sum_{r=0}^{p+1}\nD B^{(r)}_{p,1}(q)\cY_{p+1-r}(\al_i,\hb,q).$$
\end{lmm}

\begin{proof}
Since both sides of this identity with $\al_i$ replaced by $x$ are $\fC$-recursive and satisfy the MPC
with respect to $\cY(x,\hb,q)$, it is sufficient to verify this identity modulo~$\hb^{-1}$.
The latter is immediate from \e_ref{Ypexp_e}.
\end{proof}

\begin{lmm}\label{Sii_lmm}
For all $i,j=1,2,\ldots,n$,
\begin{equation*}\begin{split}
&\bigg\{\frac{\al_i}{\hb_1}+\frac{\al_j}{\hb_2}+\nD\bigg\}
\sum_{\begin{subarray}{c}p_1+p_2+r=n-1\\p_1,p_2\ge0\end{subarray}}
\!\!\!\!\!\!\!\!\!\!
(-1)^r\!\si_r\cY_{p_1}(\al_i,\hb_1,q)\cY_{p_2-l}(\al_j,\hb_2,q)\\
&\hspace{.8in}
=\left(\frac{1}{\hb_1}+\frac{1}{\hb_2}\right)
\left\{\sum_{\begin{subarray}{c}p_1+p_2+r=n-l\\ p_1,p_2\ge0\end{subarray}}
-\sum_{\begin{subarray}{c}p_1+p_2+r=n-l\\ p_1,p_2\le-1\end{subarray}}\right\}
(-1)^r\si_r\cY_{p_1}(\al_i,\hb_1,q)\cY_{p_2}(\al_j,\hb_2,q)\\
&\hspace{1in}
+\frac{1}{\hb_1}S_{ij}^{(2)}(\hb_1,\hb_2,q)
+\frac{1}{\hb_2}S_{ji}^{(2)}(\hb_2,\hb_1,q).
\end{split}\end{equation*}
\end{lmm}

\begin{proof}
For any $i,j=1,2,\ldots,n$, we define 
\begin{equation}\label{Siimpdfn_e}\begin{split}
S_{ij}^{(-)}(\hb_1,\hb_2,q)
&\equiv\cY_0(\al_i,\hb_1,q)
\sum_{\begin{subarray}{c}p+r=n-l\\ p\le-1\end{subarray}}\!\!\!\!
(-1)^r\si_r\cY_p(\al_j,\hb_2,q),\\
S_{ij}^{(+)}(\hb_1,\hb_2,q)
&\equiv\cY_0(\al_i,\hb_1,q)
\sum_{\begin{subarray}{c}p+r=n-l\\ p\ge0\end{subarray}}\!\!\!\!
(-1)^r\si_r\cY_p(\al_j,\hb_2,q)
=-S_{ij}^{(-)}(\hb_1,\hb_2,q);
\end{split}\end{equation}
the equality above follows from \e_ref{symmrel_e} and \e_ref{main_e1}.
By Lemma~\ref{B1_lmm} and the $p\!=\!-1$ case of \e_ref{BvsI_e}, 
\begin{equation}\label{DYneg_e}\begin{split}
&\bigg\{\frac{\al_i}{\hb_1}+\frac{\al_j}{\hb_2}+\nD\bigg\}
\sum_{\begin{subarray}{c}p_1+p_2+r=n-1-l\\p_1,p_2\le-1\end{subarray}}
\!\!\!\!\!\!\!\!\!\!\!\!\!
(-1)^r\!\si_r\cY_{p_1}(\al_i,\hb_1,q)\cY_{p_2}(\al_j,\hb_2,q)\\
&\hspace{1.5in}=\left(\frac{1}{\hb_1}+\frac{1}{\hb_2}\right)
\!\sum_{\begin{subarray}{c}p_1+p_2+r=n-l\\p_1,p_2\le-1\end{subarray}}
\!\!\!\!\!\!\!\!\!
(-1)^r\si_r\cY_{p_1}(\al_i,\hb_1,q)\cY_{p_2}(\al_j,\hb_2,q)\\
&\hspace{1.8in}
+\frac{I_0(q)}{\hb_1}S_{ij}^{(-)}(\hb_1,\hb_2,q)
+\frac{I_0(q)}{\hb_2}S_{ji}^{(-)}(\hb_2,\hb_1,q).
\end{split}\end{equation}
Similarly,
\begin{equation}\label{DYpos_e}\begin{split}
&\bigg\{\frac{\al_i}{\hb_1}+\frac{\al_j}{\hb_2}+\nD\bigg\}
\sum_{\begin{subarray}{c}p_1+p_2+r=n-1-l\\p_1,p_2\ge0\end{subarray}}
\!\!\!\!\!\!\!\!\!\!
(-1)^r\!\si_r\cY_{p_1}(\al_i,\hb_1,q)\cY_{p_2}(\al_j,\hb_2,q)\\
&\quad=\left(\frac{1}{\hb_1}+\frac{1}{\hb_2}\right)
\sum_{\begin{subarray}{c}p_1+p_2+r=n-l\\p_1,p_2\ge0\end{subarray}}
\!\!\!\!\!\!\!\!\!\!
(-1)^r\!\si_r\cY_{p_1}(\al_i,\hb_1,q)\cY_{p_2}(\al_j,\hb_2,q)\\
&\qquad~~+\frac{1}{\hb_1}S_{ij}^{(2)}(\hb_1,\hb_2,q)+\frac{1}{\hb_2}S_{ji}^{(2)}(\hb_2,\hb_1,q)
-\frac{I_0(q)}{\hb_1}S_{ij}^{(+)}(\hb_1,\hb_2,q)
-\frac{I_0(q)}{\hb_2}S_{ji}^{(+)}(\hb_2,\hb_1,q).
\end{split}\end{equation}
Since
\begin{equation*}\label{Ysum_e}\begin{split}
\sum_{\begin{subarray}{c}p_1+p_2+r=n-1\\p_1,p_2\ge0\end{subarray}}\!\!\!\!\!\!\!\!\!\!
(-1)^r\!\si_r\cY_{p_1}(\al_i,\hb_1,q)\cY_{p_2-l}(\al_j,\hb_2,q)
=\sum_{\begin{subarray}{c}p_1+p_2+r=n-1-l\\p_1,p_2\ge0\end{subarray}}
\!\!\!\!\!\!\!\!\!\!\!\!
(-1)^r\!\si_r\cY_{p_1}(\al_i,\hb_1,q)\cY_{p_2}(\al_j,\hb_2,q)&\\
-\!\!\!\!\!\!\!\!\!\!
\sum_{\begin{subarray}{c}p_1+p_2+r=n-1-l\\p_1,p_2\leq-1\end{subarray}}
\!\!\!\!\!\!\!\!\!\!\!\!
(-1)^r\si_r\cY_{p_1}(\al_i,\hb_1,q)\cY_{p_2}(\al_j,\hb_2,q)&
\end{split}\end{equation*}
by \e_ref{symmrel_e} and \e_ref{main_e1},
the claim now follows from \e_ref{DYneg_e}, \e_ref{DYpos_e}, and 
the equality in~\e_ref{Siimpdfn_e}.
\end{proof}

\begin{crl}\label{C1_crl}
If $p_1,p_2\!\ge\!0$, then
\begin{equation*}
\sum_{r=0}^p(-1)^r\si_r B_{p-1-r+p_1,1}^{(p-r)}(q)
=\sum_{r=0}^p(-1)^r\si_r B_{p-1-r+p_2,1}^{(p-r)}(q),
\end{equation*}
where $p\!=\!n\!-\!l\!-\!(p_1\!+\!p_2)$.
\end{crl}

\begin{proof}
By \e_ref{main_e2} and \e_ref{main_e1},
\begin{equation*}\label{Csum_e}
\hb\cdot\bigg\{\frac{\al_i}{\hb_1}+\frac{\al_j}{\hb_2}+\nD\bigg\}
\sum_{\begin{subarray}{c}p_1+p_2+r=n-1\\p_1,p_2\ge0\end{subarray}}
\!\!\!\!\!\!\!\!\!\!
(-1)^r\!\si_r\cY_{p_1}(\al_i,\hb_1,q)\cY_{p_2-l}(\al_j,\hb_2,q)
\bigg|_{\begin{subarray}{l}\hb_1=\hb\\ \hb_2=-\hb\end{subarray}} =0.
\end{equation*}
In light of Lemma~\ref{Sii_lmm}, this implies that 
$$S_{ij}^{(2)}(\hb,-\hb,q)=S_{ji}^{(2)}(-\hb,\hb,q).$$
Using \e_ref{Siisums_e} and \e_ref{Ypexp_e}, we then obtain
\begin{equation}\label{Bidentity_e}\begin{split}
&\al_i^l\al_j^l\!\!\!\!\!\!\!\!\!
\sum_{\begin{subarray}{c}p_1+p_2+r=n-l\\ p_1,p_2\ge0\end{subarray}}
\!\!\!\!\!\!(-1)^r\si_r\!
\Bigg[\sum_{r'=0}^{p_1}\!\nD B_{p_1-1,1}^{(r')}(q)\al_i^{p_1-r'}\Bigg]\!\al_j^{p_2}\\
&\hspace{2in}
=\al_i^l\al_j^l\!\!\!\!\!\!\!
\sum_{\begin{subarray}{c}p_1+p_2+r=n-l\\ p_1,p_2\ge0\end{subarray}}
\!\!\!\!\!\!\!\!(-1)^r\si_r\!
\Bigg[\sum_{r'=0}^{p_1}\!\nD B_{p_1-1,1}^{(r')}(q)\al_j^{p_1-r'}\Bigg]\!\al_i^{p_2}.
\end{split}\end{equation}
Comparing the coefficients of $\al_i^{l+p_1}\al_j^{l+p_2}$ with 
$p_1,p_2\!\ge\!0$ and $p_1\!+\!p_2\!\le\!n\!-\!l$, we obtain
$$\nD\sum_{r=0}^p(-1)^r\si_rB_{p-1-r+p_1,1}^{(p-r)}(q)
=\nD\sum_{r=0}^p(-1)^r\si_rB_{p-1-r+p_2,1}^{(p-r)}(q)\,,$$
with $p$ as in the statement of the corollary.
Since $B_{p-1-r+p_2,1}^{(p-r)}(q)\in q\cdot\Q_{\al}[[q]]$, this proves the claim.
\end{proof}

\begin{crl}\label{Sii_crl}
For all $i,j\!=\!1,2,\ldots,n$,
\begin{equation*}
S_{ij}^{(2)}(\hb_1,\hb_2,q)=S_{ji}^{(2)}(\hb_2,\hb_1,q).
\end{equation*}
\end{crl}

\begin{proof}
By the same reasoning as in the proof of \e_ref{main_e2} on page~\pageref{Z_dfn}, 
it is sufficient to verify this identity modulo $\hb_1^{-1}$ and $\hb_2^{-1}$.
By \e_ref{Siisums_e} and \e_ref{Ypexp_e},
\begin{equation*}\begin{split}
S_{ij}^{(2)}(\hb_1,\hb_2,q)& \cong
\al_i^l\al_j^l\!\!\!\!\!\!\!\!\!\!\!
\sum_{\begin{subarray}{c}p_1+p_2+r=n-l\\ p_1,p_2\ge0\end{subarray}}
\!\!\!\!\!\!\!\!\!\!(-1)^r\si_r\!
\Bigg[\sum_{r'=0}^{p_1}\!\nD B_{p_1-1,1}^{(r')}(q)\al_i^{p_1-r'}\Bigg]\!\al_j^{p_2}\,,\\
S_{ji}^{(2)}(\hb_2,\hb_1,q)& \cong
\al_i^l\al_j^l\!\!\!\!\!\!\!\!\!\!\!
\sum_{\begin{subarray}{c}p_1+p_2+r=n-l\\ p_1,p_2\ge0\end{subarray}}
\!\!\!\!\!\!\!\!\!\!(-1)^r\si_r\!
\Bigg[\sum_{r'=0}^{p_1}\!\nD B_{p_1-1,1}^{(r')}(q)\al_j^{p_1-r'}\Bigg]\!\al_i^{p_2}\,.
\end{split}\end{equation*}
The two expressions on the right-hand sides 
are the same by \e_ref{Bidentity_e}.
\end{proof}

\begin{crl}\label{Sij_crl2}
For all $i,j\!=\!1,2,\ldots,n$,
\begin{equation*}\begin{split}
&\frac{\hb_1\hb_2}{\hb_1\!+\!\hb_2}
\bigg\{\frac{\al_i}{\hb_1}+\frac{\al_j}{\hb_2}+\nD\bigg\}
\sum_{\begin{subarray}{c}p_1+p_2+r=n-1\\p_1,p_2\ge0\end{subarray}}
\!\!\!\!\!\!\!\!\!\!
(-1)^r\!\si_r\cY_{p_1}(\al_i,\hb_1,q)\cY_{p_2-l}(\al_j,\hb_2,q)\\
&\qquad=\left\{\sum_{\begin{subarray}{c}p_1+p_2+r=n-l\\ p_1,p_2\ge0\end{subarray}}
-\sum_{\begin{subarray}{c}p_1+p_2+r=n-l\\ p_1,p_2\le-1\end{subarray}}\right\}
(-1)^r\si_rI_{p_1}(q)\cY_{p_1}(\al_i,\hb_1,q)\cY_{p_2}(\al_j,\hb_2,q)\\
&\qquad\qquad
+\sum_{\begin{subarray}{c}p_1+p_2+r+r'=n-1-l\\ p_1,p_2,r'\ge0\end{subarray}}
\hspace{-.3in}(-1)^r\si_r \nD B_{p_1+r',1}^{(r'+1)}(q)
\cY_{p_1}(\al_i,\hb_1,q)\cY_{p_2}(\al_j,\hb_2,q),
\end{split}\end{equation*}
where $I_{p_1}(q)\!\equiv\!1$ if $p_1\!<\!0$.
\end{crl}

\begin{proof}
By Lemma~\ref{Sii_lmm} and Corollary~\ref{Sii_crl},
\begin{equation*}\begin{split}
&\frac{\hb_1\hb_2}{\hb_1\!+\!\hb_2}
\bigg\{\frac{\al_i}{\hb_1}+\frac{\al_j}{\hb_2}+\nD\bigg\}
\sum_{\begin{subarray}{c}p_1+p_2+r=n-1\\p_1,p_2\ge0\end{subarray}}
\!\!\!\!\!\!\!\!\!\!
(-1)^r\!\si_r\cY_{p_1}(\al_i,\hb_1,q)\cY_{p_2-l}(\al_j,\hb_2,q)\\
&\qquad=\left\{\sum_{\begin{subarray}{c}p_1+p_2+r=n-l\\ p_1,p_2\ge0\end{subarray}}
-\sum_{\begin{subarray}{c}p_1+p_2+r=n-l\\ p_1,p_2\le-1\end{subarray}}\right\}
(-1)^r\si_r\cY_{p_1}(\al_i,\hb_1,q)\cY_{p_2}(\al_j,\hb_2,q)
+S_{ij}^{(2)}(\hb_1,\hb_2,q).
\end{split}\end{equation*}
By \e_ref{Siisums_e} and \e_ref{BvsI_e}, 
\begin{equation*}\begin{split}
S_{ij}^{(2)}(\hb_1,\hb_2,q)&=
\sum_{\begin{subarray}{c}p_1+p_2+r=n-l\\ p_1,p_2\ge0\end{subarray}}
\!\!\!\!\!\!\!\!\!
(-1)^r\si_r\big(I_{p_1}(q)\!-\!1\big)\cY_{p_1}(\al_i,\hb_1,q)\cY_{p_2}(\al_j,\hb_2,q)\\
&\qquad
+\sum_{\begin{subarray}{c}p_1+p_2+r=n-l\\ p_1,p_2\ge0\end{subarray}}
\!\!\!\!\!\!\!\!\!\!(-1)^r\si_r\!
\Bigg[\sum_{r'=1}^{p_1}\!\nD B_{p_1-1,1}^{(r')}(q)
\cY_{p_1-r'}(\al_i,\hb_1,q)\Bigg]\!\cY_{p_2}(\al_j,\hb_2,q).
\end{split}\end{equation*}
The claim is obtained by combining these two identities.
\end{proof}

\begin{lmm}\label{ctCrec_lmm}
The coefficients $\ctC_{p,s}^{(r)}$ of Theorem~\ref{main_thm} satisfy
$$I_p(q)\ctC_{p,s}^{(r)}=\nD\ctC_{p-1,s}^{(r)}+I_s(q)\ctC_{p-1,s-1}^{(r)}
\!-\!\sum_{r'=1}^{\min(p,r)}\!\!\!\nD B_{p-1,1}^{(r')}(q)\ctC_{p-r',s}^{(r-r')}(q)\,
\qquad\forall\,p\in\Z^+,$$
with $\ctC_{p-1,-1}^{(r)},\ctC_{p-1,p-r}^{(r)}\equiv0$.
\end{lmm}

\begin{proof} By the proof of Theorem~\ref{main_thm}, there is a unique choice
of the coefficients $\ctC_{p,s}^{(r)}$ so that \e_ref{cYcong_e2} holds.
Thus, similarly to the proof of Theorem~\ref{main_thm}, it is sufficient to show that 
$$\sum_{r=0}^{p}\sum_{s=0}^{p-r}
\bigg(\nD\ctC_{p-1,s}^{(r)}+I_s(q)\ctC_{p-1,s-1}^{(r)}
\!\!-\!\!\sum_{r'=1}^{\min(p,r)}\!\!\!\nD B_{p-1,1}^{(r')}(q)\ctC_{p-r',s}^{(r-r')}(q)   \bigg)
\hb^{p-r-s}\D^s\cY_0(x,\hb,q)$$
is congruent to $I_p(q)x^{l+p}$ modulo $\hb^{-1}$
whenever $x\!=\!\al_1,\ldots,\al_n$.
By the $p\!-\!1$ case of \e_ref{main_e0}, $\ctC^{(0)}_{p,s}\!=\!\de_{p,s}$, and the product rule, 
$$\big\{x\!+\!\hb\nD\big\}\cY_{p-1}(x,\hb,q)
=\sum_{r=0}^{p}\sum_{s=0}^{p-r}
\bigg(\nD\ctC_{p-1,s}^{(r)}+I_s(q)\ctC_{p-1,s-1}^{(r)}\bigg)\hb^{p-r-s}\D^s\cY_0(x,\hb,q).$$
By Lemma~\ref{B1_lmm},  \e_ref{BvsI_e}, and~\e_ref{main_e0}, 
the left-hand side of this identity is congruent modulo $\hb^{-1}$ to
\begin{equation*}\begin{split}
& 
I_p(q)x^{l+p}+\sum_{r'=1}^p\nD B_{p-1,1}^{(r')}(q)\cY_{p-r'}(x,\hb,q)\\
&\hspace{1in}\cong I_p(q)x^{l+p}+\sum_{r=1}^p\sum_{s=0}^{p-r}\sum_{r'=1}^{\min(p,r)}
\!\nD B_{p-1,1}^{(r')}(q)\ctC_{p-r',s}^{(r-r')}(q)\hb^{p-r-s}\D^s\cY_0(x,\hb,q)\,.
\end{split}\end{equation*}
This implies the desired congruence.
\end{proof}

\begin{crl}\label{ctCsym_crl}
For every $p\!=\!1,2,\ldots,n\!-\!l$,
\begin{equation*}\begin{split}
\ctC_{p,0}^{(p)}+\ctC_{n-l-1,n-l-1-p}^{(p)}
&=(-1)^p\si_p\bigg(I_0\!\!\!\!\!\prod_{s=0}^{n-l-1-p}\!\!\!\!\!I_s-1\bigg)\\
&\qquad-\frac{1}{I_p}\sum_{r=1}^{p-1}
\big(\nD B_{p-1,1}^{(r)}\ctC_{p-r,0}^{(p-r)}-\nD B_{n-l-1,1}^{(r)}\ctC_{n-l-r,n-l-p}^{(p-r)}\big)\\
&\qquad+\frac{1}{I_p}\sum_{r=1}^{p-1}
(-1)^r\si_r\big(\nD B_{p-1-r,1}^{(p-r)}\!-\!\nD B_{n-l-1-r,1}^{(p-r)}\!-\!I_0\ctC_{n-l-r,n-l-p}^{(p-r)}\big).
\end{split}\end{equation*}
\end{crl}

\begin{proof}
Using \e_ref{ctCperrel_e} and $\ctC_{p,s}^{(0)}\!=\!\de_{p,s}$ first and 
then Lemma~\ref{ctCrec_lmm}, we obtain
\begin{equation*}\begin{split}
I_{n-l}\ctC_{n-l,n-l-p}^{(p)}&=
I_{n-l}\Bigg((-1)^p\si_p \bigg(\prod_{s=0}^{n-l-p}\!\!\!I_s-1\bigg)
-\sum_{r=1}^{p-1}(-1)^r\si_r\ctC_{n-l-r,n-l-p}^{(p-r)}\Bigg)\\
&=I_{n-l-p}\ctC_{n-l-1,n-l-1-p}^{(p)}
-\sum_{r=1}^p\nD B_{n-l-1,1}^{(r)}\cdot\ctC_{n-l-r,n-l-p}^{(p-r)}.
\end{split}\end{equation*}
Since $I_{n-l-p}\!=\!I_p$ by \cite[(4.8)]{Po}, this gives 
\BE{Irec_e1}\begin{split}
I_p\ctC_{n-l-1,n-l-1-p}^{(p)}&=
(-1)^p\si_p \bigg(I_0\!\!\!\prod_{s=0}^{n-l-p}\!\!\!I_s\!-I_0\bigg)
+\nD B_{n-l-1,1}^{(p)}\\
&\hspace{.5in}+
\sum_{r=1}^{p-1}\big(\nD B_{n-l-1,1}^{(r)}\cdot\ctC_{n-l-r,n-l-p}^{(p-r)}
-(-1)^r\si_rI_0\ctC_{n-l-r,n-l-p}^{(p-r)}\big).
\end{split}\EE
Lemma~\ref{ctCrec_lmm} and $\ctC^{(0)}_{0,0}\!=\!1$  give
\BE{Irec_e2}
I_p\ctC_{p,0}^{(p)}=
-\nD B_{p-1,1}^{(p)}-\sum_{r=1}^{p-1}\nD B_{p-1,1}^{(r)}\cdot \ctC_{p-r,0}^{(p-r)}\,.\EE
Adding up \e_ref{Irec_e1} and \e_ref{Irec_e2} and 
using Corollary~\ref{C1_crl} with $(p_1,p_2)=(n\!-\!l\!-\!p,0)$, \e_ref{BvsI_e},
and $I_{n-l-p}\!=\!I_p$, we obtain the claim.
\end{proof}

\subsection{Proof of \e_ref{ctCperrel_e}}
\label{ctC_subs}

Since both sides of~\e_ref{ctCperrel_e} are symmetric polynomials in 
$\al_1,\ldots,\al_n$ of degree $n\!-\!l$, it is sufficient to verify
\e_ref{ctCperrel_e} with $\al_p\!=\!0$ for all $p\!>\!n\!-\!l$.
Thus, the $n\!-\!l$ statements in~\e_ref{ctCperrel_e} are equivalent~to
\BE{ctCperrel_e2}
\sum_{p=1}^{n-l}\Bigg(\sum_{r=0}^p(-1)^r\si_r
\frac{\ctC_{n-l-r,n-l-p}^{(p-r)}(q)}{\prod\limits_{s=0}^{n-l-p}\!\!\!I_s(q)}-
(-1)^p\si_p\Bigg)x^{n-l-p}=0
\quad
\begin{aligned}
&\forall~x=\al_1,\al_2,\ldots,\al_{n-l}\\
&{}\quad \al_p=0~\forall\,p\!>\!n\!-\!l.
\end{aligned}\EE
Let $\nD= q\frac{\nd}{\nd q}$ as before.

\begin{proof}[Proof of \e_ref{ctCperrel_e2}]
Let $\nD_{x,\hb}=x+\hb\nD$; so, 
$$\D^p\cY_0=x^l\bigg\{\frac{\nD_{x,\hb}}{I_p}\bigg\}\bigg\{\frac{\nD_{x,\hb}}{I_{p-1}}\bigg\}
\ldots\bigg\{\frac{\nD_{x,\hb}}{I_1}\bigg\}
\bigg(\frac{\cY}{I_0}\bigg).$$
Throughout this argument, we assume that the conditions on $x$ and $\al$ in 
\e_ref{ctCperrel_e2} are satisfied. Thus,
\begin{alignat}{1}
\label{alg_e}
\bigg\{\prod_{k=1}^{n-l}(\nD_{x,\hb}\!-\!\al_k)-
q\lr{\a}\prod_{k=1}^l\prod_{r=1}^{a_k-1}\big(a_k\nD_{x,\hb}+r\hb\big)\bigg\}
\cY(x,\hb,q)=0,\\
\label{geom_e}
\cY_{n-l}(x,\hb,q)-\si_1\cY_{n-l-1}(x,\hb,q)+\ldots+(-1)^{n-l}\si_{n-l}\cY_0(x,\hb,q)=0;
\end{alignat}
the first identity follows directly from~\e_ref{Ydfn_e}, while the
second  from the $p\!=\!n\!-\!l$ case of \e_ref{symmrel_e}
and~\e_ref{main_e1}.
Subtracting~\e_ref{alg_e} from $1/x^l$ times~\e_ref{geom_e}
and using~\e_ref{main_e0} and the product rule, we~obtain
\BE{general_e}
\sum_{s=0}^{n-l} \sum_{p=s}^{n-l} \cA_{s;p}(q)\hb^s\nD_{x,\hb}^{n-l-p}\cY(x,\hb,q)=0\EE
for some $\cA_{s;p}(q)\!\in\!\Q[\al][[q]]$ with
\BE{A0_e}\begin{split}
\cA_{0;p}(q)&=\sum_{r=0}^p(-1)^r\si_r
\frac{\ctC_{n-l-r,n-l-p}^{(p-r)}(q)}{\prod\limits_{s=0}^{n-l-p}\!\!\!I_s(q)}
-(-1)^p\si_p+\de_{p,0}\a^{\a}q\,,\\
\cA_{n-l;n-l}(q)&=\bigg\{\frac{\nD}{I_{n-l}(q)}\bigg\}\bigg\{\frac{\nD}{I_{n-l-1}(q)}\bigg\}\ldots
\bigg\{\frac{\nD}{I_1(q)}\bigg\}\bigg(\frac{1}{I_0(q)}\bigg)+\a!\cdot q.\\
\end{split}\EE
By \e_ref{Ydfn_e}, 
\BE{nDxhbY_e}\nD_{x,\hb}^p\cY(x,\hb,q)
=x^p+\Bigg(\prod\limits_{k=1}^l\prod_{r=1}^{a_k-1}(a_kx+r\hb)\bigg)\F_p(x,\hb,q)\EE
for some $\F_p\!\in\!q\cdot\Q_{\al}(x,\hb)[[q]]$ which is regular at $\hb\!=\!-a_kx/r$
for all $r\!\in\![a_k\!-\!1]$.
Since $\cA_{n-l;n-l}\!=\!0$ by~\e_ref{A0_e} and Lemma~\ref{Iprod_lmm} below, 
it follows from~\e_ref{general_e} and~\e_ref{nDxhbY_e} that 
\BE{deg0van_e}\sum_{s=0}^{n-l-1} \sum_{p=s}^{n-l}\cA_{s;p}(q)x^{n-l-p}\hb^s=0
\qquad\forall\, \hb\!=\!-a_kx/r,\,r\!\in\![a_k\!-\!1],~k\!\in\![l].\EE
Since the left-hand side of~\e_ref{deg0van_e} is a polynomial in $\hb$ of
degree $n\!-\!l\!-\!1$ with $n\!-\!l$ zeros 
(counted with multiplicity)\footnote{This is one of the two places in the proof where 
the Calabi-Yau condition, $|\a|\!=\!n$, is used; the other place is the $p\!=\!0$ case
of \e_ref{A0_e}.}, 
it vanishes identically.
On the other hand,
since $\cA_{0;0}\!=\!0$ by~\e_ref{A0_e}, $\ctC^{(0)}_{p,s}\!=\!\de_{p,s}$, and~\cite[(4.9)]{Po}, 
the coefficient of~$\hb^0$ on the left-hand side of~\e_ref{deg0van_e}
is precisely the left-hand side of~\e_ref{ctCperrel_e2}. 
\end{proof}

\begin{lmm}\label{Iprod_lmm}
If $|\a|\!=\!n$, then
$$\bigg\{\frac{\nD}{I_{n-l}(q)}\bigg\}\bigg\{\frac{\nD}{I_{n-l-1}(q)}\bigg\}\ldots
\bigg\{\frac{\nD}{I_1(q)}\bigg\}\bigg(\frac{1}{I_0(q)}\bigg)=-\a!\cdot q.$$
\end{lmm}

\begin{proof} Let $\nD_w= w+\nD$; so, 
$$w^p\bM^pF(w,q)=I_p(q)\bigg\{\frac{\nD_w}{I_p}\bigg\}\bigg\{\frac{\nD_w}{I_{p-1}}\bigg\}
\ldots\bigg\{\frac{\nD_w}{I_1}\bigg\}
\bigg(\frac{F}{I_0}\bigg).$$
The series $F(w,q)$ of \e_ref{tiFdfn_e} satisfies the differential equations
\begin{alignat}{1}
\label{simple_e}
\bigg\{\nD_w^{n-l}-q\lr{\a}\prod\limits_{k=1}^l
\prod_{r=1}^{a_k-1}(a_k\nD_w+r)\bigg\}F&=w^{n-l},\\
\label{ZaZ_e}
\bM^{n-l}F(w,q)&=I_{n-l}(q);
\end{alignat}
the first identity follows directly from \e_ref{tiFdfn_e}, while
the second is proved in \cite[Section~4.1]{Po}.
Subtracting~\e_ref{simple_e} from $w^{n-l}/I_{n-l}(q)$ times~\e_ref{ZaZ_e}
and using the product rule, we obtain
\BE{general_e2}
\sum_{p=0}^{n-l}A_p(q)\nD_w^{n-l-p}F(w,q)=0\EE
for some $A_p(q)\!\in\!\Q[[q]]$ with
\BE{A0_e2}\begin{split}
A_0(q)&=\frac{1}{\prod\limits_{s=0}^{n-l}I_s(q)}-1+\a^{\a}q\,, \\
A_{n-l}(q)&=
\bigg\{\frac{\nD}{I_{n-l}(q)}\bigg\}\bigg\{\frac{\nD}{I_{n-l-1}(q)}\bigg\}\ldots
\bigg\{\frac{\nD}{I_1(q)}\bigg\}\bigg(\frac{1}{I_0(q)}\bigg)
+\a!\cdot q.
\end{split}\EE
By~\e_ref{tiFdfn_e}, 
\BE{nDwF_e}\nD_w^pF(w,q)
=w^p+\Bigg(\prod\limits_{k=1}^l\prod_{r=1}^{a_k-1}(a_kw+r)\bigg)H_p(w,q)\EE
for some $H_p\!\in\!q\cdot\Q(w)[[q]]$ which is regular at $w=-\frac{r}{a_k}$
for all $r\!\in\![a_k\!-\!1]$.
Since $A_0\!=\!0$ by~\e_ref{A0_e2} and \cite[(4.9)]{Po}, 
it follows from~\e_ref{general_e2} and~\e_ref{nDwF_e} that
\BE{deg0van_e2}
\sum_{p=1}^{n-l} A_p(q)w^{n-l-p}=0
\qquad\forall\, w\!=\!-r/a_k, \,r\!\in\![a_k\!-\!1],~k\!\in\![l].\EE
Since the left-hand side of~\e_ref{deg0van_e2} is a polynomial in $w$ of
degree $n\!-\!l\!-\!1$ with $n\!-\!l$ zeros (counted with multiplicity), it vanishes identically;
in particular, $A_{n-l}(q)\!=\!0$. 
The claim now follows from the second identity in~\e_ref{A0_e2}.
\end{proof}

\section{Annulus GW-Invariants of CY CI Threefolds}
\label{annulus_sec}

It remains to establish Theorems~\ref{annulus_thm} and~\ref{klein_thm}
concerning the annulus and Klein bottle invariants of $(X_{\a},\Om)$,
where $\Om\!:\Pn\!\lra\!\Pn$ is the anti-symplectic involution given 
by~\e_ref{Omdfn_e} and $X_{\a}\!\subset\!\Pn$ is 
a smooth CY CI threefold of multi-degree $\a$ such that $\Om(X_{\a})\!=\!X_{\a}$.
For the remainder of the paper, $X_{\a}$ will denote such a threefold.
Furthermore, all statements involving the $\T^n$-weights $\al\!=\!(\al_1,\ldots,\al_n)$
will be taken to mean that they hold when restricted to
the subtorus $\T^m$ of Section~\ref{OmEquiv_subs}; see \e_ref{sp_weights_e}.

\subsection{Description and graph-sum definition}
\label{setup_subs}

A \sf{degree~$d$ annulus doubled map} to $\Pn$ is an $\Om$-invariant map $(\ti{C},\tau,\ti{f})$
to~$\Pn$ such that $\tau$ has two fixed components.
An example of the restriction of~$\tau$ to the principal component~$\ti{C}_0$ of~$\ti{C}$
in this case is given~by
\BE{C0anndfn_e}\begin{split}
&\ti{C}_0=(\Z_{2k}\!\times\!\P^1/\sim), 
~~\hbox{where}~~(s,[0,1])\sim(s\!+\!1,[1,0])~~\forall\,s\!\in\!\Z_{2k}\,,\\
&\hspace{.7in}
\tau\!:\ti{C}_0\lra\ti{C}_0,\quad\tau\big(s,[z_1,z_2]\big)=\big(-s,[\bar{z}_2,\bar{z}_1]\big).
\end{split}\EE
Since $\ti{C}/\tau$ can be naturally identified with two subspaces of $\ti{C}$,
a degree~$d$ annulus doubled map $\ti{f}:\ti{C}\!\lra\!\Pn$ 
corresponds to a pair of maps $f_1,f_2\!:\ti{C}/\tau\!\lra\!\Pn$
that differ by the composition with~$\Om$; 
these maps  will be referred to as \sf{degree~$d$ annulus maps}.
This is the analogue of the definition for disk maps; see \cite[Section~1.3.3]{PSW}. 
The moduli space of stable degree~$d$ doubled annulus maps to a CY CI threefold $X_{\a}\!\subset\!\Pn$,
$$\ov\M_1(X_{\a},d)^{\Om}\equiv \ov\M_{1,2,0}(X_{\a},\Om,d),$$
is expected to have a well-defined virtual degree~$A_d$, 
with $A_d\!=\!0$ if $d$ is odd.\footnote{This is consistent with the vector bundle 
$\V_{1,2d}^{\Om}$ in \e_ref{A2dsum_e} having a direct summand of odd real rank;
thus, its Euler class is zero.}
A combinatorial description of~$A_d$ for $d$ even, due to~\cite{W1}, is recalled 
and motivated below.

If $1\!\le\!i\!\le\!2m$ and $\ga\!\in\!\Z^+$,  let 
$$\ti{f}_{i,\ga}\!:(\P^1,0) \lra \big(\ell_{i\bar{i}},P_i\big)\subset(\Pn,P_i)$$
denote the equivalence class of the degree~$\ga$ cover $\P^1\!\lra\!\ell_{i\bar{i}}$ 
branched over $P_i$ and $P_{\bar{i}}$ only and taking 
the marked point $0\!\in\!\P^1$ to~$P_i$.
Denote by 
$$\ov\M_1(\Pn,d)^{\Om;\T^m}\subset\ov\M_1(\Pn,d)^{\Om} $$
the subspace of the $\T^m$-fixed equivalence classes of stable doubled annulus maps.

Based on the disk case fully treated in~\cite{So}, \cite{PSW}, and~\cite{Sh}, 
one would expect that 
\BE{A2dsum_e} 
A_{2d}=-2\sum_{F\subset\ov\M_1(\Pn,2d)^{\Om;\T^m}}
\int_F\frac{\E(\V_{1,2d}^{\Om})|_F}{\E(NF)}\,,\EE
where the sum is taken over the connected components $F$ of 
$\ov\M_1(\Pn,2d)^{\Om;\T^m}$, $\E(NF)$ is the equivariant Euler
class of the normal ``bundle"~$NF$ of~$F$ in $\ov\M_1(\Pn,2d)^{\Om}$,
and $\E(\V_{1,2d}^{\Om})$ is the equivariant Euler class of the obstruction ``bundle"
$$\V_{1,2d}^{\Om}\equiv \ov\M_1(\cL,2d)^{\Om}\lra\ov\M_1(\Pn,2d)^{\Om}\,,$$
with $\cL$ given by \e_ref{cLdfn_e}.
The factor of~2 in \e_ref{A2dsum_e} is due to the fact that a doubled annulus map
corresponds to 2~annulus maps, while the choice of sign \cite[(3.15),(3.23)]{W1}
is due to delicate physical considerations. 
While $\ov\M_1(\Pn,2d)^{\Om}$ and $\V_{1,2d}^{\Om}$
may be singular near doubled maps $(\ti{C},\tau,\ti{f})$
with contracted principle component~$\ti{C}_0$ or a $\tau$-fixed node on~$\ti{C}$,
the restriction of $\V_{1,2d}^{\Om}$ to the fixed loci~$F$ consisting of such maps
contains a topologically trivial subbundle of weight~0 and 
thus does not contribute to~\e_ref{A2dsum_e}.
For the same reason, there is no genus~0 correction to~$A_{2d}$,
which in the closed genus~1 case arises from the stable maps
with contracted principle component; 
see \cite[(1.5)]{LiZ} and \cite[Theorem~1.1]{g1comp2}.\footnote{Such a correction
would be a multiple of the degree~$2d$ disk invariant, which vanishes.}

Thus, by Corollary~\ref{fixedmap_crl}, all nonzero contributions to~\e_ref{A2dsum_e}
come from fixed loci~$F$ consisting of stable doubled maps $(\ti{C},\tau,\ti{f})$
such that the principal component $\ti{C}_0\!\subset\!\ti{C}$ is a circle of 
spheres as in~\e_ref{C0anndfn_e} and $\ti{f}|_{\ti{C}_0}$ is not constant.
In this case, $\ti{C}_0$ contains two components, $\ti{C}_{0,1}$ and $\ti{C}_{0,2}$,
each of which is mapped by~$\tau$ into itself;
so, $\ti{C}_{0,1}/\tau$ and $\ti{C}_{0,2}/\tau$ are disks.
Similarly to the disk case treated in~\cite{PSW} and~\cite{Sh}, the fixed-point loci~$F$ 
containing maps $(\ti{C},\tau,\ti{f})$ with $\ti{f}|_{\ti{C}_{0,1}}$
or $\ti{f}|_{\ti{C}_{0,2}}$
of even degree should not contribute to~\e_ref{A2dsum_e}.
Thus, the fixed-point loci~$F$ contributing to~\e_ref{A2dsum_e}
consist of stable maps  $(\ti{C},\tau,\ti{f})$ such that 
$\ti{f}|_{\ti{C}_{0,r}}$ is the unique cover of a line $\ell_{i_r\bar{i}_r}$ 
with $1\!\le\!i_r\!\le\!2m$ of odd degree~$\ga_r$ 
branched at the nodes of~$\ti{C}_{0,r}$ and over~$P_{i_r}$ and~$P_{\bar{i}_r}$ only,
for each $r\!=\!1,2$.

Any such annulus $\T^m$-fixed doubled map $(\ti{C},\tau,\ti{f})$
corresponds to a map from the quotient $\ti{C}/\tau$, 
which is either a wedge of two disks, $C_{0,1}$ and~$C_{0,2}$,
or a tree of spheres with two disks attached.
As an equivalence class of maps from a disk with 1~marked point (the node), 
$\ti{f}|_{C_{0,r}}$ equals to the restriction of the map~$\ti{f}_{i_r,\ga_r}$ 
with $1\!\le\!i_r\!\le\!2m$ and $\ga_r\!\in\!\Z^+$ odd
to a disk containing~$0$, for each $r\!=\!1,2$.
If $\ti{C}/\tau$ contains a nonempty tree of spheres~$C'$,
$\ti{f}|_{C'}$ corresponds to a $\T^m$-fixed genus~0 stable map with two marked points
(which are mapped to~$P_{i_1}$ and $P_{i_2}$, respectively),
i.e.~an element of $\ov\M_{0,2}(\Pn,r)^{\T^m}$.
Thus, a fixed locus~$F$ contributing to~\e_ref{A2dsum_e} 
corresponds to a tuple $(i,\ga_1,\ga_2)$ in the first case
and a tuple $(i_1,\ga_1;i_2,\ga_2;F')$ in the second case, where
\begin{itemize}
\item $1\le i,i_1,i_2\le 2m$ describe the fixed points $P_k\!\in\!\Pn$ to
which the nodes of the two disks are mapped by~$\ti{f}$;
\item $\ga_1,\ga_2\!\in\!\Z^+$ are odd and describe the degrees of
the restrictions of $\ti{f}$ to the doubled disks, $\ti{C}_{0,1}$ and $\ti{C}_{0,2}$,
and in the first case $\ga_1\!+\!\ga_2\!=\!2d$;
\item $F'\subset\ov\M_{0,2}(\Pn,r)^{\T^m}$ is a component of the fixed locus
and $\ga_1\!+\!\ga_2\!+\!2r\!=\!2d$ in the second case.
\end{itemize}
It is convenient to view the first case, when the tree of spheres is empty,
as the $r\!=\!0$ extension of the second case, with $i_1\!=\!i_2\!=\!i$.
The two cases are illustrated in Figure~\ref{annulus_fig}.

\begin{figure}
\begin{pspicture}(-1,-2.1)(10,1.8)
\psset{unit=.3cm}
\pscircle*(5,4){.2}\pscircle*(5,-4){.2}
\psarc[linewidth=.07](9,0){5.66}{135}{180}\psarc[linewidth=.07](1,0){5.66}{0}{45}
\psarc[linewidth=.07,linestyle=dashed](9,0){5.66}{180}{225}
\psarc[linewidth=.07,linestyle=dashed](1,0){5.66}{-45}{0}
\rput(5,5){$\mathbf i$}\rput(5,-5){$\mathbf{\bar i}$}
\rput(2.5,0){$\ga_1$}\rput(7.5,0){$\ga_2$}
\rput(5,-7){\smsize{$r=0$ case: $(i,\ga_1,\ga_2)$}}
\rput(35,2){\begin{small}\begin{tabular}{c}
$F'\subset\ev_1^{-1}(P_{i_1})\cap\ev_2^{-1}(P_{i_2})\cap\ov\M_{0,2}(\Pn,r)^{\T^m}$\\
topological component\end{tabular}\end{small}}
\psline[linewidth=.07](22,4)(48,4)\psline[linewidth=.07](22,0)(48,0)
\psline[linewidth=.07](22,4)(22,0)\psline[linewidth=.07](48,4)(48,0)
\pscircle*(30,0){.2}\pscircle*(40,0){.2}
\psline[linewidth=.07](30,0)(30,-2)\psline[linewidth=.07](40,0)(40,-2)
\psline[linewidth=.07,linestyle=dashed](30,-5)(30,-2)
\psline[linewidth=.07,linestyle=dashed](40,-5)(40,-2)
\rput(29.5,-.7){$\mathbf i$}\rput(40.5,-.7){$\mathbf j$}
\rput(29.2,-2.5){$\ga_1$}\rput(40.8,-2.5){$\ga_2$}
\rput(35,-7){\smsize{$r>0$ cases: $(i_1,\ga_1;i_2,\ga_2;F')$}}
\end{pspicture}
\caption{$\T^m$-fixed loci contributing to $A_{\ga_1+\ga_2+2r}$, depicted as half-graphs;
half-dotted edges correspond to the two doubled disks.}
\label{annulus_fig}
\end{figure}


Denote by $NF'\!\lra\!F'$ the normal bundle of $F'$ in $\ov\M_{0,2}(\Pn,r)$
and by $T_{0,\ga}^{\Om}|_{\ti{f}_{i,\ga}}$ the tangent space at $\ti{f}_{i,\ga}$
of the space of maps from disks with~$1$ marked point.
Doubled annulus maps close to one of the above fixed loci $F$ can be obtained
by deforming the disk components of a fixed map $(\ti{C},\tau,\ti{f})$
and its component lying in~$F'$, while keeping them joined at the two nodes;
the last deformation takes place in $\ov\M_{0,2}(\Pn,r)$.
Since the disk nodes can also be smoothed out, 
\BE{NFsplit_e}\begin{split}
\E(NF)&=\frac{\E(T_{0,\ga_1}^{\Om}|_{\ti{f}_{i_1,\ga_1}})\cdot\E(NF')
\cdot\E(T_{0,\ga_2}^{\Om}|_{\ti{f}_{i_2,\ga_2}})}
{\E(T_{P_{i_1}}\Pn)\E(T_{P_{i_2}}\Pn)}
\cdot(\hb_1\!-\!\psi_1)\cdot(\hb_2\!-\!\psi_2)\\
&=\frac{\E(T_{0,\ga_1}^{\Om}|_{\ti{f}_{i_1,\ga_1}})
\cdot\E(T_{0,\ga_2}^{\Om}|_{\ti{f}_{i_2,\ga_2}})\cdot\E(NF')}
{\phi_{i_1}|_{P_{i_1}}\phi_{i_2}|_{P_{i_2}}}\cdot(\hb_1\!-\!\psi_1)\cdot(\hb_2\!-\!\psi_2)\,,
\end{split}\EE
where $\hb_r$ is the equivariant Euler class of 
the tangent space at the node of the $r$-th disk,
$\psi_1,\psi_2$ denote the restrictions of the $\psi$-classes 
on  $\ov\M_{0,2}(\Pn,r)$ to $F'$ if $r\!>\!0$, and 
\BE{NFr0_e}
\frac{\E(NF')}{\phi_{i_2}|_{P_{i_2}}}(\hb_1\!-\!\psi_1)\cdot(\hb_2\!-\!\psi_2)
\equiv \hb_1\!+\!\hb_2 \qquad\hbox{if}~~r\!=\!0.\EE
The restriction of $\V_{1,d}^{\Om}$ to $F$ is the subbundle 
of the direct sum of the corresponding bundles over the components~$F$
consisting of sections that agree at the nodes of the disks.
Thus,
\BE{Vsplit_e}\begin{split}
\E(\V_{1,2d}^{\Om})\big|_F
&=\frac{\E(\V_{0,\ga_1}^{\Om}|_{\ti{f}_{i_1,\ga_1}})\cdot\E(\V_{0,r})|_{F'}
\cdot\E(\V_{0,\ga_2}^{\Om}|_{\ti{f}_{i_2,\ga_2}})}
{\E(\cL)|_{P_{i_1}}\E(\cL)|_{P_{i_2}}}\\
&=\frac{\E(\V_{0,\ga_1}^{\Om}|_{\ti{f}_{i_1,\ga_1}})\cdot
\E(\V_{0,\ga_2}^{\Om}|_{\ti{f}_{i_2,\ga_2}})\cdot\E(\V_{0,r})|_{F'}}
{\lr{\a}^2\al_{i_1}^l\al_{i_2}^l}\,,
\end{split}\EE
where $\V_{0,r}$ denotes the bundle $\V_{\a}\!\lra\!\ov\M_{0,2}(\Pn,r)$  
of Section~\ref{mainthm_sec} if $r\!>\!0$, $\E(\V_{0,0})|_{F'}\!\equiv\!\lr{\a}\al_{i_1}^l$,
and $\V_{0,\ga_r}^{\Om}|_{f_{i_r,\ga_r}}$ is the vector space of maps
$(\ti{C}_{0,r},\tau)\!\lra\!(\cL,\Om)$ lifting~$\ti{f}_{i_r,\ga_r}$.

By \e_ref{NFsplit_e}-\e_ref{Vsplit_e},
\BE{restr_e}
\int_F\frac{\E(\V_{1,2d}^{\Om})|_F}{\E(NF)}
=\frac{1}{\lr{\a}^2\al_{i_1}^l\al_{i_2}^l}\cdot
\frac{\cD_{i_1,\ga_1}}{\hb_1}\cdot\frac{\cD_{i_2,\ga_2}}{\hb_2}\cdot
\begin{cases}
\frac{\lr{\a}\al_{i_1}^l}{\hb_1+\hb_2}\prod\limits_{k\neq i_1}(\al_{i_2}\!-\!\al_k),
&\hbox{if}~r\!=\!0,\\
\int_{F'}\left.\frac{\E(\V_{0,r})(\ev_1^*\phi_{i_1})(\ev_2^*\phi_{i_2})}
{(\hb_1-\psi_1)(\hb_2-\psi_2)}\right|_{F'}\frac{1}{\E(NF')},
&\hbox{if}~r\!>\!0,\end{cases}\EE
where $\cD_{i,\ga}$ is the contribution of the half-edge disk map $\ti{f}_{i,\ga}$
without a marked point to the degree~$\ga$ disk invariant.
This contribution is computed in \cite[Lemma~6]{PSW} and in~\cite{Sh} to be 
\BE{disk-edge_e}\begin{split}
\cD_{i,\ga}
&= (-1)^{\frac{\ga-1}{2}}
\frac{\prod\limits_{k=1}^l(a_k\ga)!!}
{2^{\ga-1}\ga^{\frac{(n-2)\ga+l+4}{2}}\ga!}\frac{\al_i^{\frac{(n-2)\ga+l+2}{2}}}
{\prod\limits_{\begin{subarray}{c}1\le k\le n\\ k\neq i,\bar{i}\end{subarray}}
\prod\limits_{s=0}^{(\ga-1)/2}
\left(\frac{\ga-2s}{\ga}\al_i-\al_k\right)}\\ 
&=2\frac{\prod\limits_{k=1}^l(a_k\ga)!!}
{\ga\underset{(k,s)\neq(i,\ga)}{\prod\limits_{k=1}^{n}
\prod\limits_{\begin{subarray}{c}1\le s\le\ga\\ s\,\tn{odd}\end{subarray}}}
\left(\frac{s}{\ga}\al_i-\al_k\right)}
\bigg(\frac{\al_i}{\ga}\bigg)^{\frac{n\ga+l+2}{2}}
  \hspace{.5in}\forall~\ga~\hbox{odd}
\end{split}\EE
whenever every component $a_k$ of $\a$ is odd;
otherwise, $\cD_{i,\ga}\!=\!0$.

By the classical Localization Theorem~\cite{ABo}, if $r\!>\!0$
\begin{equation}\label{midsum_e}
\sum_{F'}
\int_{F'}\left.\frac{\E(\V_{0,r})(\ev_1^*\phi_{i_1})(\ev_2^*\phi_{i_2})}
{(\hb_1-\psi_1)(\hb_2-\psi_2)}\right|_{F'}\frac{1}{\E(NF')}
=\int_{\ov\M_{0,2}(\Pn,r)}
\frac{\E(\V_{\a})(\ev_1^*\phi_{i_1})(\ev_2^*\phi_{i_2})}
{(\hb_1-\psi_1)(\hb_2-\psi_2)}\,,
\end{equation}
where the sum is taken 
over all fixed loci $F'\!\subset\!\ov\M_{0,2}(\Pn,r)$.
The $r\!=\!0$ case of \e_ref{restr_e} gives the coefficient of~$Q^0$ 
in the power series~$\cZ$ in \e_ref{Z_dfn}.
Since 
$$\hb_r=\frac{\al_{i_r}-\al_{\bar{i}_r}}{\ga_r}=\frac{2\al_{i_r}}{\ga_r}
\qquad\forall\,r\!=\!1,2$$
and we have made 4 choices in describing each fixed locus $F$
(choosing an embedding $\ti{C}/\tau\!\subset\!\ti{C}$ and ordering the resulting disks),
plugging in~\e_ref{restr_e} and~\e_ref{midsum_e} into~\e_ref{A2dsum_e} 
leads to the following compact reformulation 
of the combinatorial description of~$A_{2d}$ given in~\cite{W1}.  

\begin{dfn}\label{A2d_dfn}
The degree $2d$ annulus invariant $A_{2d}$ of $(X_{\a},\Om)$ is given by
\BE{A2ddfn_e}
\sum_{d=1}^{\i}Q^dA_{2d} = -\frac{1}{2\lr{\a}^2}
\sum_{\begin{subarray}{c}1\le i,j\le 2m\\ \ga,\de\in\Z^+\,\tn{odd}\end{subarray}}
Q^{\frac{\ga+\de}{2}}\frac{\cD_{i,\ga}\cD_{j,\de}}{\al_i^l\al_j^l}
\frac{1}{\hb_1\hb_2}
\cZ\big(\al_i,\al_j,\hb_1,\hb_2,Q\big)
\left|_{\begin{subarray}{c}\hb_1=\frac{2\al_i}{\ga}\\ 
\hb_2=\frac{2\al_j}{\de}\end{subarray}}\,.\right.\EE
\end{dfn}

It is shown in \cite{Sh}, as well as in \cite{PSW} in the $\a\!=\!(5)$ case, that 
\BE{diskdfn_e}
Z_{disk}(Q)=\sum_{\begin{subarray}{c}1\le i\le 2m\\
\ga\in\Z^+\,\tn{odd}\end{subarray}}\!\!\!
Q^{\frac{\ga}{2}}\frac{\cD_{i,\ga}}{\al_i^l}\cZ_0(\al_i,\hb,Q)
\left|_{\hb=\frac{2\al_i}{\ga}}\right.\,,\EE
where $Z_{disk}(Q)$ is the disk potential as in~\e_ref{Zdiskdfn_e}.
Definition~\ref{A2d_dfn} for the annulus invariants is analogous
to this property for the disk invariants.
Since $\si_1\!=\!0$ by \e_ref{sp_weights_e}, 
\e_ref{diskdfn_e} and \e_ref{main_e1} give
\BE{diskprp_e}
 Z_{disk}(Q)=\sum_{\begin{subarray}{c}1\leq i\leq 2m\\
\ga\in\Z^+\,\tn{odd}\end{subarray}}\!\!\!
q^{\frac{\al_i}{\hb}}
\frac{\cD_{i,\ga}}{\al_i^l}\cY_0(\al_i,\hb,q)\left|_{\hb=\frac{2\al_i}{\ga}}\right.\,,\EE
with $Q$ and $q$ related by the mirror map~\e_ref{mirmap_e}.
This formula is used in \cite{PSW} and \cite{Sh} to obtain \e_ref{Fdisk_e}
whenever all components of $\a$ are odd;
in the other cases, $Z_{disk}(Q)=0$ by \e_ref{diskprp_e}, since $\cD_{i,\ga}\!=\!0$.

\subsection{Proof of Theorem~\ref{annulus_thm}}
\label{annpf_sec}

If any component $a_k$ of $\a$ is even, 
Theorem~\ref{annulus_thm} is meaningless, simply stating that $0\!=\!0$.
In the remaining cases,  Theorem~\ref{annulus_thm} is proved in this section.

By \e_ref{A2ddfn_e}, \e_ref{main_e2}, \e_ref{main_e1}, and \e_ref{Z_dfn}, 
\begin{equation}\label{annulus_initial_comp2}\begin{split} 
\sum_{d=1}^{\i}Q^dA_{2d}= -& \left.\frac{1}{2\lr{\a}}
\sum_{\begin{subarray}{c}1\le i,j\le 2m\\ \ga,\de\in\Z^+\,\tn{ odd}\end{subarray}}
\right\{ \frac{1}{\left(\hb_1+\hb_2\right)\hb_1\hb_2}
\frac{\cD_{i,\ga}\cD_{j,\de}}{\al_i^l\al_j^l}\\
&\hspace{.8in}\times 
q^{\frac{\al_i}{\hb_1}+\frac{\al_j}{\hb_2}}\!\!\!\!\!\!\!\!\!\!\!
\left.\sum_{\begin{subarray}{c}p_1+p_2+r=n-1\\p_1,p_2\ge 0\end{subarray}}
\!\!\!\!\!\!\!\!\!\!\!(-1)^r
\si_r\cY_{p_1}(\al_i,\hb_1,q)\cY_{p_2-l}(\al_j,\hb_2,q)
\right\}\left|_{\begin{subarray}{c}\hb_1=\frac{2\al_i}{\ga}\\\hb_2=\frac{2\al_j}{\de}\end{subarray}}\right.
\end{split}\end{equation}
with $\cY_p(x,\hb,q)$ given by \e_ref{main_e0}, \e_ref{Crec_e}, and~\e_ref{tiCrec_e}.
Since $Q\frac{\nd}{\nd Q}=\frac{q}{I_1(q)}\frac{\nd}{\nd q}$,
by Corollary~\ref{Sij_crl2}
\begin{equation*}\label{tiZexp_e}\begin{split}
&\frac{\hb_1\hb_2}{\hb_1\!+\!\hb_2}
Q\frac{\nd}{\nd Q}\left\{q^{\frac{\al_i}{\hb_1}+\frac{\al_j}{\hb_2}}
\!\!\!\!\!\!
\sum_{\begin{subarray}{c}p_1+p_2+r=n-1\\p_1,p_2\ge0\end{subarray}}
\!\!\!\!\!\!\!\!\!\!\!(-1)^r\si_r\cY_{p_1}(\al_i,\hb_1,q)\cY_{p_2-l}(\al_j,\hb_2,q)\right\}\\
&\quad=\frac{1}{I_1(q)}q^{\frac{\al_i}{\hb_1}+\frac{\al_j}{\hb_2}}
\left[I_2(q)\cY_2(\al_i,\hb_1,q)\cY_2(\al_j,\hb_2,q)
+\!\!\!\!\sum_{\begin{subarray}{c}p_1,p_2\ge -l\\ p_1<2\,\tn{or}\,p_2<2\end{subarray}}
\!\!\!\!\!\!\!\!\!\!
A_{p_1p_2}(q)\cY_{p_1}(\al_i,\hb_1,q)\cY_{p_2}(\al_j,\hb_2,q)\right],
\end{split}\end{equation*}
for some $A_{p_1p_2}(q)\!\in\!\Q_{\al}[[q]]$.
By the $-l\!\le\!p\!<\!b\!=\!2$ case of Corollary~\ref{res_crl} below,
$$\sum_{\begin{subarray}{c}p_1,p_2\ge -l\\ p_1<2\,\tn{or}\,p_2<2\end{subarray}}
\!\!\!\!\!\!\!\!\!\!
A_{p_1p_2}(q)\!\!\!\!
\sum_{\begin{subarray}{c}1\le i,j\le 2m\\ \ga,\de\in\Z^+\,\tn{odd}\end{subarray}}
\frac{1}{\hb_1^2\hb_2^2}\frac{\cD_{i,\ga}\cD_{j,\de}}{\al_i^l\al_j^l}
q^{\frac{\al_i}{\hb_1}+\frac{\al_j}{\hb_2}}
\cY_{p_1}(\al_i,\hb_1,q)\cY_{p_2}(\al_j,\hb_2,q)
\left|_{\begin{subarray}{c}\hb_1=\frac{2\al_i}{\ga}\\
\hb_2=\frac{2\al_j}{\de}\end{subarray}}\right.=0.$$
By the $p\!=\!b\!=\!2$ case of Corollary~\ref{res_crl} and \e_ref{diskprp_e},
\begin{equation*}\begin{split}
&\frac{I_2(q)}{I_1(q)}\!\!
\sum_{\begin{subarray}{c}1\le i,j\le 2m\\ \ga,\de\in\Z^+\,\tn{odd}\\\end{subarray}}
\!\!\frac{1}{\hb_1^2\hb_2^2}\frac{\cD_{i,\ga}\cD_{j,\de}}{\al_i^l\al_j^l}
q^{\frac{\al_i}{\hb_1}+\frac{\al_j}{\hb_2}}\cY_2(\al_i,\hb_1,q)\cY_2(\al_j,\hb_2,q)
\left|_{\begin{subarray}{c}\hb_1=\frac{2\al_i}{\ga}\\
          \hb_2=\frac{2\al_j}{\de}\end{subarray}}\right.\\
&~~=\frac{I_2(q)}{I_1(q)}
\left[
\sum_{\begin{subarray}{c}1\le i\le 2m\\ \ga\in\Z^+\,\tn{odd}\\\end{subarray}}
\!\!\frac{\cD_{i,\ga}}{\al_i^l}q^{\frac{\al_i}{\hb}}
\frac{1}{I_2(q)}
\left\{\frac{\al_i}{\hb}\!+\!q\frac{\nd}{\nd q}\right\}
\left(\frac{1}{I_1(q)}
\left\{\frac{\al_i}{\hb}\!+\!q\frac{\nd}{\nd q}\right\}\cY_0(\al_i,\hb,q)\right)
\left|_{\hb=\frac{2\al_i}{\ga}}\right.   \!\right]^2\\
&~~=\frac{1}{I_1(q)I_2(q)}\!\!
\left[\! q\frac{\nd}{\nd q}
\left(\frac{1}{I_1(q)}q\frac{\nd}{\nd q}
\sum_{\begin{subarray}{c}1\le i\le 2m\\ \ga\in\Z^+\,\tn{odd}\\\end{subarray}}
\!\!\!\!\! q^{\frac{\al_i}{\hb}}\frac{\cD_{i,\ga}}{\al_i^l}\cY_0(\al_i,\hb,q)
\left|_{\hb=\frac{2\al_i}{\ga}}\right.\!\!\!\right)\!\right]^2\\
&~~=\frac{I_1(q)}{I_2(q)}
\left[\Big\{Q\frac{\nd}{\nd Q}\Big\}^2Z_{disk}(Q)\right]^2\,.
\end{split}\end{equation*}
The identity in Theorem~\ref{annulus_thm} follows immediately from 
\e_ref{annulus_initial_comp2} and the last three equations.

\begin{lmm}\label{res_lmm}
If $\a$ is an $l$-tuple of positive integers with $n\!\equiv\!|\a|\!=\!l\!+\!4$,
$b,p\!\in\!\Z$ with $-l\!\le\!p\!\le\!\min(b,0)$ and $b\!-\!p\!\le\!l\!+\!2$,
and $\cD_{i,\ga}$ is given~by \e_ref{disk-edge_e}, then
\begin{equation*}
\sum_{\begin{subarray}{c}1\le i\le 2m\\ \ga\in\Z^+\tn{odd}\end{subarray}}
\frac{1}{\hb^b}\frac{\cD_{i,\ga}}{\al_i^l}
q^{\frac{\al_i}{\hb}}\D^p\cY_0(\al_i,\hb,q)\left|_{\hb=\frac{2\al_i}{\ga}}\right.
=\begin{cases} 0,&\hbox{if}~p<b;\\
\frac{1}{2^{b-1}I_p(q)}\!\!
\sum\limits_{d\in\Z^+\,\tn{odd}}
\!\!\!\!\!d^pq^{\frac{d}{2}}\frac{\prod\limits_{k=1}^l(a_kd)!!}{(d!!)^n}\,,
&\hbox{if}~p=b. \end{cases}
\end{equation*}
\end{lmm}

\begin{proof}
By \e_ref{Yp_dfn_e2} and \e_ref{Yp_dfn_e}, 
$$\D^p\cY_0(\al_i,\hb,q)\left|_{\hb=\frac{2\al_i}{\ga}}\right.
=\frac{\al_i^l}{I_p(q)}\sum_{d=0}^{\i}q^d(\ga\!+\!2d)^p 
\left(\frac{\al_i}{\ga}\right)^{nd+p}
\frac{\prod\limits_{k=1}^l\!\frac{(a_k(\ga+2d))!!}{(a_k\ga)!!}}
{\prod\limits_{k=1}^n
\prod\limits_{\begin{subarray}{c}\ga+2\le s\le\ga+2d\\ s~\tn{odd}\end{subarray}}
\left(s\frac{\al_i}{\ga}-\al_k\right)}\,,$$
where $I_p(q)\!\equiv\!1$ if $p\!<\!0$.
Thus, by \e_ref{disk-edge_e} and the Residue Theorem on $S^2$,
\begin{equation*}\begin{split}
&2^{b-1}I_p(q)\!\!\!\!\!\!
\sum_{\begin{subarray}{c}1\le i\le 2m\\ \ga\in\Z^+\tn{odd}\end{subarray}}
\!\!\!\!\frac{1}{\hb^b}\frac{\cD_{i,\ga}}{\al_i^l}
q^{\frac{\al_i}{\hb}}\D^p\cY_0(\al_i,\hb,q)\left|_{\hb=\frac{2\al_i}{\ga}}\right.
=\!\!\!
\sum_{\begin{subarray}{c}1\le i\le 2m\\ \ga\in\Z^+\tn{odd}\end{subarray}}
\sum_{\begin{subarray}{c}t\in\Z^+\tn{odd}\\t\ge\ga\end{subarray}}\!\!\!\! t^pq^{\frac{t}{2}}
\frac{ \left(\frac{\al_i}{\ga}\right)^{\frac{nt+l+2}{2}+p-b}\prod\limits_{k=1}^l\!(a_kt)!!}
{\ga\underset{(k,s)\neq(i,\ga)}{\prod\limits_{k=1}^n
\prod\limits_{\begin{subarray}{c}1\le s\le t\\ s~\tn{odd}\end{subarray}}}
\!\left(s\frac{\al_i}{\ga}\!-\!\al_k\right)}\\
&\quad=\!\sum_{t\in\Z^+\tn{odd}}\!\!\!\! t^pq^{\frac{t}{2}}\!\!\!\!\!
\sum_{\begin{subarray}{c}1\le i\le 2m\\ \ga\in\Z^+\tn{odd}\\\ga\leq t\end{subarray}}
\!\Rs{z=\frac{\al_i}{\ga}}\!\!\left\{
\frac{z^{\frac{nt+l+2}{2}+p-b}\!\!\prod\limits_{k=1}^l\!(a_kt)!!}
{\prod\limits_{k=1}^n
\prod\limits_{\begin{subarray}{c}1\le s\le t\\ s~\tn{odd}\end{subarray}}
\!\left(sz\!-\!\al_k\right)}\right\}
=\!\sum_{t\in\Z^+\tn{odd}}\!\!\!\! t^pq^{\frac{t}{2}}
\Rs{w=0}\!\!\left\{
\frac{w^{b-1-p}\prod\limits_{k=1}^l\!(a_kt)!!}
{\prod\limits_{k=1}^n
\prod\limits_{\begin{subarray}{c}1\le s\le t\\ s~\tn{odd}\end{subarray}}
\!\left(s\!-\!\al_kw\right)}\right\}.
\end{split}\end{equation*}
The last expression above gives the right-hand side of the formula in
Lemma~\ref{res_lmm}.
Since
$$\frac{nt+l+2}{2}+p-b-\frac{t+1}{2}\ge0$$
with our assumptions, there is no residue at $z\!=\!0$. 
\end{proof}

\begin{crl}\label{res_crl}
If $\a$ is an $l$-tuple of positive integers with $n\!\equiv\!|\a|\!=\!l\!+\!4$,
$b,p\!\in\!\Z$ with $-l\!\le\!p\!\le\!b\!+\!1$ and $b\!-\!p\!\le\!l\!+\!2$,
and $\cD_{i,\ga}$ is given~by \e_ref{disk-edge_e}, then
\begin{alignat*}{1}
\sum_{\begin{subarray}{c}1\le i\le 2m\\ \ga\in\Z^+\tn{odd}\end{subarray}}
\frac{1}{\hb^b}\frac{\cD_{i,\ga}}{\al_i^l}
q^{\frac{\al_i}{\hb}}\D^p\cY_0(\al_i,\hb,q)\left|_{\hb=\frac{2\al_i}{\ga}}\right.
&=0  \hspace{2.68in} \hbox{if}~p\!<\!b;\\
\sum_{\begin{subarray}{c}1\le i\le 2m\\ \ga\in\Z^+\tn{odd}\end{subarray}}
\frac{1}{\hb^b}\frac{\cD_{i,\ga}}{\al_i^l}
q^{\frac{\al_i}{\hb}}\cY_p(\al_i,\hb,q)\left|_{\hb=\frac{2\al_i}{\ga}}\right.
&=\sum_{\begin{subarray}{c}1\le i\le 2m\\ \ga\in\Z^+\tn{odd}\end{subarray}}
\frac{1}{\hb^b}\frac{\cD_{i,\ga}}{\al_i^l}
q^{\frac{\al_i}{\hb}}\D^p\cY_0(\al_i,\hb,q)\left|_{\hb=\frac{2\al_i}{\ga}}\right.
\quad\hbox{if}~b\le l\!+\!2.
\end{alignat*}
\end{crl}

\begin{proof}
The first claim follows from \e_ref{Yp_dfn_e2}, Lemma~\ref{res_lmm}, and 
the product rule by induction on~$p$.
The $p\!<\!0$ cases of the second claim are immediate from~\e_ref{Yp_dfn_e}.
The remaining cases follow from~\e_ref{main_e0} together with
$\ctC_{p,s}^{(1)}\!=\!0$ and the first claim.
\end{proof}

\section{Klein Bottle GW-Invariants of CY CI Threefolds}
\label{klein_sec}

We now consider the one-point analogue of the graph-sum description of 
the Klein bottle invariants in~\cite{W1}.
We show that the resulting sums are weight-independent; 
so the invariants are well-defined for a fixed CY CI threefold~$X_{\a}$.
Furthermore, they satisfy the one-point analogue of 
the Klein bottle mirror symmetry prediction in~\cite{W1};
see Theorem~\ref{klein_thm}.
Once Klein bottle invariants are defined intrinsically and shown to satisfy
some sort of hyperplane relation,
the one-point version of the Klein bottle prediction of~\cite{W1}
will become equivalent to the original one, due to the divisor relation.

\subsection{Description and graph-sum definition}
\label{kleindfn_subs}

A \sf{degree~$d$ one-marked Klein bottle doubled map} to $\Pn$  is 
a tuple $(\ti{C},\tau,y,\ti{f})$ such that 
$(\ti{C},\tau,\ti{f})$ is an $\Om$-invariant map  
to~$\Pn$, $\tau$ is a fixed-point free involution, 
and $y\!\in\!\ti{C}^*$ is a smooth point.
An example of the restriction of~$\tau$ to the principal component~$\ti{C}_0$ of~$\ti{C}$
in this case is given~by
\BE{C0kleindfn_e}\begin{split}
&\ti{C}_0=(\Z_{2k}\!\times\!\P^1/\sim), 
~~\hbox{where}~~(s,[0,1])\sim(s\!+\!1,[1,0])~~\forall\,s\!\in\!\Z_{2k}\,,\\
&\hspace{.7in}
\tau\!:\ti{C}_0\lra\ti{C}_0,\quad\tau\big(s,[z_1,z_2]\big)=\big(s\!+\!k,[\bar{z}_1,\bar{z}_2]\big).
\end{split}\EE
The moduli space of stable degree~$d$ one-marked Klein bottle doubled maps 
to a CY CI threefold $X_{\a}\!\subset\!\Pn$,
$$\ov\M_{1,1}(X_{\a},d)^{\Om}\equiv \ov\M_{1,0,1}(X_{\a},\Om,d),$$
is expected to have a well-defined virtual class so that the number
$$\ti{K}_d\equiv\int_{\big[\ov\M_{1,1}(X_{\a},d)^{\Om}\big]^{vir}}\ev_1^*\nH\in\Q$$
is well-defined, with $\ti{K}_d\!=\!0$ if $d$ is odd.\footnote{This is consistent 
with the vector bundle 
$\V_{1,1,2d}^{\Om}$ in \e_ref{K2dsum_e} having a direct summand of odd real rank;
thus, its Euler class is zero.}
A combinatorial description of~$\ti{K}_d$ for $d$ even, in the spirit of~\cite{W1}, is
 motivated below.

Denote by 
$$\ov\M_{1,1}(\Pn,d)^{\Om;\T^m}\subset\ov\M_{1,1}(\Pn,d)^{\Om}$$
the subspace of the $\T^m$-fixed equivalence classes of stable
one-marked Klein bottle doubled maps to~$\Pn$.
As in the disk and annulus cases, one would expect~that 
\BE{K2dsum_e} 
\ti{K}_{2d}=-\sum_{F\subset\ov\M_{1,1}(\Pn,2d)^{\Om;\T^m}}
\int_F\frac{\E(\V_{1,1,2d}^{\Om})(\ev_1^*x)|_F}{\E(NF)}\,,\EE
where the sum is taken over the connected components $F$ of 
$\ov\M_{1,1}(\Pn,2d)^{\Om;\T^m}$, $\E(NF)$ is the equivariant Euler
class of the normal ``bundle"~$NF$ of~$F$ in $\ov\M_{1,1}(\Pn,2d)^{\Om}$,
and $\E(\V_{1,1,2d}^{\Om})$ is the equivariant Euler class of the obstruction ``bundle"
$$\V_{1,1,2d}^{\Om}\equiv\ov\M_{1,1}(\cL,2d)^{\Om}\lra\ov\M_{1,1}(\Pn,2d)^{\Om}\,,$$
with $\cL$ given by \e_ref{cLdfn_e}.
Similarly to the annulus case, $\ov\M_{1,1}(\Pn,2d)^{\Om}$ and $\V_{1,1,2d}^{\Om}$
may be singular near maps $(\ti{C},\tau,y,\ti{f})$ with contracted principle 
component~$\ti{C}_0$ or such that $\ti{C}/\tau$ contains a copy of~$\R P^2$.
However, the restriction of $\V_{1,1,2d}^{\Om}$ to the fixed loci~$F$ consisting of such maps
again contains a topologically trivial subbundle of weight~0 and 
thus does not contribute to~\e_ref{K2dsum_e}.
The same is the case for the fixed loci consisting of maps  $(\ti{C},\tau,y,\ti{f})$
such that $\ti{f}$ takes the point in the principal component $\ti{C}_0\!\subset\!\ti{C}$
closest to~$y$ to~$P_{2m+1}$ (if $n$ is odd).
There is also no genus~0 correction to~$\ti{K}_{2d}$.

Thus, by Corollary~\ref{fixedmap_crl}, all nonzero contributions to~\e_ref{K2dsum_e}
come from fixed loci~$F$ consisting of stable maps $(\ti{C},\tau,y,\ti{f})$
such that the principal component $\ti{C}_0\!\subset\!\ti{C}$ is a circle of 
spheres as in~\e_ref{C0kleindfn_e}, $\ti{f}|_{\ti{C}_0}$ is not constant, and
the irreducible component of~$\ti{C}_0$ closest to~$y$ is mapped 
to a fixed point~$P_i$ with $i\!\le\!2m$.
In this case, $\ti{C}_0$ breaks in a unique way into four connected subsets:
\begin{enumerate}
\item[(c1)] $\ti{C}_{0;c1}$,
the maximum connected union of irreducible components of $\ti{C}_0$
which is contracted by~$\ti{f}$ and contains the irreducible component of  $\ti{C}_0$
closest to~$y$;
\item[(c2)] $\ti{C}_{0;c2}=\tau(\ti{C}_{0;c1})$,
the maximum connected union of irreducible components of $\ti{C}_0$
which is contracted by $\ti{f}$ and contains the irreducible component of  $\ti{C}_0$
closest to~$\tau(y)$;
\item[(e1)] $\ti{C}_{0;e1}$, a chain of spheres running from $\ti{C}_{0;c1}$ 
to $\ti{C}_{0;c2}$;
\item[(e2)] $\ti{C}_{0;e2}\!=\!\tau(\ti{C}_{0;e1})$, the other chain of 
spheres running from $\ti{C}_{0;c1}$ to~$\ti{C}_{0;c2}$.\footnote{$e$ stands for {\it effective}, 
as $\ti{f}|_{\ti{C}_{0;e1}}$ and $\ti{f}|_{\ti{C}_{0;e2}}$ are not constant}
\end{enumerate}
The restrictions $\ti{f}_{\ti{C}|_{0;e1}}$ and $\ti{f}|_{\ti{C}_{0;e2}}$ are distinguished
in the counting scheme of~\cite{W1}.
Let $\ti{C}_{e1}\!\subset\!\ti{C}$ be the maximal connected union of irreducible
components containing~$\ti{C}_{0;e1}$, but not~any of the irreducible components
of~$\ti{C}_{0;c1}$ or~$\ti{C}_{0;c2}$.
The restriction $\ti{f}|_{\ti{C}_{e1}}$ determines an element 
$$[\ti{f}_{e1}]\in \ev_1^{-1}(P_i)\cap \ev_2^{-1}(P_{\bar{i}})\cap\ov\M_{0,2}(\Pn,r)^{\T^m}$$ 
for some $r\!\in\!\Z^+$ such that 
the restrictions of~$\ti{f}_{e1}$ to the irreducible components of~${\ti{C}_{e1}}$
containing the two marked points are not constant.
Denote the set of such maps by~$F_{r;i\bar{i}}$. 
Since $F_{r;i\bar{i}}$ is a union of topological components of $\ov\M_{0,2}(\Pn,r)^{\T^m}$,
it is smooth and so has a well-defined normal bundle,~$NF_{r;i\bar{i}}$.
In addition to the nodes $\ti{C}_{0;c1}\!\cap\!\ti{C}_{0;e1}$
and  $\ti{C}_{0;c1}\!\cap\!\ti{C}_{0;e2}$,
$\ti{C}_{0;c1}$ carries either the marked point~$y$ or the node that separates  
$\ti{C}_{0;c1}$ and~$y$ as well as $B\!\in\!\Z^{\ge0}$ additional nodes.
If $y\!\not\in\!\ti{C}_{0;c1}$, the restriction of~$\ti{f}$ to 
the maximal connected union~$\ti{C}_y$ of 
irreducible components of~$\ti{C}$ 
containing~$y$, but not any of the irreducible components of~$\ti{C}_{0;c1}$, determines an element 
$$[\ti{f}_y]\in\ev_1^{-1}(P_i)\cap\ov\M_{0,2}(\Pn,r^*)^{\T^m},$$ 
for some $r^*\!\in\!\Z^+$, such that the restriction of~$\ti{f}_y$ 
to the irreducible component of~$\ti{C}_y$ 
containing the first marked point is not constant ($y$ corresponds to the second marked point).
Denote the set of such maps by~$F_{r^*;i}$. 
Since $F_{r^*;i}$ is a union of topological components of $\ov\M_{0,2}(\Pn,r^*)^{\T^m}$,
it is smooth and so has a well-defined normal bundle,~$NF_{r^*;i}$.
Finally, for each $b\!=\!4,\ldots,B\!+\!3$,
the restriction of~$\ti{f}$ to~$\ti{C}_b$, the maximal connected union of 
irreducible components of~$\ti{C}$ containing the additional node~$b$, 
but not any of the irreducible components of~$\ti{C}_{0;c1}$,
determines an element 
$$[\ti{f}_b]\in\ev_1^{-1}(P_i)\cap\ov\M_{0,1}(\Pn,r_b)^{\T^m}$$ 
for some $r_b\!\in\!\Z^+$ such that the restriction of~$\ti{f}_b$ 
to the irreducible component of~$\ti{C}_b$ 
containing the marked point is not constant.
Denote the set of such maps by~$F_{r_b;i}'$. 
Since $F_{r_b;i}'$ is a union of topological components of $\ov\M_{0,1}(\Pn,r_b)^{\T^m}$,
it is smooth and has a well-defined normal bundle,~$NF_{r_b;i}'$.

In summary, a $\T^m$-fixed locus $F$ of $\ov\M_{1,1}(\Pn,d)^{\Om}$ contributing to~\e_ref{K2dsum_e}
admits a decomposition
\BE{Fklein_e} F\approx\big(\ov\cM_{0,3+B}\times F_{r;i\bar{i}}\times F_{r^*;i}
\times\prod_{b=4}^{B+3}F_{r_b;i}'\big)/S_B,\EE
where
\begin{itemize}
\item $B\!\in\!\Z^{\ge0}$, $1\!\le\!i\!\le\!2m$, and $\ov\cM_{0,3+B}$
is the moduli space of stable rational curves with $3\!+\!B$ marked points;
\item $r\!\in\!\Z^+$ and 
$F_{r;i\bar{i}}\subset\ev_1^{-1}(P_i)\cap\ev_2^{-1}(P_{\bar{i}})\cap\ov\M_{0,2}(\Pn,r)^{\T^m}$ 
is the subset of maps with non-trivial restrictions to 
the irreducible components of the domain containing the two marked points;
\item $r^*\!\in\!\Z^{\ge0}$ and $F_{r^*;i}\subset\ev_1^{-1}(P_i)\cap\ov\M_{0,2}(\Pn,r^*)^{\T^m}$ 
is the subset of maps  with non-trivial restriction to 
the irreducible component of the domain containing the first marked point if $r^*\!>\!0$,
while $F_{0;i}\!\equiv\!\{pt\}$;
\item $r_b\!\in\!\Z^+$ and $F_{r_b;i}'\subset\ev_1^{-1}(P_i)\cap\ov\M_{0,1}(\Pn,r_b)^{\T^m}$ 
is the subset of maps  with non-trivial restriction to 
the irreducible component of the domain containing the marked point;
\item $S_B$ is the symmetric group on $B$ elements that acts on $\ov\cM_{0,3+B}$ by permuting
the last $B$ marked points and on the $B$-fold product in~\e_ref{Fklein_e} by permuting its factors.
\end{itemize}
This is illustrated in Figure~1, with $B\!=\!2$.

\begin{figure}
\begin{pspicture}(2.3,-2.3)(10,1.8)
\psset{unit=.3cm}
\rput(35,2){\smsize{$F_{r;i\bar{i}}
 \subset\ev_1^{-1}(P_i)\cap\ev_2^{-1}(P_{\bar{i}})\cap\ov\M_{0,2}(\Pn,r)^{\T^m}$}}
\psline[linewidth=.07](22,3.5)(48,3.5)\psline[linewidth=.07](22,0.5)(48,0.5)
\psline[linewidth=.07](22,3.5)(22,0.5)\psline[linewidth=.07](48,3.5)(48,0.5)
\pscircle*(48,.5){.2}\pscircle*(22,0.5){.2}
\psline[linewidth=.07,linestyle=dashed](22,-6.5)(48,-6.5)
\psline[linewidth=.07,linestyle=dashed](22,-3.5)(48,-3.5)
\psline[linewidth=.07,linestyle=dashed](22,-6.5)(22,-3.5)
\psline[linewidth=.07,linestyle=dashed](48,-6.5)(48,-3.5)
\pscircle*(48,-3.5){.2}\pscircle*(22,-3.5){.2}
\pscircle[linewidth=.07](50,-1.5){2.83}
\rput(50.2,-.7){\smsize{$\ov\cM_{0,3+2}$}}\rput(50,-3.4){$P_i$}
\psline[linewidth=.03]{->}(50,-1.5)(50,-2.7)
\pscircle[linewidth=.07,linestyle=dashed](20,-1.5){2.83}
\rput(20.2,-.7){\smsize{$\ov\cM_{0,3+2}$}}\rput(20,-3.4){$P_{\bar{i}}$}
\psline[linewidth=.03]{->}(20,-1.5)(20,-2.6)
\rput{90}(50,3.5){\smsize{$F_{r^*;i}$}}
\psline[linewidth=.07](48.7,1.33)(51.3,1.33)\psline[linewidth=.07](48.7,5.5)(51.3,5.5)
\psline[linewidth=.07](48.7,1.33)(48.7,5.5)\psline[linewidth=.07](51.3,1.33)(51.3,5.5)
\pscircle*(50,1.33){.2}\pscircle*(20,-4.33){.2}
\psline[linewidth=.07,linestyle=dashed](18.7,-4.33)(21.3,-4.33)
\psline[linewidth=.07,linestyle=dashed](18.7,-8.5)(21.3,-8.5)
\psline[linewidth=.07,linestyle=dashed](18.7,-4.33)(18.7,-8.5)
\psline[linewidth=.07,linestyle=dashed](21.3,-4.33)(21.3,-8.5)
\rput{-90}(50,-6.5){\smsize{$F_{r_5;i}'$}}
\psline[linewidth=.07](48.7,-4.33)(51.3,-4.33)\psline[linewidth=.07](48.7,-8.5)(51.3,-8.5)
\psline[linewidth=.07](48.7,-4.33)(48.7,-8.5)\psline[linewidth=.07](51.3,-4.33)(51.3,-8.5)
\pscircle*(50,-4.33){.2}\pscircle*(20,1.33){.2}
\psline[linewidth=.07,linestyle=dashed](18.7,1.33)(21.3,1.33)
\psline[linewidth=.07,linestyle=dashed](18.7,5.5)(21.3,5.5)
\psline[linewidth=.07,linestyle=dashed](18.7,1.33)(18.7,5.5)
\psline[linewidth=.07,linestyle=dashed](21.3,1.33)(21.3,5.5)
\rput(55,-1.5){\smsize{$F_{r_4;i}'$}}
\psline[linewidth=.07](52.83,-.2)(52.83,-2.8)\psline[linewidth=.07](57,-.2)(57,-2.8)
\psline[linewidth=.07](52.83,-.2)(57,-.2)\psline[linewidth=.07](52.83,-2.8)(57,-2.8)
\pscircle*(52.83,-1.5){.2}\pscircle*(17.17,-1.5){.2}
\psline[linewidth=.07,linestyle=dashed](17.17,-.2)(17.17,-2.8)
\psline[linewidth=.07,linestyle=dashed](13,-.2)(13,-2.8)
\psline[linewidth=.07,linestyle=dashed](17.17,-.2)(13,-.2)
\psline[linewidth=.07,linestyle=dashed](17.17,-2.8)(13,-2.8)
\end{pspicture}
\caption{$\T^m$-fixed locus contributing to $\ti{K}_{2(r+r^*+r_4+r_5)}$;
the dotted portion is the reflection of the solid portion under the composition
with~$\Om$ and~$\tau$}
\label{klein_fig}
\end{figure}
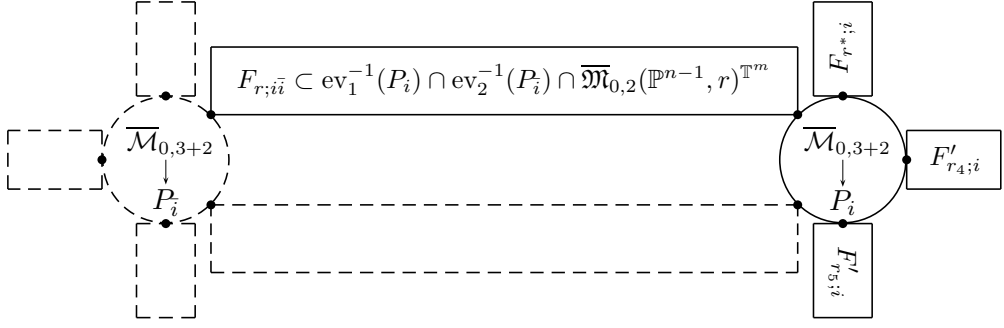

The composition of maps with $\tau$ on the right and $\Om$ on the left determines
an automorphism
$$\ov\M_{0,2}(\Pn,r)\lra\ov\M_{0,2}(\Pn,r).$$
Let $\bar\psi_i\in H^*(\ov\M_{0,2}(\Pn,r))$ denote the pull-back 
of the usual $i$-th $\psi$-class by this automorphism.
With respect to the decomposition~\e_ref{Fklein_e}, the normal bundle of the fixed locus
$F\subset\ov\M_{1,1}(\Pn,d)^{\Om}$ is described~by
\begin{equation*}\begin{split}
\frac{\E(NF)}{\E(T_{P_i}\Pn)}&=
\frac{\E(NF_{r;i\bar{i}})(\hb_1\!-\!\psi_1)(\hb_2\!-\!\bar\psi_2)}{\E(T_{P_i}\Pn)\E(T_{P_{\bar{i}}}\Pn)}
\cdot \frac{\E(NF_{r^*;i})(\hb_3\!-\!\psi_3)}{\E(T_{P_i}\Pn)}
\cdot \prod_{b=4}^{B+3}\frac{\E(NF_{r_b;i}')(\hb_b\!-\!\psi_b)}{\E(T_{P_i}\Pn)}\\
&=\frac{\E(NF_{r;i\bar{i}})(\hb_1\!-\!\psi_1)(\hb_2\!-\!\bar\psi_2)}
{\phi_i|_{P_i}\,\phi_{\bar{i}}|_{P_{\bar{i}}}}
\cdot \frac{\E(NF_{r^*;i})(\hb_3\!-\!\psi_3)}{\phi_i|_{P_i}}
\cdot \prod_{b=4}^{B+3}\frac{\E(NF_{r_b;i}')(\hb_b\!-\!\psi_b)}{\phi_i|_{P_i}}\,,
\end{split}\end{equation*}
where $\hb_b\!\equiv\!c_1(L_b)\in H^*(\ov\cM_{0,3+B})$ is the first Chern class 
of the universal tangent line bundle at the $b$-th marked point, 
$\psi_b\in  H^*\big(\ov\M_{0,1}(\Pn,r_b)\big)$
is the usual $\psi$-class for $b\!\ge\!4$, and $\psi_3$ denotes the pull-back of 
$\psi_1\in H^*(\ov\M_{0,2}(\Pn,r^*))$ by the third component projection map
in~\e_ref{Fklein_e}, with the convention~that 
$$\frac{\E(NF_{r^*;i})(\hb_3\!-\!\psi_3)}{\phi_i|_{P_i}} \equiv 1
\qquad\hbox{if}\quad r^*=0.$$
Similarly, the restriction of the numerator in~\e_ref{K2dsum_e} splits~as
\begin{equation*}\begin{split}
\frac{(\ev_1^*x)\E(\V_{1,1,2d}^{\Om})}{\E(\cL)|_{P_i}}\bigg|_F
&= \frac{\E(\V_{0,r})}{\E(\cL)|_{P_i}\E(\cL)|_{P_{\bar{i}}}}
\cdot\frac{(\ev_2^*x)\E(\V_{0,r^*})}{\E(\cL)|_{P_i}}
\cdot\prod_{b=4}^{B+3}\frac{\E(\V_{0,r_b})}{\E(\cL)|_{P_i}}\\
&=\frac{\E(\V_{0,r})}{\lr\a^2 \al_i^l\al_{\bar{i}}^l}
\cdot \frac{\E(\V_{0,r^*}'')(\ev_2^*x^{l+1})}{\al_i^l}
\cdot  \prod_{b=4}^{B+3}\E(\V_{0,r_b}')\,,
\end{split}\end{equation*}
where $\V_{0,r^*}'',\V_{0,r_b}'$ are as in Section~\ref{mainthm_sec} if $r\!>\!0$ and
$$\frac{\E(\V_{0,r^*}'')(\ev_2^*x^{l+1})}{\al_i^l} \equiv \al_i \qquad\hbox{if}\quad  r^*=0.$$
Putting this all together, we conclude that 
\BE{FKleinint_e}\begin{split}
&\int_F\frac{\E(\V_{1,1,2d}^{\Om})(\ev_1^*x)|_F}{\E(NF)}
=\frac{(-1)^{l+1}}{\lr\a\al_i^{2l}\!\!\prod\limits_{k\neq i}\!(\al_i\!-\!\al_k)}
\sum_{\begin{subarray}{c}p_1+\ldots+p_{B+3}=B\\ p_1,\ldots,p_{B+3}\ge0 \end{subarray}}\\
&\hspace{.5in}
\left(\!\!\bigg(\!\frac{\psi_1^{-(p_1+1)}\bar\psi_2^{-(p_2+1)}}{p_1!p_2!}\!\!\!\int_{F_{r;i\bar{i}}}
\!\!\!\!\!\!\frac{\E(\V_{0,r})(\ev_1^*\phi_i)(\ev_2^*\phi_{\bar{i}})}{\E(NF_{b;i\bar{i}})} \bigg)
\right.\\
&\hspace{.7in}
\left. \times\!
\bigg(\!\frac{\psi_3^{-(p_3+1)}}{p_3!}\!\!\!\int_{F_{r^*;i}}
\!\!\!\!\!\!\frac{\E(\V_{0,r^*}'')(\ev_1^*\phi_i)(\ev_2^*x^{l+1})}{\E(NF_{r^*;i})} \bigg)
\prod_{b=4}^{B+3}\!\bigg(\!\frac{\psi_b^{-(p_b+1)}}{p_b!}\!\!\!\int_{F_{r_b;i}'}
\!\!\!\!\!\!\frac{\E(\V_{0,r_b}')(\ev_1^*\phi_i)}{\E(NF_{r_b;i}')}\bigg)\!\!\right)\!,
\end{split}\EE
with
\BE{cZ0_e} \bigg(\!\frac{\psi_3^{-(p_3+1)}}{p_3!}\!\!\!\int_{F_{r^*;i}}
\!\!\!\!\!\!\frac{\E(\V_{0,r^*}'')(\ev_1^*\phi_i)(\ev_2^*x^{l+1})}{\E(NF_{b;i})} \bigg)
\equiv -\de_{p_3,0}\al_i^{l+1} \qquad\hbox{if}\quad r^*=0.\EE
Re-writing the right-hand side of \e_ref{FKleinint_e} explicitly in terms of 
the vertices and edges of the graph encoding~$F$ gives the one-marked version of the long formula
in \cite[Section~3.4]{W1}.

We will instead use~\e_ref{FKleinint_e} to re-write~\e_ref{K2dsum_e} 
in terms of the generating functions~$\cZ$ of \e_ref{Z_dfn} and~\e_ref{equivZ_e2}
and~$\cZ_p$ of~\e_ref{equivZ_e}.
If $1\!\le\!i\!\neq\!j\!\le\!n$ and $d\!\in\!\Z^+$, let 
\BE{tifCdfn_e}\ti\fC_i^j(d)=\frac{\prod\limits_{k=1}^l    
          \prod\limits_{r=1}^{a_kd}\left(a_k\al_i+r\,\frac{\al_j-\al_i}{d}\right)}
{d\underset{(r,k)\neq(d,j)}{\prod\limits_{r=1}^d\prod\limits_{k=1}^n}
            \left(\al_i-\al_k+r\,\frac{\al_j-\al_i}{d}\right)}\in\Q_{\al}.\EE
This is undefined if $n$ is odd, $d$ is even, and $j\!=\!\bar{i}$.
In such a case, we cancel the $(k,r)\!=\!(n,d/2)$ factor in the denominator 
with the $(k,r)\!=\!(1,a_1d/2)$ factor in the numerator to define
\BE{tifCdfn_e2}\ti\fC_i^{\bar{i}}(d)=\de_{l,1}\frac{n   
          \prod\limits_{r=1}^{nd/2-1}\!\!\left(r\,\frac{\al_i}{d/2}\right)
          \cdot \prod\limits_{r=1}^{nd/2}\!\!\left(r\,\frac{\al_{\bar{i}}}{d/2}\right) }
{d\prod\limits_{r=1}^{d/2-1}\prod\limits_{k=1}^n\!\!\left(\frac{2r}{d}\al_i-\al_k\right)
\cdot\prod\limits_{k=1}^{2m}(-\al_k)\cdot
\underset{(r,k)\neq(d/2,\bar{i})}{\prod\limits_{r=1}^{d/2}\prod\limits_{k=1}^n}
           \!\!\!\!\!\left(\frac{2r}{d}\al_{\bar{i}}\!-\!\al_k\right)}.\EE
With $\ev_1:\ov\M_{0,1}(\Pn,d)\lra\Pn$ as before, let 
$$\cZ'(x,\hb,Q) \equiv
\hb+\sum_{d=1}^{\i}\!Q^d\ev_{1*}\left[\frac{\E(\V_{\a}')}{\hb\!-\!\psi_1}\right]
\in \left(H^*_{\T}(\Pn)\right)\big[\big[\hb^{-1},Q\big]\big].$$
By the string relation \cite[Section~26.3]{MirSym},
\BE{cZprvscZ_e}\cZ'(x,\hb,Q)=\hb\cZ(x,\hb,Q)\EE
and thus is described by~\e_ref{Z1ms_e}.
In particular,
\BE{cZprRS_e}\Rs{\hb=\frac{\al_j-\al_i}{d}}\big\{\hb^{-(p+1)}\cZ'(\al_i,\hb,Q)\big\}
=\bigg(\frac{\al_j-\al_i}{d}\bigg)^{-(p+1)}\ti\fC_i^j(d)Q^d\cZ'(\al_j,(\al_j\!-\!\al_i)/d,Q),\EE
if $1\!\le\!i\!\neq\!j\!\le\!n$, $d\!\in\!\Z^+$, and $p\!\in\!\Z$.

If $n$ is even, the restriction of an element of~$F_{r_b;i}'$ to the component 
of its domain containing the marked point is the degree~$d_b$ cover of the line~$\ell_{ij_b}$
branched over only~$P_i$ and~$P_{j_b}$, for some $d_b\!\in\!\Z^+$ and $j_b\!\in\![n]\!-\!i$.
Similarly to \cite[Section~30.1]{MirSym} and \cite[Section~2.2]{bcov1},
\BE{sidesum_e}\begin{split}
&\sum_{\begin{subarray}{c}F_{r_b;i}'\\ d_b=d,j_b=j\end{subarray}}
\!\!\!\!\!\!Q^{r_b}\bigg(\!\psi_b^{-(p+1)}\!\!\!\int_{F_{r_b;i}'}
\!\!\!\!\!\!\frac{\E(\V_{0,r_b}')(\ev_1^*\phi_i)}{\E(NF_{r_b;i}')}\bigg)\\
&\hspace{1.5in}
=Q^d\bigg(\frac{\al_j\!-\!\al_i}{d}\bigg)^{-(p+1)}
\ti\fC_i^j(d)\cZ'(\al_j,(\al_j\!-\!\al_i)/d,Q)
\end{split}\EE
for all $j\!\in\![n]\!-\!i$ and $d\!\in\!\Z^+$.
If $n$ is odd, but $j\!\neq\!\bar{i}$ or $d$ is odd, 
\e_ref{sidesum_e} holds by the same reasoning.
If $n$ is odd, $j\!=\!\bar{i}$, and $d$ is even, the first component~$F_{r_b;i}'^{(1)}$ 
of the locus~$F_{r_b;i}'$ consists of
the $d/2$-fold covers of the conics~$\cT_i(a,b)$, with $[a,b]\!\in\!\P^1$, branched
over only~$P_i$ and~$P_{\bar{i}}$.
The contribution of this component to the integral in~\e_ref{sidesum_e} can be computed
by turning the action of the last component of~$\T^n$ back~on (making $\al_n\!\neq\!0$).
This action on~$F_{r_b;i}'^{(1)}$ has two fixed points: the $d$-fold cover of the line $\ell_{i\bar{i}}$
branched over only~$P_i$ and~$P_{\bar{i}}$ and  
the $d/2$-fold cover of $\ell_{in}\!\cup\!\ell_{n\bar{i}}$,
whose restriction to each of the two components of the domain is branched over only~$P_n$ 
and either~$P_i$ or~$P_{\bar{i}}$.
The contribution of the former to the integral in~\e_ref{sidesum_e} 
vanishes.\footnote{As in \cite[Section~30.1]{MirSym} and \cite[Section~2.2]{bcov1},
the Euler class of the restriction of the bundle~$\V_{\a}'$ to 
this $d$-fold cover is given by the numerator in~\e_ref{tifCdfn_e}, but
its $(k,r)\!=\!(1,a_1d/2)$ factor vanishes in this case.}
The contribution of the latter is given by the right-hand side of~\e_ref{sidesum_e} with  
$\fC^j_i(d)\!=\!\fC^{\bar{i}}_i(d)$ given by~\e_ref{tifCdfn_e2}.\footnote{By 
\cite[Section~27.4]{MirSym}, the normal bundle 
to~$F_{r_b;i}'^{(1)}$ contains the line bundle of smoothings
of the node of the domain of the cover of  $\ell_{in}\!\cup\!\ell_{n\bar{i}}$;
its first Chern class is 
$$\frac{\al_n\!-\!\al_i}{d/2}+\frac{\al_n\!-\!\al_{\bar{i}}}{d/2}=\frac{4\al_n}{d}\,.$$
By \cite[Section~27.2]{MirSym}, the restriction of~$\V_{\a}$ contains line bundles
with first Chern classes $a_k\al_n$ with $k\!\in\![n]$.
Thus, the only factor of~$\al_n$ in the denominator can be canceled out with a factor in
the numerator, before $\al_n$ is set to~0.}
Thus, \e_ref{sidesum_e} holds in all cases.

By \e_ref{sidesum_e}, \e_ref{cZprRS_e}, and~\e_ref{cZprvscZ_e}, 
\BE{sidesum_e2}\sum_{F_{r_b;i}'}Q^{r_b}\bigg(\!\psi_b^{-(p+1)}\!\!\!\int_{F_{r_b;i}'}
\!\!\!\!\!\!\frac{\E(\V_{0,r_b}')(\ev_1^*\phi_i)}{\E(NF_{r_b;i}')}\bigg)
=\sum_{d=1}^{\i}\sum_{j\neq i}
\Rs{\hb=\frac{\al_j-\al_i}{d}}\big\{\hb^{-p}\cZ(\al_i,\hb,Q)\big\},\EE
where the first sum is taken over all possibilities for 
$F_{r_b;i}'\subset\ov\M_{0,1}(\Pn,r_b)^{\T^m}$ as described above.
By~\e_ref{Z1ms_e}, $\cZ(\al_i,\hb,Q)$ is a power series in~$Q$ with coefficients
in~$\Q_{\al}(\hb)$ that are regular away from $\hb\!=\!(\al_j\!-\!\al_i)/d$ with $i,j\!\in\![n]$
and $d\!\in\!\Z^+$.
Furthermore, by \e_ref{Z_e},
$$\Rs{\hb=\i}\big\{\hb^{-p}\cZ(\al_i,\hb,Q)\big\}=-\de_{p,1}\,
\qquad\forall\,p\in\Z^{\ge0}.$$
Thus, by~\e_ref{sidesum_e2} and the Residue Theorem on $S^2$,
\BE{sidesum_e3}\sum_{F_{r_b;i}'}Q^{r_b}\bigg(\!\psi_b^{-(p+1)}\!\!\!\int_{F_{r_b;i}'}
\!\!\!\!\!\!\frac{\E(\V_{0,r_b}')(\ev_1^*\phi_i)}{\E(NF_{r_b;i}')}\bigg)
=-\Rs{\hb=0}\big\{\hb^{-p}\cZ^*(\al_i,\hb,Q)\big\} \qquad\forall\,p\in\Z^{\ge0},\EE
where 
$$\cZ^*(x,\hb,Q)\equiv \cZ(x,\hb,Q)-1\in Q\cdot\Q_{\al}(\hb)\big[\big[Q\big]\big].$$

By the same reasoning, 
\BE{sidesum_e4}
\sum_{F_{r^*;i}}Q^{r^*}
\bigg(\!\psi_3^{-p}\!\!\!\int_{F_{r^*;i}}
\!\!\!\!\!\!\frac{\E(\V_{0,r^*}'')(\ev_1^*\phi_i)(\ev_2^*x^{l+1})}{\E(NF_{r^*;i})} \bigg)
=-\Rs{\hb=0,\i}\big\{\hb^{-p}\cZ_1(\al_i,\hb,Q)\big\} \qquad\forall\,p\in\Z^+.\EE
The analogue of \e_ref{sidesum_e} holds with $\ti\fC_i^j(d)$
replaced by the constant $\fC_i^j(d)$ defined in~\e_ref{Cdfn} 
and its reduction similarly to~\e_ref{tifCdfn_e2}, while
$$\Rs{\hb=\i}\big\{\hb^{-p}\cZ_1(\al_i,\hb,Q)\big\}=-\de_{p,1}\al_i^{l+1}\,
\qquad\forall\,p\in\Z^+.$$
This extra correction (which appears with the opposite sign) 
is off-set by the $r^*\!=\!0$ case; see \e_ref{cZ0_e}.
Finally,
\BE{sidesum_e5}\begin{split}
&\sum_{F_{r;i\bar{i}}}Q^r
\bigg(\!\psi_1^{-p_1}\bar\psi_2^{-p_2}\!\!\!\int_{F_{r;i\bar{i}}}
\!\!\!\!\!\!\frac{\E(\V_{0,r})(\ev_1^*\phi_i)(\ev_2^*\phi_{\bar{i}})}{\E(NF_{b;i\bar{i}})} \bigg)\\
&\hspace{1in}
=(-1)^{p_2}\sum_{F_{r;i\bar{i}}}Q^r
\bigg(\!\psi_1^{-p_1}\psi_2^{-p_2}\!\!\!\int_{F_{r;i\bar{i}}}
\!\!\!\!\!\!\frac{\E(\V_{0,r})(\ev_1^*\phi_i)(\ev_2^*\phi_{\bar{i}})}{\E(NF_{b;i\bar{i}})}\bigg)\\
&\hspace{1in}
=(-1)^{p_2}\Rs{\hb_1=0}\Big\{\Rs{\hb_2=0}\big\{
\hb_1^{-p_1}\hb_2^{-p_2}\cZ(\al_i,\al_{\bar{i}},\hb_1,\hb_2,Q)\big\}\Big\}
\qquad\forall\,p_1,p_2\!\in\!\Z^+.
\end{split}\EE
The analogue of \e_ref{sidesum_e} now holds for $\hb_1$ and $\hb_2$ separately,
with  $\ti\fC_i^j(d)$ replaced by~$\fC_i^j(d)$, as above.
The coefficient of $Q^0$ in~\e_ref{Z_dfn}
has no effect on the residue in this case.

Combining \e_ref{K2dsum_e}, \e_ref{FKleinint_e}, and \e_ref{sidesum_e3}-\e_ref{sidesum_e5}, 
we obtain the following 
re-formulation of the combinatorial definition of the Klein bottle invariants
in~\cite{W1}.

\begin{dfn}\label{K2d_dfn}
The degree $2d$ one-point Klein bottle invariant $\ti{K}_{2d}$ of $(X_{\a},\Om)$ is given by
\BE{K2ddfn_e}\begin{split}
&\sum_{d=1}^{\i}Q^d \ti{K}_{2d}= \frac{(-1)^{l+1}}{\lr{\a}}
\sum_{i=1}^{2m}
\frac{1}{\al_i^{2l}\!\!\prod\limits_{k\neq i}\!(\al_i\!-\!\al_k)}
\sum_{B=0}^{\i}
\sum_{\begin{subarray}{c}p_1+\ldots+p_{B+3}=B\\ p_1,\ldots,p_{B+3}\ge0 \end{subarray}}\\
&\hspace{1.2in}\left( 
\Rs{\hb_1=0}\bigg\{\Rs{\hb_2=0}\bigg\{
\frac{(-1)^{p_1}}{p_1!p_2!\,\hb_1^{p_1+1}\hb_2^{p_2+1}}
\cZ(\al_i,\al_{\bar{i}},\hb_1,\hb_2,Q)\bigg\}\bigg\}
\right.\\
&\hspace{1.5in}\left.
\times\Rs{\hb=0}\bigg\{\frac{(-1)^{p_3}}{p_3!\,\hb^{p_3+1}}\cZ_1(\al_i,\hb,Q)\bigg\}
\prod_{b=4}^{B+3}
\Rs{\hb=0}\bigg\{\frac{(-1)^{p_b}}{p_b!\,\hb^{p_b}}\cZ^*(\al_i,\hb,Q)\bigg\}\right).
\end{split}\EE
\end{dfn}

\subsection{Some preliminaries}
\label{kleinprelim_subs}

The last factor in~\e_ref{K2ddfn_e} can be readily summed~up
over all possibilities for $(p_4,p_5,\ldots)$ if the power series
$\cZ(\al_i,\hb,Q)\in\Q_{\al}(\hb)[[Q]]$ admits
an expansion of a certain form; see Lemma~\ref{ResSum_lmm} below.
By Lemma~\ref{cYexp_lmm} below and~\e_ref{Z1ms_e}, $\cZ(\al_i,\hb,Q)$ does admit such an expansion;
this lemma also provides the two  coefficients for this expansion that are relevant for our purposes.

For $r\!\!\in\!\Z^{\ge0}$, denote by $\si_r(z)$ the $r$-th elementary symmetric polynomial
in $\{z\!-\!\al_k\}$.
Define
\BE{Ldfn_e}\begin{split}
&L(x,q)\in x+x^n q\cdot\Q[\al_1,\ldots,\al_n,x,\si_{n-1}(x)^{-1}]\big[\big[x^{n-1}q\big]\big],\\
&\qquad\hbox{by}\hspace{.2in}
\si_n\big(L(x,q)\big)-q\,\a^{\a}L(x,q)^n=\si_n(x).
\end{split}\EE

\begin{lmm}[{\cite[Lemma~2.2(ii)]{bcov1}}\footnotemark]
\label{ResSum_lmm}
Let $R\!\supset\!\Q$ be any field. 
If $\F^*\!\in\!R(\hb)[[q]]$ admits an expansion around $\hb\!=\!0$ of the form
$$1+\F^*(\hb,q)= \ne^{\ze(q)/\hb}\sum_{s=0}^{\i}\Psi_s(q)\hb^s\,$$
with $\xi,\Psi_1,\Psi_2,\ldots\in qR[[q]]$ and $\Psi_0\in 1+q R[[q]]$, then
$$\sum_{B=0}^{\i}\sum_{\begin{subarray}{c}p_1+\ldots+p_B=B-p\\ p_1,\ldots,p_B\ge0 \end{subarray}}
\prod_{b=1}^B
\Rs{\hb=0}\bigg\{\frac{(-1)^{p_b}}{p_b!\,\hb^{p_b}}\F^*(\hb,q)\bigg\}
=\frac{\ze(q)^p}{\Psi_0(q)} \qquad\forall\,p\in\Z^{\ge0}\,.$$
\end{lmm}

\footnotetext{\cite[Lemma~2.2(ii)]{bcov1} is stated only for $R\!=\!\Q_{\al}$,
but the argument applies to any field containing~$\Q$.}

\begin{lmm}\label{cY0_lmm}
The power series 
$$\wt\cY_{-l}(x,\hb,q)
\equiv
\sum_{d=0}^{\i}q^d\frac{\prod\limits_{k=1}^l\prod\limits_{r=0}^{a_kd-1}(a_kx+r\hb)}
{\prod\limits_{r=1}^d\big(\si_n(x\!+\!r\hb)-\si_n(x)\big)}
\in\Q_{\al}(x,\hb)\big[\big[q\big]\big]$$
admits an expansion around $\hb\!=\!0$ of the form
$$\wt\cY_{-l}(x,\hb,q)=\ne^{\xi(x,q)/\hb}\sum_{s=0}^{\i}\Phi_{-l;s}(x,q)\hb^s\,$$
with $\xi,\Phi_{-l;1},\Phi_{-l;2},\ldots\in q\cdot\Q_{\al}(x)[[q]]$
and $\Phi_{-l;0}\in 1+q\cdot\Q_{\al}(x)[[q]]$.
\end{lmm}

\begin{proof}
Since $\wt\cY_{-l}\!\in\!1+q\cdot\Q_{\al}(x,\hb)[[q]]$, there is an expansion
\BE{cY0ln_e}
\ln\wt\cY_{-l}(x,\hb,q)=\sum_{d=1}^{\i}\sum_{s=s_{\min}(d)}^{\i}\!\!\!\!\!C_{d,s}(x)\hb^sq^d\,\EE
around $\hb\!=\!0$, with $C_{d,s}(x)\!\in\!\Q_{\al}(x)$;
we can assume that $C_{d,s_{\min}(d)}\!\neq\!0$ if $s_{\min}(d)\!<\!0$.
The claim of Lemma~\ref{cY0_lmm} is equivalent to the statement 
$s_{\min}(d)\!\ge\!-1$ for all $d\!\in\!\Z^+$; in such a case
$$\xi(x,q)=\sum_{d=1}^{\i}C_{d,-1}(x)q^d\,.$$
Suppose instead $s_{\min}(d)\!<\!-1$ for some $d\!\in\!\Z^+$.
Let
\BE{dmin_e} d^*\equiv\min\big\{d\!\in\!\Z^+\!:\,s_{\min}(d)\!<\!-1\big\}\ge1,
\qquad s^*\equiv s_{\min}(d^*)\le-2.\EE
The power series $\wt\cY_{-l}$ satisfies the differential equation
\BE{cY0ode_e}
\Bigg\{\si_n(x\!+\!\hb\nD)
-q\prod_{k=1}^l\prod_{r=0}^{a_k-1}\big(a_kx\!+\!a_k\hb\nD\!+\!r\hb\big)\Bigg\}\wt\cY_{-l}(x,\hb,q)
=\si_n(x)\wt\cY_{-l}(x,\hb,q),\EE
where $\nD\!=\!q\frac{\nd}{\nd q}$ as before.
By \e_ref{cY0ln_e}, \e_ref{dmin_e}, and induction on the number of derivatives taken,
\BE{cY0lnchng_e}\begin{split}
\frac{\left\{\si_n(x\!+\!\hb\nD)\right\}\wt\cY_{-l}(x,\hb,q)}
{\si_n(x)\cdot\wt\cY_{-l}(x,\hb,q)}
&=1+\sum_{k=1}^n\frac{d^*C_{d^*,s^*}}{x\!-\!\al_k}\hb^{s^*+1}q^{d^*}+A(x,\hb,q)\,,\\
q\,\frac{\left\{\prod\limits_{k=1}^l\prod\limits_{r=0}^{a_k-1}
\big(a_kx\!+\!a_k\hb\nD\!+\!r\hb\big)\right\}\wt\cY_{-l}(x,\hb,q)}
{\prod\limits_{k=1}^l\prod\limits_{r=0}^{a_k-1}(a_kx\!+\!r\hb)\cdot\wt\cY_{-l}(x,\hb,q)}
&= B(x,\hb,q)\,,
\end{split}\EE
for some
$$A,B\in q\cdot\Q_{\al}(x,\hb)_0\big[\big[q\big]\big]
+q^{d^*}\hb^{s^*+2}\cdot\Q_{\al}(x,\hb)_0\big[\big[q\big]\big]
+q^{d^*+1}\cdot\Q_{\al}(x,\hb)\big[\big[q\big]\big],$$
where $\Q_{\al}(x,\hb)_0\subset\Q_{\al}(x,\hb)$ is the subring of rational functions
in $\al$, $x$, and $\hb$ that are regular at $\hb\!=\!0$.
Combining \e_ref{cY0ode_e} and \e_ref{cY0lnchng_e}, we conclude that $C_{d^*,s^*}\!=\!0$,
contrary to the assumption.
\end{proof}

\begin{lmm}\label{cYexp_lmm}
The power series 
$$\wt\cY(x,\hb,q)
\equiv
\sum_{d=0}^{\i}q^d\frac{\prod\limits_{k=1}^{l}\prod\limits_{r=1}^{a_kd}(a_kx+r\hb)}
{\prod\limits_{r=1}^{d}\big(\si_n(x\!+\!r\hb)-\si_n(x)\big)}\in\Q_{\al}(x,\hb)\big[\big[q\big]\big]$$
admits an expansion around $\hb\!=\!0$ of the form
\BE{cY0exp_e2}
\wt\cY(x,\hb,q)=\ne^{\xi(x,q)/\hb}\sum_{s=0}^{\i}\Phi_s(x,q)\hb^s\,\EE
with $\xi,\Phi_1,\Phi_2,\ldots\in q\cdot\Q_{\al}(x)[[q]]$
and $\Phi_0\in 1+q\cdot\Q_{\al}(x)[[q]]$ such that 
\begin{gather}
\label{xi_e}
x+\nD\xi(x,q)=L(x,q), \\
\label{Phi0_e}
\Phi_0(x,q)=\left(\frac{x\cdot\si_{n-1}(x)}
{L(x,q)\,\si_{n-1}(L(x,q))-n(\si_n(L(x,q))-\si_n(x))}\right)^{1/2}
\bigg(\frac{L(x,q)}{x}\bigg)^{(l+1)/2}\,.
\end{gather}
\end{lmm}

\begin{proof}
Since the power series $\wt\cY_{-l}$ admits an expansion of the form~\e_ref{cY0exp_e2}
by Lemma~\ref{cY0_lmm} and
\BE{cY0vscYml_e}\wt\cY(x,\hb,q)=\bigg\{1+\frac{\hb}{x}D\bigg\}^l\wt\cY_{-l}(x,\hb,q)\,,\EE
the power series $\wt\cY$ also admits an expansion of the form~\e_ref{cY0exp_e2},
with the same power series~$\xi$.
Let
$$\ov\cY(x,\hb,q)\equiv\ne^{-\xi(x,q)/\hb}\wt\cY(x,\hb,q)
=\sum_{s=0}^{\i}\Phi_s(x,q)\hb^s\,.$$
Since the power series $\wt\cY(x,\hb,q)$ satisfies
$$\Bigg\{\si_n(x\!+\!\hb\nD)
-q\prod_{k=1}^l\prod_{r=1}^{a_k}\big(a_kx\!+\!a_k\hb\nD\!+\!r\hb\big)\Bigg\}\wt\cY(x,\hb,q)
=\si_n(x)\wt\cY(x,\hb,q),$$
the power series $\ov\cY(x,\hb,q)$ satisfies
\BE{ovcY_e}\Bigg\{\si_n(x\!+\!\nD\xi\!+\!\hb\nD)
-q\prod_{k=1}^l\prod_{r=1}^{a_k}\big(a_k(x\!+\!\nD\xi)\!+\!a_k\hb\nD\!+\!r\hb\big)
-\si_n(x)\Bigg\}\ov\cY(x,\hb,q)=0.\EE
Considering the coefficient of $\hb^0$ in this equation, we obtain
$$\big\{\si_n(x\!+\!\nD\xi)-q\a^{\a}(x\!+\!\nD\xi)^n-\si_n(x)\big\}\Phi_0(x,q)=0.$$
Since $\Phi_0(x,0)\!=\!1$, this gives \e_ref{xi_e} by~\e_ref{Ldfn_e}.
Note that 
\BE{Lder_e}\begin{split}
\frac{\nD L(x,q)}{L(x,q)}
&=\frac{\si_n(L(x,q))\!-\!\si_n(x)}{
L(x,q)\si_{n-1}\big(L(x,q)\big)-n\big(\si_n(L(x,q))\!-\!\si_n(x)\big)}\\
&=\frac{q\a^{\a}L(x,q)^n}{L(x,q)\si_{n-1}\big(L(x,q)\big)-n\big(\si_n(L(x,q))\!-\!\si_n(x)\big)}\,
\end{split}\EE
by \e_ref{Ldfn_e}.
Substituting $L$ for $x\!+\!\nD\xi$ in~\e_ref{ovcY_e}, taking the coefficient of~$\hb^1$
of this equation, and using~\e_ref{Ldfn_e} and~\e_ref{Lder_e}, we obtain
\begin{equation*}\begin{split}
\nD\left\{
\left(\frac{x\cdot\si_{n-1}(x)}
{L(x,q)\,\si_{n-1}(L(x,q))-n(\si_n(L(x,q))-\si_n(x))}\right)^{-\frac12}
\bigg(\frac{L(x,q)}{x}\bigg)^{-\frac{l+1}2} \Phi_0\right\}=0.
\end{split}\end{equation*}
Along with $\Phi_0(x,0)\!=\!1$, this gives \e_ref{Phi0_e}.
\end{proof}

\subsection{Proof of Theorem~\ref{klein_thm}}
\label{kleinpf_subs}

With $\xi$ as in Lemma~\ref{cYexp_lmm}, let $\eta(x,q)=\xi(x,q)-J(q)x$.
Since
$$\Rs{\hb=0}\bigg\{\frac{1}{\hb^p}\cY(\al_i,\hb,q)\bigg\}
=\Rs{\hb=0}\bigg\{\frac{1}{\hb^p}\wt\cY(\al_i,\hb,q)\bigg\}
=\bigg(\Rs{\hb=0}\bigg\{\frac{1}{\hb^p}\wt\cY(x,\hb,q)\bigg\}\bigg)\bigg|_{x=\al_i}
\qquad\forall\,i\in[n],$$
by \e_ref{Z1ms_e} and Lemmas~\ref{ResSum_lmm} and~\ref{cYexp_lmm} 
\BE{ResSumCrl_e}
\sum_{B=0}^{\i}
\sum_{\begin{subarray}{c}p_4+\ldots+p_{B+3}=B-p\\ p_4,\ldots,p_{B+3}\ge0 \end{subarray}}\\
\prod_{b=4}^{B+3}
\Rs{\hb=0}\bigg\{\frac{(-1)^{p_b}}{p_b!\,\hb^{p_b}}\cZ^*(x,\hb,Q)\bigg\}\bigg|_{x=\al_i}
=\frac{\eta(x,q)^p}{\Phi_0(x,q)/I_0(q)}\bigg|_{x=\al_i}
 \quad\forall\,p\in\Z^{\ge0}.\EE
By \e_ref{Ldfn_e}, $ L(-x,q)\!=\!-L(x,q)$;
since $\al_{\bar{i}}\!=\!-\al_i$, $\eta(\al_{\bar{i}},q)\!=\!-\eta(\al_i,q)$
by~\e_ref{xi_e}.
Thus, by~\e_ref{K2ddfn_e}  and~\e_ref{ResSumCrl_e},
\BE{K2ddfn_e2}\begin{split}
&\sum_{d=1}^{\i}Q^d\ti{K}_{2d} = \frac{(-1)^{l+1}}{\lr\a}
\sum_{i=1}^{2m}\left( \frac{I_0(q)/\Phi_0(x,q)}{x^{2l}\si_{n-1}(x)}\,
\Rs{\hb=0}\bigg\{\frac{\ne^{-\frac{\eta(x,q)}{\hb}}\cZ_1(x,\hb,Q)}{\hb}\bigg\}\right.\\
&\hspace{1.8in}\times\left.
\Rs{\hb_1=0}\bigg\{\Rs{\hb_2=0}\bigg\{
\frac{\ne^{-\frac{\eta(x,q)}{\hb_1}-\frac{\eta(\bar{x},q)}{\hb_2}}
\cZ(x,\bar{x},\hb_1,\hb_2,Q)}{\hb_1\hb_2}\bigg\}\bigg\}
\right)\Bigg|_{x=\al_i}.
\end{split}\EE

By \e_ref{main_e1}, \e_ref{main_e0}, Lemmas~\ref{cY0_lmm} and~\ref{cYexp_lmm},
and~\e_ref{cY0vscYml_e}, 
\begin{gather}
\label{Res0_e}
\Rs{\hb=0}\bigg\{\frac{\ne^{-\frac{\eta(x,q)}{\hb}}\cZ_p(x,\hb,Q)}{\hb}\bigg\}
=x^l\Phi_0(x,q)L(x,q)^p
\sum_{r=0}^{|p|}\frac{\ctC_{p,p-r}^{(r)}(q)}{L(x,q)^r\!\prod\limits_{s=0}^{p-r}\!I_s(q)}\,,\\
\label{Res1_e}
\begin{split}
\Rs{\hb=0}\bigg\{\frac{\ne^{-\frac{\eta(x,q)}{\hb}}\cZ_p(x,\hb,Q)}{\hb^2}\bigg\}
=x^l\Phi_0(x,q)L(x,q)^p
\left(\sum_{r=2}^{p-1}
\frac{\ctC_{p,p-r-1}^{(r)}(q)}{L(x,q)^{r+1}\!\prod\limits_{s=0}^{p-r-1}\!\!\!I_s(q)}
\right.\qquad\qquad&\\
\left.+\sum_{r=0}^{|p|}\frac{\ctC_{p,p-r}^{(r)}(q)}{L(x,q)^r\!
\prod\limits_{s=0}^{p-r}\!I_s(q)}\bigg(\frac{\Phi_1(x,q)}{\Phi_0(x,q)}+A_{p-r}(x,q)\bigg)
\right)&\,,
\end{split}
\end{gather}
with $\ctC_{p,s}^{(r)}\!\equiv\!\de_{p,s}\de_{r,0}$ for $p\!<\!0$ or $s\!<\!0$ and
$$A_p=\frac{1}{L}\bigg(p\frac{\nD\Phi_0}{\Phi_0}+\frac{p(p\!-\!1)}{2}\frac{\nD L}{L}
-\sum_{s=0}^p(p\!-\!s)\frac{\nD I_s}{I_s}\bigg).$$
In this case, $n\!-\!l\!=\!4$.
Thus, by \e_ref{symmrel_e}, \e_ref{main_e2}, \e_ref{Z_dfn}, and $\si_r\!=\!0$ for $r\!\in\!\Z$ odd,
\BE{Z2sym_e}
\cZ(x,\bar{x},\hb_1,\hb_2,Q)=\frac{\lr\a}{\hb_1\!+\!\hb_2}
\Bigg\{\sum_{\begin{subarray}{c}p_1+p_2+2r=3\\ p_1,p_2\ge0 \end{subarray}}
   -\sum_{\begin{subarray}{c}p_1+p_2+2r=3\\ -l\le p_1,p_2<0 \end{subarray}}
     \Bigg\}\si_{2r}\cZ_{p_1}(x,\hb_1,Q)\cZ_{p_2}(\bar{x},\hb_2,Q).\EE
By the $p\!=\!2$ case of Corollary~\ref{ctCsym_crl} and $\si_1,B^{(1)}_{s,1}=0$, 
\BE{ctCrel_e2} \ctC_{2,0}^{(2)}+\ctC_{3,1}^{(2)}+\si_2=I_0^2I_1\si_2\,.\EE  
Since $L(-x,q)\!=\!-L(x,q)$ by \e_ref{Ldfn_e}, $\Phi_0(-x,q)\!=\!\Phi_0(x,q)$ by \e_ref{Phi0_e}.
By \cite[(4.8)]{Po}, $I_3(q)\!=\!I_1(q)$.
Thus, by \e_ref{Z2sym_e}, \e_ref{Res0_e}, \e_ref{Res1_e}, \e_ref{ctCrel_e2}, and $\ctC^{(0)}_{p,p}\!=\!1$,
\BE{Res2_e}\begin{split}
&\frac{(-1)^l}{\lr\a x^{2l}}
\Rs{\hb_1=0}\bigg\{\Rs{\hb_2=0}\bigg\{
\frac{\ne^{-\frac{\eta(x,q)}{\hb_1}-\frac{\eta(\bar{x},q)}{\hb_2}}
\cZ(x,\bar{x},\hb_1,\hb_2,Q)}{\hb_1\hb_2}\bigg\}\bigg\}
=\Phi_0(x,q)^2\bigg(\frac{\ctC_{3,0}^{(2)}(q)}{I_0(q)^2}-\si_2\frac{\nD I_0(q)}{I_0(q)}\bigg)\\
&\hspace{.5in}+\frac{L(x,q)^2\Phi_0(x,q)^2}{I_0(q)^2I_1(q)^2I_2(q)}
\bigg(2\frac{\nD\Phi_0(x,q)}{\Phi_0(x,q)}+2\frac{\nD L(x,q)}{L(x,q)}
-2\frac{\nD I_0(q)}{I_0(q)}-\frac{\nD I_1(q)}{I_1(q)}-\frac{\nD I_2(q)}{I_2(q)} \bigg)\\
&\hspace{.5in}-\sum_{r=1}^m(r\!-\!2)
\frac{\si_{2r}\Phi_0(x,q)^2}{L(x,q)^{2r-2}}
\bigg(\frac{\nD\Phi_0(x,q)}{\Phi_0(x,q)}-(r\!-\!1)\frac{\nD L(x,q)}{L(x,q)}\bigg).
\end{split}\EE
We note that the identities \e_ref{Res0_e}-\e_ref{Res2_e} above hold in 
$H^*_{\T^m}(\Pn)$, or equivalently with $x\!=\!\al_i$ for each $i\!\in\![n]$
and $\al_i$ as in~\e_ref{sp_weights_e}.

Let
\begin{equation*}\begin{split}
\Psi(x,q)&=L(x,q)\si_{n-1}\big(L(x,q)\big)-n\big(\si_n(L(x,q))\!-\!\si_n(x)\big),\\
\dot\Psi(x,q)&\equiv L(x,q)\frac{\nd \Psi}{\nd L}(x,q)
=2L(x,q)^2\si_{n-2}\big(L(x,q)\big)-(n\!-\!1)L(x,q)\si_{n-1}\big(L(x,q)\big).
\end{split}\end{equation*}
By \cite[(4.8),(4.9)]{Po},
$$I_0(q)^2I_1(q)^2I_2(q)=\big(1-\a^{\a}q\big)^{-1}.$$
Thus, by \e_ref{K2ddfn_e2}, \e_ref{Res2_e}, the $p\!=\!1$ case of~\e_ref{Res0_e}, 
and~\e_ref{Phi0_e}, 
\BE{K2ddfn_e3}\begin{split}
\sum_{d=1}^{\i}Q^d\ti{K}_{2d} &=\frac{1}{I_1(q)}\sum_{i=1}^{2m}
\left\{\frac{L(x,q)^n}{\Psi(x,q)}\left[
\frac12\sum_{r=0}^m\frac{(r\!-\!2)\si_{2r}}{L(x,q)^{2r}}
\bigg(n\!-\!1\!-\!2r-\frac{\dot\Psi(x,q)}{\Psi(x,q)}\bigg)\frac{\nD L(x,q)}{L(x,q)}
\right.\right.\\
&\hspace{.7in}+\a^{\a}q\bigg(n\!-\!1\!-\!\frac{\dot\Psi(x,q)}{\Psi(x,q)}\bigg)\frac{\nD L(x,q)}{L(x,q)}
+\bigg(\a^{\a}q-(1\!-\!\a^{\a}q) \frac{\nD I_1(q)}{I_1(q)}\bigg)\\
&\hspace{.7in}\left.\left.
-\frac1{L(x,q)^2}
\bigg(\frac{\ctC_{3,0}^{(2)}(q)}{I_0(q)^2}- \si_2\frac{\nD I_0(q)}{I_0(q)}\bigg)
\right]\right\}\Bigg|_{x=\al_i}\,.
\end{split}\EE
By \e_ref{Ldfn_e} and \e_ref{Lder_e}, 
$$\frac{L(x,q)^n}{\Psi(x,q)}\sum_{r=0}^m\frac{(r\!-\!2)\si_{2r}}{L(x,q)^{2r}}
=-\frac12-2\frac{\nD L(x,q)}{L(x,q)}+\frac{n\!-\!4}{2}\cdot\frac{\si_n(x)}{\Psi(x,q)}\,.$$
This identity and \e_ref{Lder_e} give
$$\frac{L(x,q)^n}{\Psi(x,q)}\Bigg[
\frac12\sum_{r=0}^m\frac{(r\!-\!2)\si_{2r}}{L(x,q)^{2r}}
\bigg(n\!-\!1\!-\!2r-\frac{\dot\Psi(x,q)}{\Psi(x,q)}\bigg)
+\a^{\a}q\bigg(n\!-\!1\!-\!\frac{\dot\Psi(x,q)}{\Psi(x,q)}\bigg)\Bigg]
=-\frac34$$
for $x\!=\!\al_i$ with $i\!\in\![n]$.
Combining this with \e_ref{K2ddfn_e3} and using \e_ref{Lder_e} again, we obtain
\BE{K2ddfn_e4}\begin{split}
\sum_{d=1}^{\i}Q^d\ti{K}_{2d} &=\frac{1}{I_1(q)}\Bigg( 
\bigg(\frac14-\frac{(1\!-\!\a^{\a}q)}{\a^{\a}q}\frac{\nD I_1(q)}{I_1(q)}\bigg) 
\nD\sum_{i=1}^{2m}\ln L(\al_i,q)\\
&\hspace{1in}+\frac{1}{2\a^{\a}q}\bigg(\frac{\ctC_{3,0}^{(2)}(q)}{I_0(q)^2}- 
\si_2\frac{\nD I_0(q)}{I_0(q)}\bigg)
\nD\sum_{i=1}^{2m}\frac1{L(\al_i,q)^2}\Bigg).
\end{split}\EE
Since $\{L(\al_i,q)\}$ with $i\!\in\![2m]$ are the roots of the polynomial equation
$\si_{2m}(z)=\a^{\a}qz^{2m}$,
\begin{equation*}\begin{split}
\nD \ln\prod_{i=1}^{2m}L(\al_i,q)
&=\nD\ln\big((1\!-\!\a^{\a}q)^{-1}\si_{2m}\big)=\frac{\a^{\a}q}{1\!-\!\a^{\a}q}\,,\\
\nD\sum_{i=1}^{2m}\frac1{L(\al_i,q)^2}&=
\nD\bigg(-2\frac{\si_{2m-2}}{\si_{2m}}\bigg)=0~~\forall\,m\ge2.
\end{split}\end{equation*}
Plugging these two identities into \e_ref{K2ddfn_e4} and using 
$Q\frac{\nd}{\nd Q}=\frac{1}{I_1(q)}\nD$, we obtain \e_ref{klein_e}.

\appendix

\section{Proof of \e_ref{ntcsym_e}}
\label{ntcsym_app}

We assume that $\nu_{\a}\!>\!0$ and denote $\a^{\a}$ by $A$ throughout this section.
Let ${\mathfrak R}$ be as defined in~\e_ref{Rsdfn_e} and~\e_ref{Rssumdfn_e}.
For each $p\!\in\!\Z$, let $\hat{p}=n\!-\!1\!-\!l\!-\!p$.

The $d\!=\!0$ case of \e_ref{ntcsym_e} follows immediately from the first equation
in~\e_ref{cs1_e}.
Lemma~\ref{ncsym_lmm} is the key observation in the proof 
of the remaining cases of~\e_ref{ntcsym_e}.
Since $\ntc_{p,s}^{(1)}=-\nc_{p,s}^{(1)}$ whenever $s\!\le\!p\!-\!\nu_{\a}$,
the $(d,s)\!=\!(1,p\!-\!\nu_{\a})$ case 
of this lemma is precisely the $d\!=\!1$ case of~\e_ref{ntcsym_e}.

\begin{lmm}\label{ncsym_lmm}
If $d\!\in\!\Z^+$ and $p,s\!\in\!\Z^{\ge0}$ are such that 
$p,s\!\le\!n\!-\!1\!-\!l$ and $p\!-\!s\!\le\!\nu_{\a}d$, then
$$\nc_{p,s}^{(d)}=\de_{p-s,\nu_{\a}d}A^d
-(-1)^{\nu_{\a}d+p+s}\nc_{\hat{s},\hat{p}}^{(d)}
-\sum_{\begin{subarray}{c}d_1+d_2=d\\ d_1,d_2\ge1\end{subarray}} 
\sum_{t=0}^{n-1-l}\!\!
(-1)^{\nu_{\a}d_2+s+t}\nc_{p,t}^{(d_1)}\nc_{\hat{s},\hat{t}}^{(d_2)}\,.$$
\end{lmm}

\begin{proof}
If $d,p,s$ are as above,
\BE{ncrsres_e}\begin{split}
\nc_{p,s}^{(d)}&=\Rs{w=0}\left\{
\frac{(w\!+\!d)^p\!\!\prod\limits_{k=1}^l\prod\limits_{r=1}^{a_kd}(a_kw\!+\!r)}
{w^{s+1}\prod\limits_{r=1}^d(w\!+\!r)^n}\right\}
=\lr\a\Rs{w=0}\left\{
\frac{\prod\limits_{k=1}^l\prod\limits_{r=1}^{a_kd-1}(a_kw\!+\!r)}
{w^{s+1}(w\!+\!d)^{n-l-p}\prod\limits_{r=1}^{d-1}(w\!+\!r)^n}\right\}\\
&=(-1)^{\nu_{\a} d+p+s}\lr\a\Rs{w=-d}\left\{
\frac{\prod\limits_{k=1}^l\prod\limits_{r=1}^{a_kd-1}(a_kw\!+\!r)}
{(w\!+\!d)^{s+1}w^{n-l-p}\prod\limits_{r=1}^{d-1}(w\!+\!r)^n}\right\};
\end{split}\EE
the last equality is obtained by substituting $-w-d$ for $w$.
If in addition $d_2\!\in\!\Z^+$ is such that $d_2\!<\!d$,
\begin{equation*}\begin{split}
&\Rs{w=-d_2}\left\{
\frac{(w\!+\!d)^p\!\!\prod\limits_{k=1}^l\prod\limits_{r=1}^{a_kd}(a_kw\!+\!r)}
{w^{s+1}\prod\limits_{r=1}^d(w\!+\!r)^n}  \right\}\\
&\quad =\lr\a\!\!\!\sum_{t=0}^{n-1-l}
\Rs{w=-d_2}\!\!\left\{\!\!\frac{\prod\limits_{k=1}^l\prod\limits_{r=1}^{a_kd_2-1}\!\!\!(a_kw\!+\!r)}
{w^{s+1}(w\!+\!d_2)^{\hat{t}+1}\prod\limits_{r=1}^{d_2-1}\!\!(w\!+\!r)^n}\!\! \right\}
\Rs{w=-d_2}\!\!\left\{\!\!\frac{(w\!+\!d)^p
\prod\limits_{k=1}^l\prod\limits_{r=a_kd_2+1}^{a_kd}(a_kw\!+\!r)}
{(w\!+\!d_2)^{t+1}\prod\limits_{r=d_2+1}^d\!\!\!(w\!+\!r)^n}\!\! \right\}\\
&\quad =\sum_{t=0}^{n-1-l}(-1)^{\nu_{\a} d_2+s+t} \nc_{\hat{s},\hat{t}}^{(d_2)}\nc_{p,t}^{(d-d_2)}\,;
\end{split}\end{equation*}
the last equality follows from \e_ref{ncrsres_e} with $(d,p,s)$ replaced
by $(d_2,\hat{s},\hat{t})$.
Thus, by the Residue Theorem~on~$S^2$, 
\begin{equation*}\begin{split}
\nc_{p,s}^{(d)}&=\Rs{w=0}\left\{
\frac{(w\!+\!d)^p\!\!\prod\limits_{k=1}^l\prod\limits_{r=1}^{a_kd}(a_kw\!+\!r)}
{w^{s+1}\prod\limits_{r=1}^d(w\!+\!r)^n}  \right\}
=-\Rs{w=\i,-1,-2,\ldots,-d}\left\{
\frac{(w\!+\!d)^p\!\!\prod\limits_{k=1}^l\prod\limits_{r=1}^{a_kd}(a_kw\!+\!r)}
{w^{s+1}\prod\limits_{r=1}^d(w\!+\!r)^n}  \right\}\\
&=\de_{p-s,\nu_{\a}d}A^d-(-1)^{\nu_{\a} d+p+s}\nc_{\hat{s},\hat{p}}^{(d)}
-\sum_{\begin{subarray}{c}d_1+d_2=d\\ d_1,d_2\ge1\end{subarray}} 
\sum_{t=0}^{n-1-l}\!\!
(-1)^{\nu_{\a}d_2+s+t}\nc_{p,t}^{(d_1)}\nc_{\hat{s},\hat{t}}^{(d_2)}\,,
\end{split}\end{equation*}
as claimed.
\end{proof}

For any $k$-tuple $\bfd\!\equiv\!(d_i)_{i\in[k]}\!\in\!(\Z^+)^k$, let
$$\ell(\bfd)\equiv k, \qquad |\bfd)\equiv d_k\,.$$
If in addition $p\!\in\!\Z^{\ge0}$ and  
$\bfr\!\equiv\!(r_i)_{i\in[k]}\!\in\!(\Z^{\ge0})^k$
is another $k$-tuple, let
$$\nc_{p\bfr}^{(\bfd)}\equiv (-1)^k\cdot
\nc^{(d_1)}_{p,r_1}\cdot\prod\limits_{i=2}^{k}\!\nc^{(d_i)}_{r_{i-1},r_i}\,.
$$
If $d\!\in\!\Z^+$ and $p\!\in\!\Z^{\ge0}$, let
\begin{equation*}\begin{split}
\cS_p(d)&\equiv
\big\{(\bfd,\bfr)\!\in\!(\Z^+)^k\!\times\!(\Z^{\ge0})^{k-1}\!:\,k\!\in\!\Z^+,~
\sum_{i=1}^kd_i\!=\!d,~r_i\le p\!-\!\nu_{\a}\sum_{j=1}^id_j~~
\forall\,i\!\in\![k\!-\!1]\big\},\\
\cS_p^*(d)&\equiv\big\{(\bfd,\bfr)\!\in\!\cS_p(d)\!:\,
r_i< p\!-\!\nu_{\a}\sum_{j=1}^id_j~~
\forall\,i\!\in\![k\!-\!1]\big\}.
\end{split}\end{equation*}
By \e_ref{littletic_e}, if in addition $s\!\le\! p-\nu_{\a}d$, then 
\BE{ntcexp_e}
\ntc_{p,s}^{(d)}=\sum_{(\bfd,\bfr)\in \cS_p(d)}\!\!\!\!\!\!\nc_{p\bfr s}^{(\bfd)}\,,
\qquad\hbox{where}\quad 
\nc_{p\bfr s}^{(\bfd)}\equiv \nc_{ps}^{(d)}\equiv-\nc_{p,s}^{(d)}~~~\hbox{if}~~\ell(\bfd)\!=\!1.
\EE

\begin{crl}\label{ncflip_crl}
If $d\!\in\!\Z^+$ and $p,s\!\in\!\Z^{\ge0}$ are such that 
$p,s\!\le\!n\!-\!1\!-\!l$ and $p\!-\!s\!\le\!\nu_{\a}d$, then
$$\nc_{p,s}^{(d)}=\de_{p-s,\nu_{\a}d}A^d
+(-1)^{\nu_{\a}d+p+s}\hspace{-.15in}
\sum_{(\bfd,\bfr )\in \cS_{\hat{s}}^*(d)}\!\!\!
\nc_{\hat{s}\bfr\hat{p}}^{(\bfd)}
+\sum_{\begin{subarray}{c}d_1+d_2=d\\ d_1,d_2\ge1\end{subarray}}
\sum_{(\bfd,\bfr )\in \cS_{\hat{s}}^*(d_2)}\!\!\!\!
\sum_{t=0}^{s+\nu_{\a}d_2}\!\!\! (-1)^{\nu_{\a}d_2+s+t}
\nc_{p,t}^{(d_1)}\nc_{\hat{s}\bfr\hat{t}}^{(\bfd)}\,.$$
\end{crl}

\begin{proof}
By Lemma~\ref{ncsym_lmm} and induction on $K$, 
\begin{equation*}\begin{split}
\nc_{p,s}^{(d)}&=\de_{p-s,\nu_{\a}d}A^d
+(-1)^{\nu_{\a}d+p+s}\hspace{-.15in}
\sum_{\begin{subarray}{c}(\bfd,\bfr )\in \cS_{\hat{s}}^*(d)\\ \ell(\bfd)\le K\end{subarray}}\!\!\!
\nc_{\hat{s}\bfr\hat{p}}^{(\bfd)}
+\sum_{\begin{subarray}{c}d_1+d_2=d\\ d_1,d_2\ge1\end{subarray}}
\sum_{\begin{subarray}{c}(\bfd,\bfr )\in \cS_{\hat{s}}^*(d_2)\\ \ell(\bfd)\le K\end{subarray}}\!\!\!\!
\sum_{t=0}^{s+\nu_{\a}d_2}\!\!\! (-1)^{\nu_{\a}d_2+s+t}
\nc_{p,t}^{(d_1)}\nc_{\hat{s}\bfr\hat{t}}^{(\bfd)}\\
&\qquad+\sum_{\begin{subarray}{c}d_1+d_2=d\\ d_1,d_2\ge1\end{subarray}}
\sum_{\begin{subarray}{c}(\bfd,\bfr )\in \cS_{\hat{s}}^*(d_2)\\ \ell(\bfd)=K\end{subarray}}
\sum_{t=s+\nu_{\a}d_2+1}^{n-1-l}\!\!\!\!\!\!\! (-1)^{\nu_{\a}d_2+s+t}
\nc_{p,t}^{(d_1)}\nc_{\hat{s}\bfr\hat{t}}^{(\bfd)}\,
\end{split}\end{equation*}
for all $K\!\ge\!1$.\footnote{The $K\!=\!1$
case of this identity is Lemma~\ref{ncsym_lmm}.
The inductive step is carried out by applying Lemma~\ref{ncsym_lmm} to 
the factor $\nc_{p,t}^{(d_1)}$ on the last line of this identity and noting
that the assumptions imply that $p\!-\!t\!<\!\nu_{\a}d_1$.}
Setting $K\!=\!d$, we obtain the claim.
\end{proof}

\begin{crl}\label{ntcflip_crl}
If $d,p\!\in\!\Z$ are such that $d\!\ge\!2$ and
$\nu_{\a}d\!\le\!p\le\!n\!-\!1\!-\!l$, then
$$\sum_{(\bfd,\bfr)\in \cS_p(d)}\!\!\!\!\!\!\nc_{p\bfr(p-\nu_{\a}d)}^{(\bfd)}
=-\!\!\!\sum_{(\bfd,\bfr)\in \cS_{\hat{p}+\nu_{\a}d}^*(d)}\!\!\!\!\!\!\!\!\!
\nc_{(\hat{p}+\nu_{\a}d)\bfr\hat{p}}^{(\bfd)}
+A\sum_{(\bfd,\bfr)\in \cS_{\hat{p}+\nu_{\a}(d-1)}^*(d-1)}\!\!\!\!\!\!\!\!\!
\nc_{(\hat{p}+\nu_{\a}(d-1))\bfr\hat{p}}^{(\bfd)}\,.$$
\end{crl}

\begin{proof}
By Corollary~\ref{ncflip_crl} and induction on $d_2^*$,
\begin{equation*}\begin{split}
&\sum_{\begin{subarray}{c}(\bfd,\bfr)\in \cS_p(d)\\ |\bfd)\le d_2^*\end{subarray}}
\!\!\!\!\!\!\nc_{p\bfr(p-\nu_{\a}d)}^{(\bfd)}
=-A\sum_{\begin{subarray}{c}(\bfd,\bfr)\in \cS_p(d-1)\\ |\bfd)\ge d_2^*\end{subarray}}
\!\!\!\!\!\!\!\!\!
\nc_{p\bfr(p-\nu_{\a}(d-1))}^{(\bfd)}\\
&\qquad\quad
-\sum_{\begin{subarray}{c}d_1+d_2=d\\ 1\le d_2\le d_2^*\end{subarray}}
\sum_{\begin{subarray}{c}(\bfd',\bfr')\in \cS_p(d_1)\\ 
|\bfd')>d_2^*-d_2\end{subarray}}
\sum_{(\bfd'',\bfr'')\in \cS_{\hat{p}+\nu_{\a}d}^*(d_2)}
\!\!\!\sum_{t=0}^{p-\nu_{\a}d_1}\!\!\!(-1)^{p-\nu_{\a}d_1+t}
 \nc_{p\bfr't}^{(\bfd')}\nc_{(\hat{p}+\nu_{\a}d)\bfr''\hat{t}}^{(\bfd'')}\\
&\qquad\quad
+A\sum_{\begin{subarray}{c}d_1+d_2=d\\ 2\le d_2\le d_2^*\end{subarray}}
\sum_{\begin{subarray}{c}(\bfd',\bfr')\in \cS_p(d_1)\\ 
|\bfd')>d_2^*-d_2\end{subarray}}
\sum_{(\bfd'',\bfr'')\in \cS_{\hat{p}+\nu_{\a}(d-1)}^*(d_2-1)}
\!\!\!\sum_{t=0}^{p-\nu_{\a}d_1}\!\!\!(-1)^{p-\nu_{\a}d_1+t}
 \nc_{p\bfr't}^{(\bfd')}\nc_{(\hat{p}+\nu_{\a}(d-1))\bfr''\hat{t}}^{(\bfd'')}
\end{split}\end{equation*}
for all $d_2^*\!=\!1,\ldots,d\!-\!1$.\footnote{The $d_2^*\!=\!1$
case of this identity is obtained from Corollary~\ref{ncflip_crl}
with $(d,p,s)$ replaced by $(1,r_{\ell(\bfd)-1},p\!-\!\nu_{\a}d)$.
The inductive step is carried out by using Corollary~\ref{ncflip_crl}
with $(d,p,s)$ replaced by $(d_2^*\!+\!1,r_{\ell(\bfd)-1},p\!-\!\nu_{\a}d)$
and $(d_2^*,r_{\ell(\bfd)-1},p\!-\!\nu_{\a}(d-1))$.}
Combining the $d_2^*\!=\!d\!-\!1$ case of this identity and 
Corollary~\ref{ncflip_crl} with $(d,s)$ replaced by 
$(d,p\!-\!\nu_{\a}d)$ and $(d\!-\!1,p\!-\!\nu_{\a}(d\!-\!1))$, we obtain the claim.
\end{proof}

We now verify \e_ref{ntcsym_e} for $d\!\ge\!2$.
By \e_ref{ntcexp_e} and Corollary~\ref{ntcflip_crl},
\begin{equation*}\begin{split}
&\ntc^{(d)}_{p,p-\nu_{\a}d}=
-\Bigg(\ntc^{(d)}_{\hat{p}+\nu_{\a}d,\hat{p}} 
-\sum_{\begin{subarray}{c} d_1+d_2=d\\ d_1,d_2\ge1\end{subarray}} 
\sum_{\begin{subarray}{c}(\bfd',\bfr')\in \cS_{\hat{p}+\nu_{\a}d}^*(d_1)\\
(\bfd'',\bfr'')\in \cS_{\hat{p}+\nu_{\a}d_2}(d_2)\end{subarray}}
\hspace{-.45in}
\nc_{(\hat{p}+\nu_{\a}d)\bfr'(\hat{p}+\nu_{\a}d_2)}^{(\bfd')}
\nc_{(\hat{p}+\nu_{\a}d_2)\bfr''\hat{p}}^{(\bfd'')} \Bigg)\\
&\hspace{.7in} +A\Bigg(\ntc^{(d-1)}_{\hat{p}+\nu_{\a}(d-1),\hat{p}} 
-\sum_{\begin{subarray}{c} d_1+d_2=d\\ 2\le d_1\le d-1\end{subarray}} 
\sum_{\begin{subarray}{c}(\bfd',\bfr')\in \cS_{\hat{p}+\nu_{\a}(d-1)}^*(d_1-1)\\
(\bfd'',\bfr'')\in \cS_{\hat{p}+\nu_{\a}d_2}(d_2)\end{subarray}}
\hspace{-.45in}
\nc_{(\hat{p}+\nu_{\a}(d-1))\bfr'(\hat{p}+\nu_{\a}d_2)}^{(\bfd')}
\nc_{(\hat{p}+\nu_{\a}d_2)\bfr''\hat{p}}^{(\bfd'')} \Bigg)\\
&\quad 
=-\ntc^{(d)}_{\hat{p}+\nu_{\a}d,\hat{p}}+ 
\ntc^{(1)}_{\hat{p}+\nu_{\a}d,\hat{p}+\nu_{\a}(d-1)}\ntc^{(d-1)}_{\hat{p}+\nu_{\a}(d-1),\hat{p}} \
+A\ntc^{(d-1)}_{\hat{p}+\nu_{\a}(d-1),\hat{p}} 
-\sum_{\begin{subarray}{c} d_1+d_2=d\\ 2\le d_1\le d-1\end{subarray}}\hspace{-.15in} 
\ntc^{(d_1)}_{p-\nu_{\a}d_2,p-\nu_{\a}d} \ntc^{(d_2)}_{\hat{p}+\nu_{\a}d_2,\hat{p}}\,.
\end{split}\end{equation*}
Combining this with the first equation in~\e_ref{cs1_e}, we obtain
\begin{equation*}\begin{split}
\sum_{\begin{subarray}{c}d_1+d_2=d\\ d_1,d_2\ge0\end{subarray}}\hspace{-.1in} 
\ntc^{(d_1)}_{p-\nu_{\a}d_2,p-\nu_{\a}d}\ntc^{(d_2)}_{\hat{p}+\nu_{\a}d_2,\hat{p}}
=\big(\ntc^{(1)}_{p-\nu_{\a}(d-1),p-\nu_{\a}d} 
+\ntc^{(1)}_{\hat{p}+\nu_{\a}d,\hat{p}+\nu_{\a}(d-1)} +A\big)
\ntc^{(d-1)}_{\hat{p}+\nu_{\a}(d-1),\hat{p}}=0;
\end{split}\end{equation*}
the last equality is the $d\!=\!1$ case of \e_ref{ntcsym_e} with $p$ replaced by 
$p-\nu_{\a}(d\!-\!1)$.\\

\vspace{3mm}

\noindent
{\it Department of Mathematics, Rutgers University, Piscataway, NJ 08854-8019 \\
alexandra@math.rutgers.edu}\\

\noindent
{\it Department of Mathematics, SUNY Stony Brook, NY 11794-3651\\
azinger@math.sunysb.edu}

\end{document}